\documentclass{amsart}
\usepackage[margin=1.3in]{geometry}
\usepackage{color}
\usepackage{float}
\usepackage{amsmath}
\usepackage{amssymb}
\usepackage{amsthm}
\usepackage{mathrsfs}
\usepackage{enumerate}
\usepackage{appendix}
\usepackage{tikz}
\usepackage{bm,amsbsy}

\numberwithin{equation}{section}

\newtheorem{thm}{Theorem}[section]
\newtheorem*{thm*}{Theorem}
\newtheorem{prop}[thm]{Proposition}
\newtheorem{lemma}[thm]{Lemma}

\newtheorem{cor}[thm]{Corollary}

\theoremstyle{definition}

\theoremstyle{remark}
\newtheorem{remark}[thm]{Remark}

\definecolor{pink}{rgb}{1,.2,.6}
\definecolor{orange}{rgb}{0.7,0.3,0}
\definecolor{blue}{rgb}{.2,.6,.75}
\definecolor{green}{rgb}{.4,.7,.4}
\definecolor{purple}{RGB}{127,0,255}

\newcommand{\bF}{\mathbb{F}}

\newcommand{\bR}{\mathbb{R}}

\newcommand{\bZ}{\mathbb{Z}}

\def\t{\tau}

\newcommand{\Gal}{\mathrm{Gal}}

\newcommand{\GL}{\mathrm{GL}}

\newcommand{\rank}{\mathrm{rank\ }}

\newcommand{\modd}[1]{\; ( \text{mod} \; #1)}

\newcommand{\vol}{{\rm vol}}

\newcommand{\al}{\alpha}

\newcommand{\lam}{\lambda}

\newcommand{\Lcal}{\mathcal{L}}

\newcommand{\Rcal}{\mathcal{R}}

\newcommand{\Lbf}{\mathbf{L}}
\newcommand{\Hbf}{\mathbf{H}}

\newcommand{\Nbf}{\mathbf{N}}

\newcommand{\dbf}{\mathbf{d}}
\newcommand{\tbf}{\mathbf{t}}

\newcommand{\Kbf}{\mathbf{K}}

\newcommand{\Xbf}{\mathbf{X}}

\newcommand{\abf}{{\bf a}}
\newcommand{\bbf}{{\bf b}}

\newcommand{\kbf}{\mathbf{k}}

\newcommand{\nbf}{{\bf n}}

\newcommand{\ubf}{{\bf u}}
\newcommand{\vbf}{{\bf v}}

\newcommand{\x}{{\bf x}}
\newcommand{\xbf}{{\bf x}}
\newcommand{\ybf}{{\bf y}}
\newcommand{\zbf}{{\bf z}}

\newcommand{\Ltil}{\Tilde{L}}

\newcommand{\Qtil}{\Tilde{Q}}
\newcommand{\Util}{\Tilde{U}}

\newcommand{\beq}{\begin{equation}}
\newcommand{\eeq}{\end{equation}}

\newcommand{\ba}{\begin{align}}
\newcommand{\ea}{\end{align}}

\makeatletter
\def\@tocline#1#2#3#4#5#6#7{\relax
  \ifnum #1>\c@tocdepth 
  \else
    \par \addpenalty\@secpenalty\addvspace{#2}%
    \begingroup \hyphenpenalty\@M
    \@ifempty{#4}{%
      \@tempdima\csname r@tocindent\number#1\endcsname\relax
    }{%
      \@tempdima#4\relax
    }%
    \parindent\z@ \leftskip#3\relax \advance\leftskip\@tempdima\relax
    \rightskip\@pnumwidth plus4em \parfillskip-\@pnumwidth
    #5\leavevmode\hskip-\@tempdima
      \ifcase #1
       \or\or \hskip 1em \or \hskip 2em \else \hskip 3em \fi%
      #6\nobreak\relax
    \hfill\hbox to\@pnumwidth{\@tocpagenum{#7}}\par
    \nobreak
    \endgroup
  \fi}
\makeatother

\newcommand{\genlegendre}[4]{%
  \genfrac{(}{)}{}{#1}{#3}{#4}%
  \if\relax\detokenize{#2}\relax\else_{\!#2}\fi
}

\makeatletter
\let\@@pmod\pmod
\DeclareRobustCommand{\pmod}{\@ifstar\@pmods\@@pmod}
\def\@pmods#1{\mkern4mu({\operator@font mod}\mkern 6mu#1)}
\makeatother

\usepackage{enumitem}

\begin{document}

\title{Estimates for short character sums evaluated at homogeneous polynomials}

\author[Rena Chu]{Rena Chu}
\address{Georg-August-Universität Göttingen, Bunsenstraße 3-5, 37073 Göttingen, Germany}
\email{rena.chu@duke.edu}

\keywords{short Dirichlet character sums, Burgess bounds, multiplicative energy}

\makeatletter
\@namedef{subjclassname@2020}{%
  \textup{2020} Mathematics Subject Classification}
\makeatother
\subjclass[2020]{11L40, 11T24, 11H06}

\begin{abstract}
    Let $p$ be a prime. We prove bounds on short Dirichlet character sums evaluated at a class of homogeneous polynomials in arbitrary dimensions. In every dimension, this bound is nontrivial for sums over boxes with side lengths as short as $p^{1/4 + \kappa}$ for any $\kappa>0$. Our methods capitalize on the relationship between characters mod $p$ and characters over finite field extensions as well as bounds on the multiplicative energy of sets in products of finite fields.
\end{abstract}

\maketitle

\section{Introduction}
Let $p$ be a prime and $\chi$ be a nonprincipal Dirichlet character modulo $p$. Let $N,H$ be integers such that $0\leq N, H<p$. The P\'olya-Vinogradov bound states that
    \begin{align*}
         \sum_{x\in (N,N+H]} \chi(x) \ll p^{1/2}\log p.
    \end{align*}
This bound is nontrivial as long as $p^{1/2} \log p = o(H)$, motivating the study of short character sums where $H\ll p^{1/2}$. A classical result in this area is the influential work of Burgess, who proved in \cite{Bur63} that for every integer $r\geq 1$,
    \begin{align}\label{eqn:burgess}
        \sum_{x\in (N,N+H]} \chi(x) \ll_r H^{1-\frac{1}{r}}p^{\frac{r+1}{4r^2}}\log p.
    \end{align}
This bound is nontrivial for $H> p^{1/4+\kappa}$ for any $\kappa>0$. A motivation for studying these sums was to give a subconvexity bound on Dirichlet $L$-functions of characters with composite moduli $q$; Burgess' extension of \eqref{eqn:burgess} for these characters led to a bound of $L(1/2+it,\chi) \ll_{t,\varepsilon} q^{3/16 + \varepsilon}$, which remained a record
for nearly six decades until the breakthrough of Petrow and Young, who showed a Weyl-strength bound of $(q(|t|+1))^{1/6+\varepsilon}$ for all Dirichlet $L$-functions \cite{PY20,PY23}. The bound \eqref{eqn:burgess} remains essentially the best bound for short multiplicative character sums.

We are interested in multiplicative character sums with polynomial arguments. A central motivation here is quite different; these play a role, for example, in counting solutions to Diophantine equations, such as in the works of \cite{Pie06}, \cite{HB08}, \cite{Mun09}, \cite{HBP12}, \cite{Bon21}, and \cite{BP24}. Fix integers $n,k\geq 1$, and let $F\in \bZ[X_1,...,X_n]$ be a homogeneous polynomial of degree $k$. Define
    \begin{align*}
        S(F;\Nbf,\Hbf) = \sum_{\substack{x_1,...,x_n\\x_i\in (N_i, N_i+H_i]}} \chi(F(x_1,...,x_n)),
    \end{align*}
    where $\Nbf=(N_1,...,N_n)$ and $\Hbf = (H_1,...,H_n)$.
In general, if $H_i\gg p^{1/2+\varepsilon}$ for all $1\leq i \leq n$, the standard method of completing the sum and applying results of Weil strength on complete, multi-dimensional, mixed character sums, provided by e.g. \cite{Kat07}, is a suitable strategy to produce a nontrivial upper bound.
In this paper, we focus on the complementary regime of short sums in which $H_i\ll p^{1/2}$, and we prove nontrivial bounds on such sums evaluated at a specific class of forms in arbitrarily many variables and of arbitrary degree.
Our main result is the following.

\begin{thm}\label{thm:main}
Fix a prime $p$ and an integer $n\geq 1$, and let $F\in \bF_p[X_1,...,X_n]$ be a form in $n$ variables and of degree $n$.
    Suppose $F$ splits into linear forms over the algebraic closure $\overline{\bF}_p$,
    where the linear forms are linearly independent over $\overline{\bF}_p$, i.e. $F=(\underline{\beta_1^T}\Xbf)\cdots (\underline{\beta_n^T}\Xbf)$ with $\underline{\beta_i}\in (\overline{\bF}_p)^n$, and $\underline{\beta_1},...,\underline{\beta_n}$ are linearly independent over $\overline{\bF}_p$. 
    Let $\Nbf=(N_1,...,N_n)$ and $\Hbf=(H_1,...,H_n)$ with $H_{\mathrm{min}}\leq p^{1/2}$, where $H_{\mathrm{min}}=\min_i H_i$. Let $\|\Hbf\|$ denote $H_1\cdots H_n$. Then for any integer $r> n$ and $\varepsilon>0$,
        \begin{align}\label{eqn:mainbound}
            S(F;\Nbf,\Hbf)\ll_{n,r,\varepsilon} 
             \|\Hbf\|H_{\mathrm{min}}^{-\frac{n}{r}} p^{\frac{n(r+n)}{4r^2}+\varepsilon}.
        \end{align}
\end{thm}

We have the following consequence.
\begin{cor} \label{cor:main} Assume the hypotheses of Theorem \ref{thm:main}.
    Let $\kappa>0$ and $H_i\geq p^{1/4+\kappa}$ for all $i$, and suppose $H_{\mathrm{min}}\leq p^{1/2}$. Then there exists a positive $\delta=\delta(n,\kappa)$ such that $S(F;\Nbf,\Hbf)\ll_{n,\kappa} \|\Hbf\| p^{-\delta}$.
\end{cor}

We compute the amount of savings $\delta$ explicitly in \S \ref{sec:delta}. Note that if $p\leq n$, then the trivial bound agrees with \eqref{eqn:mainbound}, since $|S(F;\Nbf,\Hbf)| \ll_{n} p^n\ll_{n} 1$. Hence going forward, we assume $p> n$.

Prior to our work, results for short multiplicative character sums evaluated at forms were known in two extreme regimes. On the one hand, Pierce and Xu \cite{PX20} showed the first nontrivial upper bounds on $S(F;\Nbf,\Hbf)$ for all $F$ belonging to a generic class of forms in arbitrary dimensions and of arbitrary degree. Their result is nontrivial for boxes with side lengths roughly of size $H_i> p^{\frac{1}{2}-\frac{1}{2(n+1)}+\kappa}$ for any $\kappa>0$. This is a threshold that has appeared in many efforts to generalize the Burgess bound to $n$ dimensions (see \S \ref{sec:literature} and also \cite[\S 1.3]{PX20} for details). 
On the other hand, results for lengths as short as the original Burgess threshold of $H> p^{1/4+\kappa}$ 
were known for forms of two very special shapes. First, Chang \cite{Cha09} treated sums of multiplicative characters mod $p$ evaluated at binary quadratic forms, and Heath-Brown \cite{HB16} generalized this result to characters with composite, square-free moduli. Second, Bourgain and Chang \cite{BC10} treated forms that split over $\bF_p$ (more precisely, products of $n$ linear forms over $\bF_p$ in $n$ variables, linearly independent over $\bF_p$).  
Thus, the main achievement of Theorem \ref{thm:main} is attaining the Burgess threshold of $p^{1/4+\kappa}$ for a greater class of homogeneous polynomials of arbitrary degree and in arbitrarily many variables.

For every form to which Theorem \ref{thm:main} applies, we improve the result of \cite{PX20} by widening the range of $H_i$ for which we achieve nontrivial bounds. Furthermore, Theorem 1.1 enlarges the class of polynomials considered in \cite{BC10} by removing the restriction that forms must split over $\bF_p$, while strengthening their result (which required $H_i=H)$ to apply to boxes with independent ranges $H_1,...,H_n$ (by reducing this to a case where $H_i\approx H$). We additionally recover the result of \cite{Cha09} when $n=2$, for both the irreducible and reducible cases, and the result of \cite{Bur63} (up to a factor of $p^\varepsilon$) for prime moduli when $n=1$.
(We note that for the rest of the paper, the words ``character sum'' implicitly means \textit{multiplicative} character sum.)

\subsection{Previous literature}\label{sec:literature} Shortly after the seminal work of Burgess on the classic short character sum in a series of papers \cite{Bur57, Bur62a,Bur62b,Bur63}, Burgess extended his method to estimate character sums evaluated at certain forms. 
In a paper \cite{Bur67_quadratic} on finding the least $x$ for which a polynomial $f(x)$ is a quadratic non-residue mod $p$, Burgess studied the quadratic character of a binary quadratic form and was motivated to bound character sums for a general Dirichlet character evaluated at binary quadratic forms, on which he wrote concurrently the paper \cite{Bur68}. There he noted that the case where $F$ is an \textit{irreducible} binary quadratic form followed the work of Davenport and Lewis \cite{DL63} on character sums over finite fields, upon noticing that when $F$ is irreducible, $\chi\circ F$ is a multiplicative character of $\bF_{p^2}$. The result of \cite{DL63} then implies in this special case of $n=2$ a nontrivial bound for $H>p^{1/3+\kappa}$. In the same paper \cite{Bur68}, Burgess extended this result to \textit{reducible} binary quadratic forms (that are not perfect squares) by following the method developed in his earlier work. He further stated without proof that this would generalize to forms in $n$ variables that are products of $n$ linear forms, linearly independent over $\bF_p$, in which case one could achieve nontrivial bounds for $H>p^{\beta_n+\kappa}$, where
       \begin{align*}
        \beta_n = \frac{n}{2(n+1)} = \frac{1}{2} - \frac{1}{2(n+1)}.
    \end{align*}
Note that $\beta_1=1/4$, $\beta_2=1/3$, and $\beta_n$ approaches $1/2$ as $n$ tends to infinity. This work initiated the study of character sums evaluated at homogeneous polynomials (also known as forms).

\subsubsection{Results on character sums over finite fields}\label{sec:history_ff} Before we survey progress on character sums of polynomials following \cite{Bur68}, we provide an overview of character sums over finite fields, as many methods and results on bounding the former build on the latter.
Fix $n\geq 1$, and consider the finite field $\bF_{p^n}$. Let $\omega_1,...,\omega_n$ be a basis of $\bF_{p^n}$, and let $\psi$ be a multiplicative character of $\bF_{p^n}$. Consider the character sum
    \begin{align}\label{eqn:sum_ff}
        \sum_{\substack{x_1,...,x_n\\x_i\in (N_i,N_i+H_i]}} \psi(x_1\omega_1+\cdots +x_n\omega_n).
    \end{align}
As briefly mentioned above, Davenport and Lewis \cite{DL63} extended the methods of \cite{Bur62a} to character sums over finite fields, proving a nontrivial bound on \eqref{eqn:sum_ff}, uniformly for all bases, for $\Hbf=(H,...,H)$ with $H>p^{\beta_n+\kappa}$. 
Burgess \cite{Bur67} improved the case $n=2$, providing nontrivial bounds for $H>p^{1/4+\kappa}$, at the cost of restricting to certain choices of bases, e.g. a power basis $1, \omega$, where $\omega$ is the root of $X^2-a$ and $a<0$ is not a square mod $p$. In the years following, Karacuba \cite{Kar68,Kar70} achieved the threshold of $p^{1/4+\kappa}$ for arbitrary dimensions with $\Hbf=(H,...,H)$, for bases that are in the form of a power basis $1,\omega,...,\omega^{n-1}$, where $\omega$ is the root of a polynomial of degree $n$ that is irreducible over $\bF_p$. Karacuba developed a new technique to shift the box of summation, different to the redistribution procedure in the work of Burgess (see e.g. \cite[\S 8]{Pie20_super} for a comprehensive modern formulation of Burgess' technique), at least for dimensions $n\geq 2$. This introduced a multiplicative energy that counts the number of integral tuples $(\xbf,\xbf',\ybf,\ybf')$, where $N_i\leq x_i,x_i'\leq N_i+H_i$ and $0<y_i,y_i'\leq K_i$ with certain conditions on $H_i,K_i$, such that
    \begin{multline}\label{eqn:literature_energy}
       (x_1\omega_1+x_2\omega_2+\cdots +x_n\omega_n)(y_1'\omega_1+y_2'\omega_2+\cdots+ y_n'\omega_n)\\
       = (x'_1\omega_1+x'_2\omega_2+\cdots +x'_n\omega_n)(y_1\omega_1+y_2\omega_2+\cdots +y_n\omega_n).
    \end{multline}
Karacuba's energy bound relies crucially on the restriction that the basis is a power basis. On the other hand, his new shifting method is insensitive to the type of basis, and results on character sums over finite fields thereafter followed this approach, ultimately reducing the main task of bounding a short character sum to bounding the associated multiplicative energy.
Chang \cite{Cha08} was able to bound this energy for arbitrary bases by applying the proof of a sum-product result from \cite{KS08} in the area of additive combinatorics, thus giving a nontrivial bound on \eqref{eqn:sum_ff} for $\Hbf=(H_1,...,H_n)$ with $\|\Hbf\| >p^{(2/5+\kappa)n}$, uniformly for all bases. This uniform bound superseded that of \cite{DL63} for $n\geq 5$. Subsequently, Chang \cite{Cha09} applied the divisor bound for quadratic number fields to improve the bound on the multiplicative energy associated to the $n=2$ case and obtained a nontrivial bound for $H_i> p^{1/4+\kappa}$, again uniform in the choice of basis. 
In 2010, Konyagin \cite{Kon10} translated the energy into a bound involving lattices and applied results from the geometry of numbers to produce an optimal and uniform bound on the energy in \eqref{eqn:literature_energy}, hence achieving the Burgess threshold of $p^{1/4+\kappa}$ for all dimensions, uniformly for all bases. For $n=2,3$, Gabdullin \cite{Gab18} proved nontrivial bounds on \eqref{eqn:sum_ff} for $\Hbf$ satisfying the weaker condition of $\|\Hbf\|>p^{(1/4+\kappa)n}$.
In our paper, we adapt ideas from many of these works.

\subsubsection{Results on character sums evaluated at forms}
For polynomials that take on the special shapes (products of $n$ linear forms in $n$ variables, linearly independent over $\bF_p$) mentioned in \cite{Bur68}, there have since been significant improvements that provide nontrivial bounds for lengths as short as $p^{1/4+\kappa}$. In the case of binary quadratic forms, Chang \cite{Cha09_binary} showed bounds for $H_i>p^{1/4+\kappa}$ for reducible forms of the shape $X_1^2+cX_2^2$ where $c$ is nonzero. When this is irreducible and $c>0$, the work of Burgess \cite{Bur67} implies a nontrivial bound in the range $H_i>p^{1/4+\kappa}$. 
This was extended in \cite{Cha09} to general binary quadratic forms of the shape $X_1^2+aX_1X_2+bX_2^2$ that are not perfect squares mod $p$; the irreducible case followed from Chang's result (in the same paper) on the $n=2$ case of character sums over finite fields, while the reducible case followed from an estimate on the multiplicative energy using the divisor bound in quadratic number fields. 
In 2016, Heath-Brown \cite{HB16} generalized this result to composite moduli that are odd and square-free.
When $F=\prod_{i=1}^n L_i(X_1,...,X_n)$, where $L_1,...,L_n$ are linear forms over $\bF_p$ that are linearly independent over $\bF_p$, Bourgain and Chang \cite{BC10} proved nontrivial bounds for $\Hbf=(H,...,H)$ with $H>p^{1/4+\kappa}$ for any $\kappa>0$.
Their approach is based on the shifting technique originating in \cite{Kar68,Kar70} mentioned above, and they bound the relevant multiplicative energy using results from geometry of numbers, as seen in \cite{Kon10}.

In addition to these special cases, Gillett \cite{Gil73} showed nontrivial bounds for forms belonging to a larger class, namely those in $n$ dimensions of degree $n$ that split over the algebraic closure $\overline{\bF}_p$, for $\Hbf = (H_1,...,H_n)$ with $H_i>p^{\beta_n+\kappa}$. More recently, Pierce and Xu \cite{PX20} proved nontrivial bounds for forms belonging to a generic class of homogeneous polynomials; their result is nontrivial for $\Hbf = (H_1,...,H_n)$ with $\|\Hbf\| \min_i H_i \gg p^{n/2+\kappa}$, and in particular, when $\Hbf=(H,...,H)$, this is satisfied as long as $H>p^{\beta_n+\kappa}$. Their method implements a novel stratification technique to bound complete character sums evaluated at multivariate polynomials whose leading form defines a singular variety, a case that remained intractable from the Weil bound applied in all previously mentioned short character sums results.

Theorem \ref{thm:main} treats the same class of forms as in \cite{Gil73}, and we improve that result in all dimensions $n\geq 2$. This is a special case of a more general result which we present in Theorem \ref{thm:main2}.

\subsection{An overview of the method}\label{sec:overview} 
We now highlight the essential difficulties of bounding short character sums and the methods we build to overcome them.
The overarching technique in all of the generalizations of Burgess-type bounds we surveyed in \S \ref{sec:literature} derives from one of the two variations of the Burgess amplification method. Both variations involve distributing copies of the original box $(\Nbf,\Nbf+\Hbf]=\prod_{i=1}^n (N_i,N_i+H_i]$ of summation, with the aim of efficiently obtaining a complete sum (over a different variable). 

The classic version in the one-dimensional case was developed by Burgess and refined by Gallagher and Montgomery \cite{GM10} and Heath-Brown \cite{HB13} (see \cite[\S 8]{Pie20_super} for an overview), which was then generalized to the multivariate case by Pierce \cite{Pie16} and further adapted to character sums evaluated at generic forms by Pierce and Xu \cite{PX20}. 
We will refer to this version as ``method I". The first step is to redistribute a copy of the sum by breaking the original box into residue classes with respect to another prime $q$ (different from $p$). For example, for $\xbf\in (\Nbf,\Nbf+\Hbf]$, we write $\xbf=\abf p+\zbf q$. 
By periodicity and multiplicativity of $\chi$ and homogeneity of $F$, $\chi(F(\abf p+\zbf q))=\chi(q^{\deg F})\chi(F(\zbf))$, leading to a sum over the new variable $\zbf$.
Averaging this sum over many choices of primes $q$ then produces many copies of the sum over $\zbf$ such that their starting points nearly cover a complete set of residues mod $p$. 
This amplification process initiates two further tasks.
One is to count the redundancies incurred in the averaging process, and the other is to estimate a complete sum of a higher moment of a character sum:
    \begin{align}\label{eqn:overview_tasks}
        \sum_{\substack{z_1,...,z_n \\ z_i \modd p}}\mathcal{A}(z_1,...,z_n)^2, 
        \qquad \sum_{\substack{z_1,...,z_n \\ z_i \modd p}} |\sum_{\substack{t_1,...,t_n\\ t_i\in (0,T_i]}} \chi(F(\tbf + \zbf))|^{2r},
    \end{align}
    where $\mathcal{A}(z_1,...,z_n)$ counts the number of $(a_1,...,a_n)\modd p ^n$ and primes $q$ (lying in a dyadic range) such that $z_i$ lies in $((N_i-a_ip-H_i)/q, (N_i-a_ip)/q]$ for $1\leq i \leq n$.
Here, the first task (bounding the first sum) is straightforward, even in the multivariate case, since the sum is independent of $F$; by an argument originating in \cite{HB13} then generalized in \cite{HBP15,Pie16,PX20}, this task reduces to counting ``starting points'' $(N_i-a_ip)/q$ (of the variable $z_i$) that lie close to each other. The second task (bounding the second sum), traditionally handled by the Weil bound for complete multiplicative character sums (as recorded e.g. in \cite[Theorem 11.23]{IK04}), becomes much more challenging in the multivariate case, because of the presence of a character sum evaluated at a polynomial whose leading form is highly singular. 
Pierce and Xu \cite{PX20} applied a stratification result of \cite{Xu20} for complete multiplicative character sums to overcome this difficulty.

In this paper, we follow another version of the Burgess method as introduced by Karacuba \cite{Kar68,Kar70} for multiplicative character sums over finite fields and adapted by Chang \cite{Cha09_binary,Cha09} to multiplicative character sums evaluated at reducible binary quadratic forms and by Bourgain and Chang \cite{BC10} to products of $n$ linear forms in $n$ variables. We will refer to this version as ``method II". 
Before comparing the two methods, let us briefly describe the amplification in this second approach.
For illustration purposes, let $F$ be a linear form $L$ in $n$ variables. We begin by translating the original range of summation by a product, say from $\xbf$ to $\xbf+ t\ybf$. (See e.g. \cite[\S 12.4]{IK04} for a proof of the classic Burgess bound in $n=1$ using this method.) By linearity of $L$, we write $L(\xbf+t\ybf) = L(\xbf) + tL(\ybf)$. After averaging over many choices of $t$ and $\ybf$ and rescaling by $\overline{L(\ybf)}$ mod $p$ (the multiplicative inverse of $L(\ybf)$ mod $p$), we get a character sum over the variable $t$. This step crucially relies on the linearity of $L$, and for this reason, before the present paper, method II had seen success only with character sums evaluated at products of $n$ linear forms over $\bF_p$ in $n$ variables (such as in \cite{BC10}, \cite{Cha09_binary,Cha09}) and character sums over finite fields where the argument is of the form $x_1\omega_1+\cdots + x_n\omega_n$, which is linear in $x_1,...,x_n$.
(Note that Heath-Brown's treatment in \cite{HB16} of character sums with square-free moduli $q$ evaluated at binary quadratic forms shifts initially by a product $t\ybf$ akin to method II, but proceeds with special properties of binary quadratic forms.)
 
Like in method I, the amplification process here produces two separate tasks. One is a redundancy count and the other is an estimate for a complete sum of a higher moment; for $F$ taking on the shape as described in Theorem \ref{thm:main}, these look like:
    \begin{align}\label{eqn:overview_tasks2}
           \sum_{\substack{z_1,...,z_s \\ z_i \in \bF_{p^{n_i}}}}\eta(z_1,...,z_s)^2, 
        \qquad \sum_{\substack{z_1,...,z_s \\ z_i \in \bF_{p^{n_i}}}} |\sum_{t\in (0,T]} \prod_{i=1}^s \psi_i(t+z_i)|^{2r},
    \end{align}
    where $\eta(z_1,...,z_s)$ counts the number of $(\xbf,\ybf)$'s in a box such that the simultaneous equations $z_i=\lambda_i(\xbf)\lambda_i(\ybf)^{-1}\in \bF_{p^{n_i}}$ for $1\leq i \leq s$ hold. (We defer the definitions of $s, n_i,\lambda_i,\psi_i$ until later, but we note now that $\lam_i$ depends on $F$.)
Here, the latter task is straightforward, because it involves a one-dimensional character sum over $t$, instead of a multi-dimensional sum over $\tbf$ in \eqref{eqn:overview_tasks}. The redundancy count, on the other hand, now encodes information of $F$, thus becoming the more challenging task. Interestingly, this count is a multiplicative energy, an object studied in arithmetic combinatorics.

Methods I and II differ, at least in dimensions $n\geq 2$, in two significant ways.
The first difference is the shape of forms each method can accommodate: method I (in the work of \cite{PX20}) applies to a generic class of forms (roughly those forms that cannot be made independent of one variable after a nonsingular, linear change of variables), while method II requires $F$ to factor into linear forms. 
While we still require linear structure over $\overline{\bF}_p$, we no longer require linearity over $\bF_p$. We achieve this by a change of variables that rewrites our polynomials of interest as products of norm forms, adapting a strategy of Gillett \cite{Gil73}. 
Then we exploit these norms to lift our character mod $p$ to a character over a finite field, at which point the original character sum evaluated at a polynomial transforms into a character sum over a finite field. We remark that the clever change of variables in \cite{Gil73} is very helpful, but Gillett applies method I of the Burgess amplification technique and arrives at a nontrivial bound only for $H> p^{\beta_n+\kappa}$, whereas we proceed with method II, which allows us to break past the $p^{\beta_n+\kappa}$ barrier.

The second difference is the distinctive essential difficulty each method must address; we can view this disparity as a trade-off between the difficulty levels of the two tasks post amplification. As described above, the main challenge of method I is to bound the complete sum (the second task), because it carries the information of the form $F$ (as clearly seen in \eqref{eqn:overview_tasks}), while the redundancy count (the first task) is independent of $F$. The relative difficulty of these two tasks is reversed in method II. In particular, the main difficulty of method II is to bound the first sum of \eqref{eqn:overview_tasks2}.

We now define and state the main energy bound required to prove Theorem \ref{thm:main}. Fix a positive integer $s\geq 1$, and let $(n_1,...,n_s)$ be a partition of $n$. For each $1 \leq i \leq s$, fix a basis $1,\omega_i,...,\omega_i^{n_i-1}$ of $\bF_{p^{n_i}}$.
Fix $n$ linear forms $L_{i,j}$, indexed by $1\leq i \leq s$ and $1\leq j\leq n_i$, that are linearly independent over $\bF_p$.
Define $ \lambda_i(\cdot) := L_{i,1}(\cdot)+L_{i,2}(\cdot)\omega_i + \cdots + L_{i,n_i}(\cdot)\omega_i^{n_i-1}$
for each $i$, and define the energy \begin{align}\label{eqn:energy_def}
    E((\Nbf,\Nbf+\Hbf]) := |\{(\xbf,\xbf',\ybf,\ybf')\in (\Nbf,\Nbf+\Hbf]^4: \lambda_i(\xbf)\lambda_i(\ybf')=\lambda_i(\xbf')\lambda_i(\ybf) \in \bF_{p^{n_i}} \text{ for all } i\}|.
    \end{align}
The usual definition of the multiplicative energy of a set $S$ (say in a ring) is $E(S) = |\{a,b,c,d\in S: ab=cd\}|$. We can equivalently define $E((\Nbf,\Nbf+\Hbf])$ as the energy of the set
    \begin{align*}
        S = \{(\lam_1(\xbf),...,\lam_s(\xbf))\in \bF_{p^{n_1}}\times \cdots \times \bF_{p^{n_s}}: \xbf\in (\Nbf,\Nbf+\Hbf]\},
    \end{align*}
where multiplication is coordinate-wise. Indeed, let $\underline{\lam}(\cdot)$ denote the tuple $(\lam_1(\cdot),...,\lam_s(\cdot))$.
Then $E(S)$ is equivalent to \eqref{eqn:energy_def} since $\underline{\lam}(\xbf) = \underline{\lam}(\xbf')$ if and only if  $\lam_i(\xbf)=\lam_i(\xbf')$ for all $1\leq i\leq s$, which holds if and only if $L_{i,j}(\xbf)\equiv L_{i,j}(\xbf') \modd p$ for all $1\leq i \leq s$ and $1\leq j \leq n_i$, and this is equivalent to $\xbf\equiv \xbf' \modd p$ by linear independence of the $L_{i,j}$'s. 
It is straightforward to check that $S$ has ample additive structure, e.g. a small sumset. In light of the sum-product phenomenon, we expect $S$ to have little multiplicative structure and hence small multiplicative energy. We confirm this by proving the following optimal upper bound, essential to our method.
(In our application, we will consider boxes whose side lengths are comparable in size, in the sense that $H\leq H_i\leq 2H$ for all $i$, for some $H$, e.g. $H=\min_i H_i$.)

\begin{thm}\label{thm:main_energy} Let $1\leq H\leq p^{1/2}$ be an integer. Let $\Nbf=(N_1,...,N_n)$ and $\Hbf=(H_1,...,H_n)$ such that $H\leq H_i\leq 2H$ for all $1\leq i\leq n$. 
Then for any $\varepsilon>0$, 
    \begin{align*}
        E((\Nbf,\Nbf+\Hbf])\ll_{n,\varepsilon}  H^{2n} p^\varepsilon.
    \end{align*}
\end{thm}

A priori, we have the following ``elementary'' upper and lower bounds. 
\begin{lemma}\label{lem:energy_elementary} Let $1\leq H\leq p$ be an integer. Let $\Nbf=(N_1,...,N_n)$ and $\Hbf=(H_1,...,H_n)$ such that $H\leq H_i\leq 2H$ for all $1\leq i\leq n$. Then
    $H^{2n}\ll_n E((\Nbf,\Nbf+\Hbf])\ll_n H^{3n}$.
\end{lemma}

The lower bound in Lemma \ref{lem:energy_elementary} is straightforward since the diagonal solutions, where $\xbf'=\xbf$ and $\ybf'=\ybf$, contribute $H^{2n}$; we give the full proof in \S \ref{sec:energy_prelim}. Theorem \ref{thm:main_energy}, whose proof is the heart of this paper, shows that this lower bound is essentially an upper bound as well.

We extend ideas from \cite{BC10} to prove Theorem \ref{thm:main_energy}.
At the core of this method is a recursive formula (Proposition \ref{prop:recursive}) for a de-symmetrized version of the energy in \eqref{eqn:energy_def}.  
The recursion then relies on gaining an advantage either by passing to a smaller box of the same dimension or by reducing the dimension of the box, until arriving at a trivial case. 
Here we present a rough, simplified version of this key recursive relation, to illustrate the philosophy. 
Let $E(n,H)$ denote, informally, the energy that counts tuples $(\xbf,\xbf',\ybf,\ybf')$, each lying in an $n$-dimensional box with side lengths $\Hbf$ where $H_i\approx H$, that satisfy certain simultaneous equations over certain finite fields. Then for a fixed $\kappa>0$ and $H'=\lfloor p^{-\kappa}H\rfloor$,
    \begin{align*}
        H^{-2n}E(n,H) \ll ( H^{-2n}E(n,H))^{1/2} ( (H')^{-2n}E(n,H'))^{1/2} + p^{n\kappa} 
        \sum_{n'<n}
        H^{-2n'}E(n',H).
    \end{align*}

To prove the recursive formula, we apply results from the geometry of numbers, including Minkowski's second theorem and bounds on the number of points in the intersection of a convex set and a lattice, as well as its dual lattice (see \S \ref{sec:lattices}).
Due to the lack of several desirable properties exploited in \cite{BC10}, the new features in our setting demand new treatment. These arise from two key differences between \eqref{eqn:energy_def} and the energy considered in \cite{BC10}. First, the equations in \eqref{eqn:energy_def} belong to finite field extensions of $\bF_p$ whereas, previously, all equations lived in $\bF_p$. Second, the equations are no longer symmetric, in the sense that they live in \textit{different} finite fields. These distinctions bring several difficulties to adapting the method of \cite{BC10}.
First, to access the dual lattice (see \S \ref{sec:key_results}), we require a lattice defined by equations over $\bF_p$; this is readily available in the setting of Bourgain-Chang. Second, this dual lattice is defined by equations involving a certain matrix $M_\zbf$ over $\bF_p$ (as we will see in \S \ref{sec:proof_dual}). In the case of \cite{BC10}, where $F$ splits into linear forms over $\bF_p$, this matrix is conveniently diagonal and hence symmetric, a property that is required for the argument to proceed. However, in the more general cases we consider, $M_\zbf$ is not symmetric. To handle this, we introduce a symmetrizing matrix that corresponds to choosing new bases for the family of finite fields that appear in \eqref{eqn:energy_def}. Finally, 
we develop a more intricate recursion procedure, in the proofs of Proposition \ref{prop:recursive} and Corollary \ref{cor:recursive}, to overcome the asymmetry in the dimension arising from the different degrees of finite field extensions in each step of the iteration.

The role of recursion is prominent and recurring in our work. It is the key engine that produces not only the optimal energy bound in Theorem \ref{thm:main_energy}, but also the main character sum bound in Theorem \ref{thm:main} (and Theorem \ref{thm:main2}). Indeed, we will see in \S \ref{sec:pf_burgess_recursive} that the amplification procedure (of Burgess, Karacuba, Konyagin) generates a recursive relation (Proposition \ref{prop:burgess_recursive}) that passes to character sums over smaller boxes, which we then iterate in \S \ref{sec:recursive_kon}.

Lastly, we mention an interesting feature that allows us to deduce bounds on character sums over \textit{arbitrary} boxes from bounds on character sums over boxes whose side lengths lie in a controlled range. In \S \ref{sec:comparable} and \S \ref{sec:pf_burgess_recursive}, we adapt an idea of Konyagin \cite{Kon10} to break the initial disproportionate box $\prod_{i=1}^n (N_i,N_i+H_i]$ into a union of comparable cubes, all of whose side lengths are roughly $H$ for some $H$. More precisely, all cubes look like $\prod_{i=1}^n (N'_i,N'_i+H'_i]$, where $H\leq H'_i\leq 2H$ for all $i$, for some uniformly chosen $H$ (e.g. $H=\min_i H_i$). 
This allows us to focus on boxes with a dyadic range of variation, which we show in \S \ref{sec:lattices} are
still permissible in the methods of the geometry of numbers; in contrast, the previous works of \cite{Kon10} and \cite{BC10} imposed that $H_i=H_j$ for all $i,j$.

\subsection{The strength of Theorem \ref{thm:main}}\label{sec:delta}
Since the bound \eqref{eqn:mainbound} depends on the parameter $r$, we can quantify the strength of the bound by computing the value of $r$ for which we achieve the strongest bound. In particular, for $H_{\mathrm{min}}=p^{1/4+\kappa}$ with a fixed $\kappa>0$, \eqref{eqn:mainbound} gives a bound of $\|\Hbf\| p^{-\delta}$ where we save approximately
    \begin{align*}
       \delta \approx \frac{4nr\kappa - n^2}{2r(n+2r)}.
    \end{align*}
Viewing this as a function in $r$, we see that it is maximized at the real value $n(1+\sqrt{1+2\kappa})/4\kappa$. So by choosing $r$ to be the nearest integer to $n(1+\sqrt{1+2\kappa})/4\kappa$, we obtain savings of approximately $\delta \approx 4\kappa^2/(1+\sqrt{1+2\kappa})^2$. For small $\kappa$, this behaves like $\kappa^2$ (independent of $n$).

\subsection{The cases $n\leq k$ and $k\leq n$} Without additional difficulty, we can extend Theorem \ref{thm:main} to a larger class of forms.
We prove the following more general result.
\begin{thm}\label{thm:main2}
Fix a prime $p$ and integers $n, k$ such that $1\leq k<2n$. Let $F\in \bF_p[X_1,...,X_n]$ be a form in $n$ variables and of degree $k$.
    Suppose $F$ splits into linear forms over the algebraic closure $\overline{\bF}_p$, 
     i.e. $F=(\underline{\beta^T_1}\Xbf)\cdots (\underline{\beta^T_k}\Xbf)$ with $\underline{\beta_i}\in (\overline{\bF}_p)^n$, and suppose the $k\times n$ matrix with rows $\underline{\beta^T_1},...,\underline{\beta^T_k}$ has full rank $\min\{k,n\}$ over $\overline{\bF}_p$. 
    Let $\Nbf=(N_1,...,N_n)$ and $\Hbf=(H_1,...,H_n)$ with $H_{\mathrm{min}}\leq p^{1/2}$. 
    Let $k'=\max\{k,2n-k\}$.
    Then for any integer $r>  k$ and $\varepsilon>0$,
        \begin{align*}
            S(F;\Nbf,\Hbf)\ll_{n,k,r,\varepsilon} 
             \|\Hbf\|H_{\mathrm{min}}^{-\frac{(2n-k')}{r}} p^{\frac{k(r+2n-k')}{4r^2}+\varepsilon}.
        \end{align*}
\end{thm}

Note that $k'=k$ when $n\leq k$ and $k'=2n-k$ when $k\leq n$. It suffices to assume $p> \max \{n,k\}$.
Theorem 1.1 corresponds to the case $n=k$ in Theorem \ref{thm:main2}, so we prove Theorem \ref{thm:main2} instead. (The potential for proving a result of the strength of Theorem \ref{thm:main} for $\Hbf=(H,...,H)$ has been mentioned by Chang in \cite{Cha16}.) 
We have the following immediate consequences.

\begin{cor} \label{cor:main2}
Assume the hypotheses of Theorem \ref{thm:main2} for $F\in \bF_p[X_1,...,X_n]$ a form of degree $k$ with $n\leq k$.
    Let $\kappa>0$ and $H_i\geq p^{k/(4(2n-k)) + \kappa}$ for all $i$, and suppose $H_{\mathrm{min}}\leq p^{1/2}$. Then there exists a positive $\delta=\delta(n,k,\kappa)$ such that $S(F;\Nbf,\Hbf)\ll_{n,k,\kappa} \|\Hbf\| p^{-\delta}$.
\end{cor}
This gives nontrivial bounds on short sums for a nonempty class of forms as long as $n\leq k<4n/3$ (since then $k/(4(2n-k))<1/2$). Since \cite{PX20} provides nontrivial bounds for generic forms, we compare their result to Corollary \ref{cor:main2}. 
Let $F$ denote a form satisfying the hypotheses of Theorem \ref{thm:main2}. Suppose $\chi$ has order $d$ and $F$ is ``$(d,p)$-admissible'' in the sense of \cite[Condition 1.1]{PX20}, i.e. $F$ can be written in the form $F=G^d H$ where $G,H\in \bF_p[X_1,...,X_n]$ and $H$ is $d$th-power-free over $\bF_p$, and there is no change of variables in $\GL_n(\bF_p)$ under which $H$ can be written in terms of fewer variables. Then for $\Hbf=(H,...,H)$, \cite[Corollary 1.2]{PX20} gives nontrivial bounds in the range $H > p^{\beta_n+\kappa}$, where recall $\beta_n = n/(2(n+1))$.
Corollary \ref{cor:main2} provides a nontrivial bound on $S(F;\Nbf,\Hbf)$ for $H_i$ as short as $H_i\geq p^{k/(4(2n-k))}$, which is an improvement when $n\leq k< 4n^2/(3n+1)$.
(For $n< k$, this gives a nonempty class of forms as soon as $n\geq 5$.)

\begin{cor} \label{cor:main3}
Assume the hypotheses of Theorem \ref{thm:main2} for $F\in \bF_p[X_1,...,X_n]$ a form of degree $k$ with $k\leq n$.
    Let $\kappa>0$ and $H_i\geq p^{1/4 + \kappa}$ for all $i$, and suppose $H_{\mathrm{min}}\leq p^{1/2}$. Then there exists a positive $\delta=\delta(n,k,\kappa)$ such that $S(F;\Nbf,\Hbf)\ll_{n,k,\kappa} \|\Hbf\| p^{-\delta}$.
\end{cor}
As we will see in \S \ref{sec:gillett_vars}, the case $k<n$ consists of forms $F$ such that there exists a nonsingular, linear change of variables over $\bF_p$ with which $F$ can be written in terms of fewer than $n$ variables. In particular, these forms do not satisfy \cite[Condition 1.1]{PX20} and hence lie outside of the class of forms to which \cite{PX20} is applicable. Thus this is the first time that nontrivial bounds are shown to hold for sums, as short as $p^{1/4+\kappa}$, of characters evaluated at forms with $k<n$.

To prove Theorem \ref{thm:main2}, we require a bound on a multiplicative energy analogous to \eqref{eqn:energy_def}.
Fix a positive integer $s\geq 1$, and let $(k_1,...,k_s)$ be a partition of $k$. For each $1 \leq i \leq s$, fix a basis $1,\omega_i,...,\omega_i^{k_i-1}$ of $\bF_{p^{k_i}}$. 
Fix $k$ linear forms $L_{i,j}$, indexed by $1\leq i \leq s$ and $1\leq j \leq k_i$, each in $n$ variables, such that the $k\times n$ matrix defined by the coefficients of $L_{i,j}$ has full rank $\min\{k,n\}$. Define $ \lambda_i(\cdot) := L_{i,1}(\cdot)+L_{i,2}(\cdot)\omega_i + \cdots + L_{i,k_i}(\cdot)\omega_i^{k_i-1}$
for each $i$, and define the energy
\begin{align}\label{def:energy_general}
    E_{n,k}((\Nbf,\Nbf+\Hbf]) := |\{(\xbf,\xbf',\ybf,\ybf')\in (\Nbf,\Nbf+\Hbf]^4: \lambda_i(\xbf)\lambda_i(\ybf')=\lambda_i(\xbf')\lambda_i(\ybf) \in \bF_{p^{k_i}} \text{ for all } i\}|.
    \end{align}

\begin{thm}\label{cor:main_energy} Fix integers $n,k$ with $1\leq k\leq 2n$, and let $k'=\max\{k,2n-k\}$. Let
 $1\leq H\leq p^{1/2}$, and let $\Hbf=(H_1,...,H_n)$ such that $H\leq H_i\leq 2H$ for all $1\leq i\leq  n$. 
Then for any $\varepsilon>0$, 
    \begin{align*}
        E_{n,k}((\Nbf,\Nbf+\Hbf])\ll_{n,k,\varepsilon}  H^{2k'} p^\varepsilon.
    \end{align*}
\end{thm}
This bound is a straightforward corollary of Theorem \ref{thm:main_energy}, and we prove it in \S \ref{sec:cor_energy}. We expect this bound to be sharp (up to $p^\varepsilon$) when $k<n$, while in the case $n<k$, we expect the truth to be at most $H^{2n}p^\varepsilon$; in \S \ref{ref:sharp}, we give some examples that support this.

\subsection{Outline of the paper} 
In \S \ref{sec:method}, we motivate the employment of norm forms and prove Proposition \ref{prop:gil_main0}, which 
provides an important change of variables that ``lifts'' the character sum to a finite field extension. 
In \S \ref{sec:burgess}, we carry out the Burgess argument to generate a recursive relation, Proposition \ref{prop:burgess_recursive}, conditional on Theorem \ref{cor:main_energy}, from which we deduce Theorem \ref{thm:main2}.
In \S \ref{sec:energy_prelim}, we prove the preliminary energy bounds given in Lemma \ref{lem:energy_elementary}.
In \S \ref{sec:lattices}, we present properties of lattices and set the stage to efficiently bound multiplicative energy. In \S \ref{sec:energy}, we prove a recursive formula, Proposition \ref{prop:recursive}, from which we derive Theorem \ref{thm:main_energy}. In \S \ref{sec:cor_energy}, we deduce Theorem \ref{cor:main_energy} from Theorem \ref{thm:main_energy}.

\subsection{Notation}
We let $A=O(B)$ or Vinogradov notation $A\ll B$ denote $|A|\leq C|B|$ for some constant $C>0$, and $A\ll_\alpha B$ if the constant $C$ depends on $\alpha$. Let $A\asymp B$ denote $B\ll A\ll B $. 
For a finite set $S$, let $|S|$ denote its cardinality; for a set $S\subseteq \bR^m$, let $\vol(S)$ denote its Lebesgue measure. For a matrix $M$, let $M^T$ denote the transpose matrix. We use boldface $\xbf\in \bR^m$ to denote either a row vector $(x_1,...,x_m)$ or a column vector $(x_1,...,x_m)^T$, and this will be clear from context; additionally, $\mathbf{0}=(0,...,0)$. For two integral vectors $\xbf$ and $\ybf$, we write $\xbf \equiv \ybf \modd p$ to mean $x_i\equiv y_i \modd p$ for all $i$. We use underline $\underline{\omega}$ to denote a vector with entries belonging to a finite field extension. Let $M_{n\times m}(R)$ denote the ring of $n\times m$ matrices with entries in a ring $R$, and let $\GL_m(R)$ denote the set of invertible $m\times m$ matrices with entries in $R$. For matrices $M_1,...,M_r$ where $M_i$ is $m_i\times n$, let
    \begin{align*}
        [M_i]_{1\leq i \leq r}:=\begin{bmatrix}
            M_1\\
            \vdots\\
            M_r
        \end{bmatrix}
    \end{align*}
denote the stacked matrices.

All interval notation will represent discrete intervals, i.e. $[a,b]$ denotes all integers $x$ such that $a\leq x \leq b$.  For $\Nbf=(N_1,...,N_m)$ and $\Hbf=(H_1,...,H_m)$,
let $(\Nbf,\Nbf+\Hbf]$ denote the $m$-dimensional, discrete box $\prod_{i=1}^m (N_i,N_i+H_i]$, and let $\|\Hbf\| = \prod_{i=1}^m H_i$. 
In \S \ref{sec:lattices} and \S \ref{sec:energy}, we will let $B_\Hbf$ denote the symmetric, real-variable box $\{\xbf\in \bR^m: -H_i\leq x_i\leq H_i\}$ and $D_{\Hbf}$ denote the symmetric, discrete box $[-\Hbf,\Hbf]$.

\section{Norm forms and lifting of characters}\label{sec:method}

\subsection{Lifting characters mod $p$}\label{sec:lifting}

There is a correspondence between a multiplicative character mod $p$ evaluated at an irreducible binary quadratic form mod $p$ and a multiplicative character over $\bF_{p^2}$, as observed in \cite{Bur68} and applied in \cite{Cha09} and \cite{HB16}. Indeed, given a binary quadratic form $f(X_1,X_2)=X_1^2+aX_1X_2+bX_2^2$ that is irreducible over $\bF_p$, we can factorize it over a finite field by identifying $\bF_{p^2}$ with $\bF_p(\omega)$, where $\omega$ is a root of $f$. Then $f(X_1,X_2)=(X_1+X_2\omega )(X_1 + X_2\omega')$ in $\bF_{p^2}$ where $\omega'$ is the conjugate of $\omega$; moreover, $1,\omega$ is a basis of $\bF_{p^2}$ over $\bF_p$. We can thus define a multiplicative character $\psi$ of $\bF_{p^2}$ by
    \begin{align*}
        \psi(x_1+\omega x_2) = \chi(f(x_1,x_2)).
    \end{align*}
    One can check that this is multiplicative by verifying that $f(x_1,x_2)f(x_1',x_2') = f(x_1'',x_2'')$ where $x_1''+x_2''\omega = (x_1+x_2\omega )(x_1'+x_2'\omega )$.
Once this identification is made, results on character sums over finite fields become accessible; for example, for irreducible binary quadratic forms, \cite{Cha09} passes to \textit{incomplete} character sums over finite fields, while \cite{HB16} applies the Weil bound for \textit{complete} character sums over finite fields.

This raises a natural question, essential to our work, of how such a dictionary between characters of a base field (say $\bF_p$) evaluated at a form and characters of a field extension of finite degree (say $\bF_{p^n}$) can be generalized to higher dimensions. Precisely, fix $n\geq 1$ and fix a power basis $1,\omega,...,\omega^{n-1}$ for $\bF_{p^n}$, where $\omega$ is the root of an irreducible polynomial mod $p$ of degree $n$. We ask what types of polynomials $f$ satisfy a correspondence of the form
     \begin{align}\label{eqn:normlike_0}
        \psi(x_1+x_2\omega+\cdots + x_n\omega^{n-1}) = \chi(f(x_1,...,x_n))
    \end{align}
for all $(x_1,...,x_n) \in (\bF_p)^n$, for some multiplicative character $\psi$ of $\bF_{p^n}$.
Our method capitalizes on this dictionary;
in fact, let us see that lifting a character mod $p$ to a character of a finite field essentially restricts to the class of polynomials considered in Theorem \ref{thm:main}, thus demonstrating that the theorem maximizes the capabilities of this method.

Indeed, for \eqref{eqn:normlike_0}
to hold for all $x_1,...,x_n$, we need $f$ to exhibit multiplicative behaviour in the sense of 
    \begin{align*}
        f(x_1,...,x_n)f(x'_1,...,x'_n) = f(x''_1,...,x''_n),
    \end{align*}
where $x''_1+x''_2\omega+\cdots+x''_n\omega^{n-1}=(x_1+x_2\omega+\cdots+x_n\omega^{n-1})(x'_1+x'_2\omega+\cdots+x'_n\omega^{n-1})$, for all $(x_1,...,x_n)$ and $(x_1',...,x_n')$ in $(\bF_p)^n$. We can implicitly view $f$ as a function $\phi$ from $\bF_{p^n}$ to $\bF_p$ with the identification $\phi(x_1+x_2\omega+\cdots +x_n\omega^{n-1}) = f(x_1,...,x_n)$. Then the above multiplicative property is equivalent to 
    \begin{align}\label{eqn:normlike_1}
        \phi(ab) = \phi(a)\phi(b)
    \end{align}
for all $a,b\in \bF_{p^n}$. In addition, for $a\in \bF_p$, we have $\phi(a) =    f(a,0,...,0)$. If this vanishes for some $a \neq 0$, then by \eqref{eqn:normlike_0}, $\psi(a) =0$ which is a contradiction since $\psi^{p-1}(a)=1$ again by \eqref{eqn:normlike_0}. Hence $f(a,0...,0)\neq 0$ for $a\neq 0$, which means $f$ has a monomial of the form $cX_1^{d_1}$ for some nonzero $c$ and $d_1$. 
In fact, $c=1$, since for any integer $m\geq 1$, $\phi(a^m) = f(a^m,0,...,0) = c(a^m)^{d_1}$, but also $\phi(a^m) = \phi(a)^m = f(a,0,...,0)^m = (ca^{d_1})^m$. It follows that $c^m=c \in \bF_p$ for all $m$ and hence $c=1$. We can repeat this to show that $f$ has a monomial of the form $X_i^{d_i}$ for all $i$. It is then reasonable to expect $f$ to be homogeneous. 
Indeed, for an element $a_1a_2\omega$ where $a_1,a_2\in \bF_p$, we compute $\phi(a_1a_2 \omega)$ in two ways: on one hand, $\phi(a_1a_2 \omega) = f(0,a_1a_2,0,...0)=(a_1a_2)^{d_2}$, and on the other hand, assuming \eqref{eqn:normlike_1}, $\phi(a_1a_2\omega) = \phi(a_1)\phi(a_2\omega) = f(a_1,0,...,0)f(0,a_2,0,...0)=a_1^{d_1}a_2^{d_2}$. Hence $d_1=d_2$. If we additionally impose that the degree of $f$ is $n$,
then $\phi$ satisfies the following properties:
    \begin{enumerate}
        \item $\phi(x_1+x_2\omega+\cdots +x_n\omega^{n-1}) = f(x_1,...,x_n)$ for a polynomial $f\in \bF_p[X_1,...,X_n]$,
        \item $\phi(ab) = \phi(a)\phi(b)$ for all $a,b\in \bF_{p^n}$, and
        \item $\phi(a)=a^n$ for all $a\in \bF_p$. 
    \end{enumerate}
A function on $\bF_{p^n}$ that satisfies (1) is called a polynomial function. By the notation of \cite{Fla53}, a function that satisfies (2) and (3) is called a norm-like function. 
We may now conclude, based on our observations and the next lemma, that
if a polynomial $f$ of degree $n$ were to satisfy \eqref{eqn:normlike_0}, then it must be a norm form. 

\begin{lemma}[{\cite[Theorem 5]{Fla53}}]\label{lem:flanders} Let $\bF_q$ be a finite field and let $n\geq 1$. Suppose $\phi:\bF_{q^n} \rightarrow \bF_q$ is a norm-like function, and suppose further $\phi$ is a polynomial function of degree at most $n$. Then $\phi = N$, where $N:\bF_{q^n}\rightarrow \bF_q$ is the norm.
\end{lemma}

This provides a necessary condition on $f$; the sufficient condition is also known. Indeed, 
let $\chi$ be a character modulo $p$. Fix $n\geq 1$ and let $\omega_1,...,\omega_n$ be a basis for $\bF_{p^n}$. Let $N:\bF_{p^n}\rightarrow \bF_p$ denote the norm, and define $\psi=\chi\circ N$ to be the lifted character of $\chi$ to $\bF_{p^n}$ (see e.g. \cite[\S 5.4]{LN97}). Then $\psi$ is a multiplicative character of $\bF_{p^n}$ by the multiplicative property of the norm. (This lift was applied by Gillett \cite{Gil73} to access complete character sums over finite fields.)

Nontrivial bounds hold for character sums evaluated at norm forms, since
    \begin{align*}
      \sum_{\xbf\in (\Nbf,\Nbf+\Hbf]} \chi(N(x_1,...,x_n)) =\sum_{\xbf\in (\Nbf,\Nbf+\Hbf]} \psi(x_1,...,x_n),
    \end{align*}
    and
Konyagin's work \cite[Theorem 1]{Kon10} proves nontrivial bounds on the sum on the right-hand side for $H_i>p^{1/4+\kappa}$. We will lift characters with norm forms to exploit this. A priori, with no extra work, we can further consider products of norm forms in disjoint variables. Precisely, fix a positive integer $s$. For $1\leq i \leq s$, let $\omega_{i,1},...,\omega_{i,n_i}$ be a basis for the finite field $\bF_{p^{n_i}}$. Write $\Xbf = (X_{1,1},...,X_{1,n_1},...,X_{s,1},...,X_{s,n_s})$ and consider the form 
    \begin{align}\label{eqn:disjoint}
        F(\Xbf) = \prod_{i=1}^s N_i (X_{i,1}\omega_{i,1}+\cdots + X_{i,n_i}\omega_{i,n_i}),
    \end{align}
where $N_i:\bF_{p^{n_i}}\rightarrow \bF_p$ is the norm.
Then the character sum evaluated at $F(\xbf)$ reduces to a product of character sums over (different) finite fields, at which point we can again apply \cite{Kon10} $s$ times.
The forms we work with in Theorem \ref{thm:main}, as we will show in the next section, are precisely the orbit of these products of norm forms under the action of $\GL_n(\bF_p)$ on the $n$ variables. 
We ultimately express the original $S(F;\Nbf,\Hbf)$ as a sum that resembles
    \begin{align*}
        \sum_{\xbf\in (\Nbf,\Nbf+\Hbf]}\prod_{i=1}^s\psi_i(L_{i,1}(\xbf)+L_{i,2}(\xbf)\omega_i + \cdots +L_{i,n_i}(\xbf)\omega_i^{n_i-1}),
    \end{align*}
where the $L_{i,j}$'s are linear forms, $1,\omega_i,...,\omega_i^{n_i-1}$ is a basis for $\bF_{p^{n_i}}$, and $\psi_i$ is a multiplicative character of $\bF_{p^{n_i}}$.
Here, we cannot directly apply \cite{Kon10} when $s\geq 2$ because unlike \eqref{eqn:disjoint}, the variables in each factor of the product are not disjoint. Nevertheless, we will later take advantage of the linear behaviour in the argument of the character, since it is amenable to the shifting technique of \cite{Kar68,Kar70}.

\subsection{A key change of variables}\label{sec:gillett_vars}
There exists a change of variables with which we can rewrite the polynomials considered in Theorem \ref{thm:main2} in terms of products of norm forms. The proof builds on \cite[Lemma 3]{Gil73}.

\begin{prop}\label{prop:gil_main0} Fix a prime $p$ and integers $n,k\geq 1$. Let $F\in \bF_p[X_1,...,X_n]$ be a form in $n$ variables and of degree $k$.
Suppose $F$ splits into linear factors over $\overline{\bF}_p$, 
    i.e. $F=(\underline{\beta^T_1}\Xbf)\cdots (\underline{\beta^T_k}\Xbf)$ with $\underline{\beta_i}\in (\overline{\bF}_p)^n$, and suppose the $k\times n$ matrix with rows $\underline{\beta^T_1},...,\underline{\beta^T_k}$ has full rank $\min\{k,n\}$ over $\overline{\bF}_p$. 
Write the factorization of $F$ over $\bF_p$ as
    \begin{align*}
        F(X_1,...,X_n) = \prod_{i=1}^s F_i(X_1,...,X_n),
    \end{align*}
where each $F_i$ is irreducible over $\bF_p$ and $\deg F_i = k_i$, so that $k_1+\cdots + k_s =k$. Then there exists a change of variables $\ubf=A\mathbf{X}$, where $A\in M_{k\times n}(\bF_p)$ has full rank $\min\{k,n\}$, which we write as
\begin{align*}
        \begin{bmatrix}
            u_{1,1}\\
            \vdots\\
            u_{1,k_1}\\
            \vdots\\
             u_{s,1}\\
            \vdots\\
            u_{s,k_s}
        \end{bmatrix}
        \equiv
        \begin{bmatrix}
            a_{1,1,1}&\cdots &a_{1,1,n}\\
            \vdots&&\vdots\\
            a_{1,k_1,1}&\cdots &a_{1,k_1,n}\\
            \vdots&&\vdots\\
             a_{s,1,1}&\cdots &a_{s,1,n}\\
            \vdots&&\vdots\\
            a_{s,k_s,1}&\cdots &a_{s,k_s,n}
        \end{bmatrix}
            \begin{bmatrix}
                X_1\\
                \vdots\\
                X_n
            \end{bmatrix}
            \equiv A\mathbf{X},
    \end{align*}
such that 
    \begin{align*}
        F(X_1,...,X_n) = \prod_{i=1}^s N_i(u_{i,1}+u_{i,2}\omega_i+\cdots +u_{i,k_i}\omega_i^{k_i-1}),
    \end{align*}
    where for every $1\leq i\leq s$, $N_i:\bF_{p^{k_i}}\rightarrow \bF_p$ is the norm and $1,\omega_i,...,\omega_i^{k_i-1}$ is a basis for $\bF_{p^{k_i}}$.
    Moreover, the submatrix
        \begin{align}\label{eqn:def_Ui}
        U_i := \begin{bmatrix}
            a_{i,1,1}&\cdots &a_{i,1,n}\\
            \vdots&&\vdots\\
            a_{i,k_i,1}&\cdots &a_{i,k_i,n}
            \end{bmatrix}_{k_i\times n}
    \end{align}
    has full rank $\min\{k_i,n\}$, and
   if $n=k$, then for all distinct indices $1\leq i_1,...,i_r\leq s$, 
        \begin{align}\label{eqn:rank_Ui}
          \rank
          [U_{i_j}]_{1\leq j \leq r}
           =\min\{k_{i_1}+\cdots+k_{i_r},n\}.
        \end{align}
\end{prop}

To prove Proposition \ref{prop:gil_main0}, we will show that it suffices to consider the case when $F$ is irreducible over $\bF_p$ and the coefficient of $X_1^k$ is 1, for which we record the following result.

\begin{prop}\label{prop:gil_main2} Fix a prime $p$ and integers $n,d\geq 1$ where $p>\max\{n,d\}$.
    Let $F\in \bF_p[X_1,...,X_n]$ be an irreducible form over $\bF_p$ of degree $d$. Suppose $F$ factors into linear forms over $\overline{\bF}_p$ and the coefficient of $X_1^d$ is 1, i.e.
    \begin{align}\label{eqn:F_linear_factors}
          F(X_1,...,X_n) = \prod_{i=1}^d (X_1+\beta_{i2}X_2+\cdots + \beta_{in}X_n)
    \end{align}
where $\beta_{ij}\in \overline{\bF}_p$. 
 Then
    \begin{align}\label{eqn:gil_cond}
        \begin{bmatrix}
            1&\beta_{12}&\cdots &\beta_{1n}\\
            \vdots&\vdots&&\vdots\\
            1&\beta_{d2}&\cdots &\beta_{dn}
        \end{bmatrix}_{d\times n}
    \end{align}
has full rank $\min\{d,n\}$ over $\bF_p$,
and
there exists a change of variables $\ubf=U\mathbf{X}$, where $U\in M_{d\times n}(\bF_p)$ has full rank, such that 
    \begin{align*}
        F(X_1,...,X_n) = N(u_1+u_2\omega+\cdots +u_d\omega^{d-1}),
    \end{align*}
where $N:\bF_{p^d}\rightarrow \bF_p$ is the norm and $1,\omega,...,\omega^{d-1}$ is a basis for $\bF_{p^d}$.
\end{prop}

\begin{proof}
We break the proof into several steps. 
First we will show that $F=F(X_1)$ is irreducible and separable as a single-variable polynomial over $K=\bF_p(X_2,...,X_n)$.
Then we will prove the key fact that
    \begin{align}\label{eqn:gillett_keyfact}
        \beta_{ij}\in \bF_{p^d}
    \end{align}
for all $i,j$ (or more precisely, $\bF_p(\beta_{12},...,\beta_{1n})\cong \bF_{p^d}$ and for each $i,j$, $\beta_{ij} = \sigma(\beta_{1j})$ for some $\sigma \in \Gal(\bF_{p^d}/\bF_p)$ depending on $i,j$).
We then deduce from these two claims that all of the linear factors of $F$ in \eqref{eqn:F_linear_factors} are conjugates of the first linear factor. This allows us to show that \eqref{eqn:gil_cond} has full rank over $\bF_p$ and generate a change of variables of full rank over $\bF_p$ such that $F$ is expressed as a norm, hence establishing the proposition.

First we claim that $F=F(X_1)$ is irreducible as an element of $K[X_1]$, where $K=\bF_p(X_2,...,X_n)$.
Indeed, $F$ is irreducible (and nonconstant) over $\bF_p[X_2,...,X_n]$ by assumption, where $\bF_p[X_2,...,X_n]$
  is a unique factorization domain. Then $F$ is irreducible over the field of fractions $K$, as a consequence of Gauss' lemma on factorization of polynomials (see e.g. \cite[Theorem 11.3.9d]{Art91}). 
To show separability, suppose the contrary that $F$ has a double root $B$. Then $X-B$ divides both $F$ and the formal derivative $F'$, so $X-B$ divides their greatest common divisor. Since $F$ is irreducible over $K$ and $K$ is a principal ideal domain, it follows that $F$ is a prime element (see e.g. \cite[Proposition 11.2.11b]{Art91}) and hence $F$ divides $F'$. But $F'$ has degree less than $F$, and so it must be the case that all coefficients of $F'$ are $0\modd p$ and hence each monomial (in $X_1$) in $F$ has degree that is a power of $p$. In particular, $d\geq p$ which is a contradiction.

Now we prove the key fact \eqref{eqn:gillett_keyfact} that $\beta_{ij}\in \bF_{p^d}$ for all $i,j$.  
Define
        \begin{align}\label{eqn:def_h}
            G(X_1,...,X_n) = \prod_i (X_1+C_i),
        \end{align}
where $C_1 = B_1=\beta_{12}X_2+\cdots + \beta_{1n}X_n$ and the product runs through all expressions of the form
    \begin{align*}
        C_i= \gamma_{i2}X_2+\cdots + \gamma_{in}X_n,
    \end{align*}
where $\gamma_{ij} = \sigma (\beta_{1j})$ for some $\sigma \in \Gal(\bF_p(\beta_{1j})/\bF_p)$.
In particular, this means that for each $j$, $\beta_{1j}$ and its conjugates appear equally often.
We claim that $G\in \bF_p[X_1,...,X_n]$. Expanding the right-hand side of \eqref{eqn:def_h}, we see that the coefficient of the monomial $X_1^{c_1}X_2^{c_2}\cdots X_n^{c_n}$ is
    \begin{align}\label{eqn:coeff}
\sum_{\substack{i_{21},...,i_{2c_2}\\i_{21}+\cdots+i_{2c_2}=c_2}} \cdots \sum_{\substack{i_{n1},...,i_{nc_n}\\i_{n1}+\cdots+i_{nc_n}=c_n}} (\gamma_{i_{21},2}\cdots \gamma_{i_{2c_2},2})\cdots (\gamma_{i_{n1},n}\cdots \gamma_{i_{nc_n},n}).   
        \end{align}
For each $2\leq j \leq n$, the term
    \begin{align}\label{eqn:coeff2}
\sum_{\substack{i_{j1},...,i_{jc_j}\\i_{j1}+\cdots+i_{jc_j}=c_j}}  \gamma_{i_{j1},j}\cdots \gamma_{i_{jc_j},j}
    \end{align}
is a symmetric polynomial in the $\gamma_{ij}$'s, with coefficients in $\bF_p$, since each conjugate appears equally often. Hence by the fundamental theorem of symmetric polynomials, this can be written as a polynomial, with coefficients in $\bF_p$, in the elementary symmetric polynomials of $\sigma(\beta_{1j})$ for $\sigma \in \Gal(\bF_p(\beta_{1j})/\bF_p)$. Recall the fundamental theorem of symmetric polynomials states that any symmetric polynomial with coefficients in a ring $R$ can be written uniquely as a polynomial, with coefficients in $R$, in terms of the elementary symmetric polynomials (see e.g. \cite[Theorem 14.3.4]{Art91}). Furthermore, these elementary symmetric polynomials themselves are elements of $\bF_p$, since they are coefficients of the minimal polynomial $P_1=\prod_{\sigma} (X-\sigma(\beta_{1j}))$, where the product runs through $\sigma \in \Gal(\bF_p(\beta_{1j})/\bF_p)$. Hence \eqref{eqn:coeff2} is an element of $\bF_p$ and so is \eqref{eqn:coeff}. This shows that $G\in \bF_p[X_1,...,X_n]$.

Next we compare $G$ to $F$.
View $G$ as a polynomial in $K[X_1]$, where recall $K=\bF_p(X_2,...,X_n)$, and observe that $G$ and $F$ share a root $-B_1$. Since $F$ is irreducible (and prime) in $K[X_1]$, it follows that $F$ divides $G$, and so each $\beta_{ij}$ is a conjugate of $\beta_{1j}$. We further claim that each conjugate appears equally often, by following a similar argument as above.
Indeed, we can expand $F$ and write its coefficients as in \eqref{eqn:coeff}. Then for each $j$, the term as in \eqref{eqn:coeff2} is a (formal) symmetric polynomial in the $\beta_{ij}$'s, and hence again by the fundamental theorem of symmetric polynomials, they can be written as polynomials in the elementary symmetric polynomials, which are coefficients of the polynomial $P_2=\prod_{i=1}^d (X-\beta_{ij})$. Note that the coefficients of $P_2$ lie in $\bF_p$, since the coefficients of $F$ lie in $\bF_p$. Then by the irreducibility of the minimal polynomial $P_1$, we deduce that $P_1$ must divide $P_2$, so write $P_2=P_1P_3$. If $P_3$ is a constant, then we are done. If not, then $P_3$ has a conjugate of $\beta_{1j}$ as a root, so again we deduce that $P_1$ divides $P_3$, so that $P_2=P_1^2 P_4$. Repeat this until $P_i$ is a constant. We conclude that $P_2$ is a power of $P_1$ up to a constant and so each conjugate of $\beta_{1j}$ must appear the same number of times.

Since each conjugate of $\beta_{1j}$ appears equally often, every conjugate of $\beta_{1j}$ appears at least once, and since $F$ has degree $d$, the number of distinct conjugates divides $d$. Hence for each $j$, $\bF_p(\beta_{1j})$ is an extension of degree dividing $d$, and so $\bF_p(\beta_{12},...,\beta_{1n})\cong \bF_{p^{d'}}$ for some $d'$ dividing $d$. 
If $d'<d$, then the polynomial $H = \prod_{\sigma} (X_1 + \sigma(\beta_{12})X_2+\cdots + \sigma(\beta_{1n})X_n)$,
where the product runs through $\sigma \in \Gal(\bF_{p^{d'}}/\bF_p)$, lies in $\bF_p[X_1,...,X_n]$ and shares a root with $F$. Since by supposition $F$ is irreducible in $\bF_p[X_1,...,X_n]$, $F$ divides $H$ which is a contradiction since $\deg H = d'$. Hence we must have $d'=d$, which proves \eqref{eqn:gillett_keyfact}.

Let $1, \omega, ..., \omega^{d-1}$ denote a basis of $\bF_{p^d}$. For $2\leq j \leq n$, write
        \begin{align*}
            \beta_{1j} = b_{j1} + b_{j2}\omega  + \cdots + b_{jd}\omega^{d-1},
        \end{align*}
where $b_{j1},...,b_{jd}\in \bF_p$. 
        Next we show 
 \begin{align}\label{eqn:betaij_b}
      \beta_{ij}=  \sigma_i(\beta_{1j}) = b_{j1} + b_{j2}\omega^{p^{i-1}}  + \cdots + b_{jd}\omega^{p^{i-1}(d-1)}
    \end{align}
for $1\leq i \leq d$ and $2\leq j \leq n$.
For $1\leq i \leq d$, let 
$B_i = \beta_{i2}X_2+\cdots + \beta_{in}X_n$
    and write
        \begin{align*}
            F = (X_1+B_1)\prod_{i=2}^d (X_1+B_i).
        \end{align*}
Since $F$ is irreducible over $K$ and has degree $d$, the extension $K(B_1)$ has degree $d$.
We claim that
    \begin{align}\label{eqn:galois}
        \Gal(K(B_1)/K)\cong \Gal(\bF_{p^d}/\bF_p).
    \end{align}
Consider the splitting field $K(B_1,...,B_d)$ of $F$ over $K$. Since $F$ is separable over $K$, the splitting field is a Galois extension. Consider the map
    \begin{align*}
      \phi:  \Gal(K(B_1,...,B_d)/K) \rightarrow  \Gal(\bF_{p}(\beta_{12},...,\beta_{1n})/\bF_p)
    \end{align*}
defined by $\phi(\sigma)=\sigma|_{\bF_{p^d}}$. 
We first check this is well-defined. Let $\sigma \in \Gal(K(B_1,...,B_d)/K)$ so that $\sigma(B_1) = B_i$ for some $i$. By definition of $B_i$, we have $\phi(\sigma)(\beta_{1j})=\beta_{ij}$, which lies in $\bF_{p}(\beta_{12},...,\beta_{1n})$ by assumption. Furthermore, $\phi(\sigma)$ is injective since $\sigma$ is, and hence $\phi(\sigma)\in \Gal(\bF_{p}(\beta_{12},...,\beta_{1n})/\bF_p)$. This verifies that $\phi$ is well-defined.
Next note that $\phi$ is injective since if $\sigma|_{\bF_{p^d}} =\tilde{\sigma}|_{\bF_{p^d}}$, then $\sigma(\beta_{ij}) = \tilde{\sigma}(\beta_{ij})$ for all $i,j$ and hence $\sigma = \tilde{\sigma}$. In particular this tell us that $|\Gal(K(B_1,...,B_d)/K)| \leq  |\Gal(\bF_{p^d}/\bF_p)|=d$. On the other hand, $|\Gal(K(B_1,...,B_d)/K)| = [K(B_1,...,B_d):K]\geq [K(B_1):K] =d$. We conclude that $K(B_1) = K(B_1,...,B_d)$ and that $\phi$ is an isomorphism.

By \eqref{eqn:galois}, the other roots of $F$ are conjugates of $B_1$ i.e. $B_i = \sigma_i(B_1)$ for some $\sigma_i \in \Gal(K(B_1)/K)$. 
Reorder the $B_i$'s, if necessary, so that $\sigma_i$ corresponds to the automorphism $\alpha\mapsto \alpha^{p^{i-1}}$ in $\Gal(\bF_{p^d}/\bF_p)$.
Hence $B_i =  \sigma_i(\beta_{12}X_2+\cdots + \beta_{1n}X_n) 
        = \sigma_i(\beta_{12})X_2+\cdots + \sigma_i(\beta_{1n})X_n$,
and so
   \eqref{eqn:betaij_b} holds.

Finally we are ready to prove the first claim in the proposition that the matrix in \eqref{eqn:gil_cond} has full rank over $\bF_p$. Suppose instead that the rows are linearly dependent, so there exist $\alpha_1,...,\alpha_d\in \bF_p$ not all zero such that
$\alpha_1 \beta_{1j}+\cdots + \alpha_d \beta_{dj}=0$ for all $2\leq j \leq n$.
By \eqref{eqn:betaij_b}, $\beta_{ij}=\sigma_i(\beta_{1j})$, so 
    \begin{align*}
        \alpha_1 \beta_{1j}+\alpha_2\sigma_2(\beta_{1j})+\cdots + \alpha_d \sigma_d(\beta_{1j})=0, \qquad 2\leq j \leq n.
    \end{align*}
But since $\beta_{12},...,\beta_{1n}$ generate $\bF_{p^d}$, this implies that $\alpha_1 \omega+\alpha_2\sigma_2(\omega)+\cdots + \alpha_d \sigma_d(\omega)=0$ for all $\omega \in \bF_{p^d}$. This contradicts the existence of a so-called normal element (or a normal basis) in $\bF_{p^d}$ (see e.g. \cite[Theorem 2.35]{LN97}), where recall an element $\omega\in \bF_{p^d}$ is normal if $\omega,\omega^p,...,\omega^{p^{d-1}}$ are linearly independent over $\bF_p$.

We end by producing the desired change of variables.
By \eqref{eqn:F_linear_factors}, we have
    \begin{align*}
         F
         =\prod_{i=1}^d (X_1+\sum_{j=2}^n \beta_{ij}X_j)
         =\prod_{i=1}^d (X_1+\sum_{j=2}^n (\sum_{k=1}^d b_{jk}\omega^{p^{i-1}(k-1)} )X_j)
         =\prod_{i=1}^d (X_1+ \sum_{k=1}^d (\sum_{j=2}^n b_{jk}X_j)\omega^{p^{i-1}(k-1)} ).
    \end{align*}
    Let $\ubf=(u_1,...,u_d)$, $\Xbf=(X_1,...,X_n)$, and consider the change of variables $\ubf =  U \Xbf$,  where
    \begin{align*}
        U := \begin{bmatrix}
                1&b_{21}&\cdots &b_{n1}\\
                0&b_{22}&\cdots&b_{n2}\\
                \vdots &\vdots&&\vdots\\
                0&b_{2d}&\cdots&b_{nd}\\
            \end{bmatrix}_{d\times n}.
    \end{align*}
Then $F= F(X_1,...,X_n)=  N_{\bF_{p^d}/\bF_p}(u_1+u_2\omega + \cdots + u_d \omega^{d-1}).$
To see that the matrix $U$ has full rank, let 
    \begin{align}\label{eqn:def_VW}
V:=    \begin{bmatrix}
            1&\beta_{12}&\cdots &\beta_{1n}\\
            \vdots&\vdots&&\vdots\\
             1&\beta_{d2}&\cdots &\beta_{dn}
        \end{bmatrix}_{d\times n} \quad
        \text{and} \quad
        W:= \begin{bmatrix}
            1&\omega&\cdots &\omega^{d-1}\\
            1&\omega^{p}&\cdots&\omega^{p(d-1)}\\
            \vdots&\vdots&&\vdots\\
            1&\omega^{p^{d-1}}&\cdots &\omega^{p^{d-1}(d-1)}
        \end{bmatrix}_{d\times d}.
    \end{align}
    Then by \eqref{eqn:betaij_b}, 
    we have the relation $ V = WU.$
For two matrices $M_1,M_2$, we have $\rank (M_1M_2) \leq \min\{\rank M_1, \rank M_2\}$. The matrix $W$ has full rank (see e.g. \cite[Corollary 2.38]{LN97}) over $\overline{\bF}_p$. (This is due to the standard fact that a subset of elements $\{\al_1,...,\al_d\} \subseteq \bF_{p^d}$ is a basis if and only if its discriminant is nonzero; in the case of $\{1,\omega,...,\omega^{d-1}\}$ which we know is a basis, its discriminant is precisely the determinant of $W$.) 
The matrix $V$ has full rank over $\bF_p$ as shown above. So $U$ must have full rank (over $\overline{\bF}_p$ and hence over $\bF_p$).
\end{proof}

\begin{proof}[Proof of Proposition \ref{prop:gil_main0}] 
Write the factorization of $F$ over $\overline{\bF}_p$ as $ F = \prod_{i=1}^k (\beta_{i1}X_1+\beta_{i2}X_2+\cdots + \beta_{in}X_n)$. 
By hypothesis,
the $k\times n$ matrix 
    \begin{align}\label{eqn:matrix_betas_proof}
         \begin{bmatrix}
            \beta_{11}&\beta_{12}&\cdots &\beta_{1n}\\
            \vdots&\vdots&&\vdots\\
            \beta_{k1}&\beta_{k2}&\cdots &\beta_{kn}
        \end{bmatrix}
    \end{align}
has full rank $\min\{k,n\}$ over $\overline{\bF}_p$. 
Note that $\beta_{i1}\neq 0$ for all $i$, since $F$ is homogeneous and in $n$ variables.
Factor $F= c\prod_{i=1}^s F_i(X_1,...,X_n)$ over $\bF_p$, where $c=\prod_{i=1}^k \beta_{i1}$ and each $F_i$ is irreducible over $\bF_p$. Then each $F_i$ factors over $\overline{\bF}_p$ in the form of \eqref{eqn:F_linear_factors}.

Apply Proposition \ref{prop:gil_main2}
    to get, for each $i$,
        \begin{align*}
            F_i(X_1,...,X_n) = N_i(u_{i,1}+u_{i,2}\omega_i+\cdots + u_{i,k_i}\omega_i^{k_i-1})
        \end{align*}
    where $N_i:\bF_{p^{k_i}}\rightarrow \bF_p$ is the norm, $1,\omega_i,...,\omega_i^{k_i-1}$ is a basis of $\bF_{p^{k_i}}$ and $(u_{i,1},...,u_{i,k_i})^T\equiv
                U_i \mathbf{X} \modd p$,
    where $U_i$ is a $k_i\times n$ matrix of full rank. Then 
            \begin{align*}
                A :=
                [U_i]_{1\leq i \leq s}
            \end{align*}
        is a $k\times n$ matrix since $k_1+\cdots+k_s=k$. 
    Now we show $A$ has full rank over $\bF_p$ and the claim \eqref{eqn:rank_Ui} together. 
    For distinct indices $1\leq i_1,...,i_r\leq s$, applying Proposition \ref{prop:gil_main2} to each $F_{i_j}$ produces $V_{i_j}, W_{i_j}$ as defined in \eqref{eqn:def_VW} such that $V_{i_j} = W_{i_j} U_{i_j}$. Then
        \begin{align*}
            \begin{bmatrix}
                V_{i_1}\\
                \vdots\\
                V_{i_r}
            \end{bmatrix}
            \equiv 
            \begin{bmatrix}
                W_{i_1}\\
                &\ddots\\
                && W_{i_s}
            \end{bmatrix}
            \begin{bmatrix}
                U_{i_1}\\
                \vdots\\
                U_{i_r}
            \end{bmatrix}.
        \end{align*}
Again we apply the fact that for two matrices $M_1,M_2$, $\rank M_1M_2 \leq \min\{\rank M_1, \rank M_2\}$. The matrix of $W_i$'s has full rank since it is block diagonal and each block has full rank (from the proof of Proposition \ref{prop:gil_main2}). 
    If $\{i_1,...,i_r\}=\{1,...,s\}$, then by supposition the matrix of $V_i$'s has full rank (over $\overline{\bF}_p$), and so $A$ (equal in this case to the matrix of $U_i$'s) must have full rank (over $\overline{\bF}_p$ and hence over $\bF_p$).
    If, however $\{i_1,...,i_r\}\neq\{1,...,s\}$, then it is not guaranteed that the matrix of $V_i$'s has full rank, unless $n=k$, in which case, we may conclude that the matrix of $U_i$'s has full rank (over $\overline{\bF}_p$ and hence over $\bF_p$).
\end{proof}

\begin{remark}
    If $k\leq n$, then the hypothesis of Proposition \ref{prop:gil_main0} that \eqref{eqn:matrix_betas_proof}
    has full rank over $\overline{\bF}_p$ already guarantees that the matrix \eqref{eqn:gil_cond} associated to each factor $F_i$ has full rank over $\bF_p$ (in fact, over $\overline{\bF}_p$). But this is not obvious when $n<k$, and hence why we prove it in Proposition \ref{prop:gil_main2}.
\end{remark}

\section{Deduction of Theorem \ref{thm:main2} from Theorem \ref{cor:main_energy}}\label{sec:burgess}

\subsection{The initial setup}
Fix a prime $p$ and integers $n,k\geq 1$. Assume the hypotheses of Theorem \ref{thm:main2}, that is, let $F\in \bF_p[X_1,...,X_n]$ be a form of degree $k$ where $ k<2n$, and suppose $F$ splits into linear factors over $\overline{\bF}_p$, 
    i.e. $F=(\underline{\beta^T_1}\Xbf)\cdots (\underline{\beta^T_k}\Xbf)$ with $\underline{\beta_i}\in (\overline{\bF}_p)^n$, and the $k\times n$ matrix with rows $\underline{\beta^T_1},...,\underline{\beta^T_k}$ has full rank $\min\{k,n\}$ over $\overline{\bF}_p$.
Write the factorization of $F$ over $\bF_p$ as $F_1\cdots F_s$, where each $F_i$ is irreducible over $\bF_p$ and $\deg F_i = k_i$, so $k_1+\cdots +k_s=k$. By Proposition \ref{prop:gil_main0}, there exists a change of variables $\ubf=A\mathbf{X}$, where $A$ is a $k\times n$ matrix with full rank $\min\{k,n\}$, such that 
    \begin{align}\label{eqn:apply_gillett}
        F(X_1,...,X_n)=\prod_{i=1}^s F_i(X_1,...,X_n) = \prod_{i=1}^s N_i(u_{i,1} + u_{i,2} \omega_i + \cdots + u_{i,k_i}\omega_i^{k_i-1}),
    \end{align}
where $N_i: \bF_{p^{k_i}} \rightarrow \bF_p$ is the norm and $1,\omega_i,...,\omega_i^{k_i-1}$ is a power basis with $\omega_i$ the root of some irreducible polynomial of degree $k_i$. Let $L_{i,j}(\Xbf): = a_{i,j,1}X_1+\cdots +a_{i,j,n}X_n=u_{i,j}$, so by \eqref{eqn:def_Ui},
    \begin{align*}
        \begin{bmatrix}
            L_{i,1}(\Xbf)\\
            \vdots\\
            L_{i,k_i}(\Xbf)
        \end{bmatrix}
        \equiv U_i\Xbf.
    \end{align*}
To keep track of the original variable $\xbf$ over which we are summing, we will use this linear form $L_{i,j}(\cdot)$ notation instead of $u_{i,j}$. Recall the definition
    \begin{align}\label{eqn:def_lambda_0}
        \lam_i(\cdot) = L_{i,1}(\cdot)+L_{i,2}(\cdot)\omega_i + \cdots + L_{i,k_i}(\cdot) \omega_i^{k_i-1},
    \end{align}
and note that $\lam_i$ is linear, i.e. $\lam_i(\xbf+\ybf) = \lam_i(\xbf) + \lambda_i(\ybf)$ and $\lam_i(t\xbf) = t\lam_i(\xbf)$.
Then by \eqref{eqn:apply_gillett}, multiplicativity of $\chi$, and the identification from \S \ref{sec:lifting}, we have
    \begin{align}\label{eqn:S_lift}
      S(F;\Nbf,\Hbf) &= \sum_{\xbf\in (\Nbf,\Nbf+\Hbf]} \chi(\prod_{i=1}^s N_i( \lam_i(\xbf)))
        = \sum_{\xbf\in (\Nbf,\Nbf+\Hbf]} \prod_{i=1}^s\psi_i( \lam_i(\xbf)),
    \end{align}
where $\psi_i=\chi \circ N_i$ is a multiplicative character of $\bF_{p^{k_i}}$.

\subsection{Reducing to boxes with comparable side lengths}\label{sec:comparable}
We follow an idea of \cite{Kon10} to reduce the initial box (over which we are summing) with uncontrolled side lengths to boxes whose side lengths lie in the same dyadic range.

\begin{lemma}\label{lem:comparable}
    Let $\Hbf=(H_1,...,H_n)$ and fix $H'$ such that $H'\leq H_i$ for all $i$. Then  \begin{align}\label{eqn:comparable_length}
       | S(F;\Nbf,\Hbf)| \leq 
       \|\Hbf\|
       \max_{\substack{\Nbf',\Hbf'\in \bZ^n\\ H'_i\in [H',2H']}} \|\Hbf'\|^{-1} |S(F;\Nbf',\Hbf')|.
    \end{align}
\end{lemma}

\begin{proof}
For each $i$, break the interval $(N_i,N_i+H_i]$ into $\lfloor H_i/H' \rfloor -1$ many sub-intervals of length $H'$ and one sub-interval of length between $H'$ and $2H'$:
    \begin{align*}
        (N_i,N_i+H_i] = \bigcup_{j=1}^{\lfloor H_i/H' \rfloor } I_{i,j},
    \end{align*}
where $I_{i,j} = (N_i+(j-1)H',N_i+jH']$ for $1\leq j \leq \lfloor H_i/H' \rfloor -1$ and $I_{i,\lfloor H_i/H' \rfloor} = (N_i+ (\lfloor \frac{H_i}{H'} \rfloor -1)H',N_i+H_i]$.
Then
    \begin{align}\label{eqn:break_interval_disjoint}
       (\Nbf,\Nbf+\Hbf] 
       =\prod_{i=1}^n (N_i,N_i+H_i]
       =\prod_{i=1}^n \bigcup_{j=1}^{\lfloor H_i/H' \rfloor} I_{i,j}
       =\bigcup_{\substack{j_1,...,j_n\\1\leq j_i\leq \lfloor H_i/H' \rfloor }} \prod_{i=1}^n  I_{i,j_i},
    \end{align}
where the $\prod_{i=1}^n I_{i,j_i}$'s are disjoint and $H'\leq |I_{i,j_i}|<2H'$ for all pairs $(i,j_i)$. Then
    \begin{align*}
        |S(F;\Nbf,\Hbf)| 
        &\leq  \sum_{\substack{j_1,...,j_n\\1\leq j_i\leq \lfloor H_i/H' \rfloor }}  
        (\prod_{i=1}^n|I_{i,j_i}|) (\prod_{i=1}^n|I_{i,j_i}|)^{-1}
       | \sum_{\xbf \in \prod_{i=1}^n  I_{i,j_i}}\chi(F(\xbf))|\\
        &\leq \sum_{\substack{j_1,...,j_n\\1\leq j_i\leq \lfloor H_i/H' \rfloor }}  
        (\prod_{i=1}^n|I_{i,j_i}|)  \max_{\substack{\Nbf',\Hbf'\in \bZ^n\\ H'_i\in [H',2H']}} \|\Hbf'\|^{-1} |S(F;\Nbf',\Hbf')|,
    \end{align*}
and we obtain \eqref{eqn:comparable_length}, upon recalling the decomposition \eqref{eqn:break_interval_disjoint}.
\end{proof}

Henceforth, we assume $\Hbf=(H_1,...,H_n)$ satisfies $H\leq H_i\leq 2H$ for some $H\geq 1$.

\subsection{The Burgess amplification method}\label{sec:pf_burgess_recursive}
Here, we finally carry out method II of the so-called Burgess argument (as discussed in \S\ref{sec:overview}), which is by now standard.
The first step in the amplification process is to shift the box of summation by an amount that is controlled by the original size of the box. This gives us the following recursive formula.
\begin{lemma}\label{lem:burgess_shift}
Fix an integer $H\geq 1$. Let $\Nbf=(N_1,...,N_n)$ and $\Hbf=(H_1,...,H_n)$ such that $H\leq H_i\leq 2H$ for all $i$. 
Let $H'\leq H$ and fix $\abf=(a_1,...,a_n)$ such that $a_i\leq H'$ for all $i$. Then 
    \begin{align*}
        S(F;\Nbf,\Hbf) = S(F;\Nbf+\abf,\Hbf) + O(n  \|\Hbf\|(\frac{H'}{H}) \max_{\substack{\Nbf',\Hbf'\in \bZ^{n}\\ H_i'\in [H', 2H']}} \|\Hbf'\|^{-1}|S(F;\Nbf',\Hbf')|).
    \end{align*}
\end{lemma}

\begin{proof}

Let $E_{\abf}$ denote the remainder from shifting, so by \eqref{eqn:S_lift},
    \begin{align*}
        E_{\abf} =S(F;\Nbf,\Hbf) - S(F;\Nbf+\abf,\Hbf)
        =\sum_{\xbf\in (\Nbf,\Nbf+\Hbf]}  \prod_{i=1}^s \psi_i (\lam_i(\xbf)) 
        -\sum_{\xbf\in (\Nbf+\abf,\Nbf+\abf+\Hbf]}  \prod_{i=1}^s \psi_i (\lam_i(\xbf)).
    \end{align*}
By computing each coordinate separately, we have    
        \begin{multline*}
      E_{\abf}   = \sum_{\ell=1}^n \sum_{\substack{x_1,...,x_{\ell-1}\\x_i\in (N_i, N_i+H_i]}} \sum_{\substack{x_{\ell+1},...,x_n\\x_i\in(N_i+a_i,N_i+a_i+H_i]}} [\sum_{\substack{x_\ell\in (N_\ell, N_\ell+a_\ell]}} \prod_{i=1}^s \psi_i (\lam_i(\xbf))\\
      - \sum_{\substack{x_\ell\in (N_\ell+H_\ell, N_\ell+H_\ell+a_\ell]}} \prod_{i=1}^s \psi_i (\lam_i(\xbf))].
    \end{multline*}
Rewrite the interval of summation for each of the two inner sums as the difference of two intervals of length between $H'$ and $2H'$:
    \begin{align*}
    (N_\ell,N_\ell+a_\ell] &= 
    (N_\ell,N_\ell+a_\ell+H'] \setminus
    (N_\ell+a_\ell,N_\ell+a_\ell+H'],\\
            (N_\ell+H_\ell,N_\ell+H_\ell+a_\ell]
        &=(N_\ell+H_\ell,N_\ell+H_\ell+a_\ell + H']
        \setminus (N_\ell+H_\ell+a_\ell,N_\ell+H_\ell+a_\ell +H'].
    \end{align*}
Apply Lemma \ref{lem:comparable} to reduce $E_\abf$ to character sums over boxes with side lengths comparable to $H'$:
    \begin{align*}
        |E_{\abf} |
        &\leq 8 \sum_{\ell=1}^n H_1\cdots H_{\ell-1}H' H_{\ell+1}\cdots H_n 
        \max_{\substack{\Nbf',\Hbf'\in \bZ^n\\ H_i'\in [H', 2H']}}  \|\Hbf'\|^{-1} |S(F;\Nbf',\Hbf')|\\
        &\leq 8n \|\Hbf\| (\frac{H'}{H})\max_{\substack{\Nbf',\Hbf'\in \bZ^n\\ H_i'\in [H', 2H']}} 
        \|\Hbf'\|^{-1} |S(F;\Nbf',\Hbf')|,
    \end{align*}
which proves the lemma.
\end{proof}

To exploit the advantages of the above shift (as well as the multiplicativity of the $\psi_i$'s and linearity of the $\lam_i$'s), we shift by a product of the form $\abf = t\ybf$, where $t\in (0,T]$ for some fixed $T>0$ that we choose later and $\ybf=(y_1,...,y_n)$ with $y_i\in (0,K_i]$ for all $i$ such that $K_i\in [K,2K]$ for some $K\geq 1/2$ that we choose later.
Suppose $TK\leq H/2$, so that $ty_i \leq 2TK\leq H$ for all $i$. Then applying Lemma \ref{lem:burgess_shift} with $H'=2TK$ gives
    \begin{align*}
        S(F;\Nbf,\Hbf) = S(F;\Nbf+t\ybf,\Hbf) + E_{t\ybf},
    \end{align*}
where
   \begin{align*}
        S(F;\Nbf+t\ybf,\Hbf)=\sum_{\xbf\in (\Nbf,\Nbf+\Hbf]} \prod_{i=1}^s \psi_i (\lam_i(\xbf + t\ybf))=\sum_{\xbf\in (\Nbf,\Nbf+\Hbf]} \prod_{i=1}^s \psi_i (\lam_i(\xbf) + t\lam_i(\ybf))
    \end{align*}
    by linearity of $\lam_i$,
and
    \begin{align*}
        E_{t\ybf} \ll n \|\Hbf\| (\frac{H'}{H})  \max_{\substack{\Nbf',\Hbf'\in \bZ^n\\ H_i'\in [H',2H']}} \|\Hbf'\|^{-1} |S(F;\Nbf',\Hbf')|.
    \end{align*}
Average over $t\in (0,T]$ and $\ybf\in (\mathbf{0},\Kbf]$ so that 
    \begin{align*}
        S(F;\Nbf,\Hbf) 
        =M+E,
    \end{align*}
where
\begin{align}\label{eqn:burgess_E}
    |E| \leq   T^{-1}\|\Kbf\|^{-1}\sum_{t\in (0,T]}\sum_{\ybf\in (\mathbf{0},\Kbf]} |E_{t\ybf}|
   \ll n \|\Hbf\| (\frac{H'}{H})  \max_{\substack{\Nbf',\Hbf'\in \bZ^n\\ H_i'\in [H',2H']}} \|\Hbf'\|^{-1} |S(F;\Nbf',\Hbf')|
\end{align}    
and
    \begin{align*}
        M= T^{-1}\|\Kbf\|^{-1}\sum_{t\in (0,T]}\sum_{\ybf\in (\mathbf{0},\Kbf]}\sum_{\xbf\in (\Nbf,\Nbf+\Hbf]}\prod_{i=1}^s \psi_i (\lambda_i(\xbf) +t\lambda_i(\ybf) ).
    \end{align*}

Let us separate from the sum the $\ybf$'s for which there exists $1\leq i\leq s$ such that $\lambda_i(\ybf)=0$ in $\bF_{p^{k_i}}$.
Recalling \eqref{eqn:def_lambda_0}, this is equivalent to $L_{i,j}(\ybf)\equiv 0 \modd p$ for all $1\leq j \leq k_i$ and hence further equivalent to $\ybf\in \ker U_i$, where recall the definition of the $k_i\times n$ matrix $U_i$ in \eqref{eqn:def_Ui}. 
        Write
        \begin{align*}
         M &=M' + E',
        \end{align*}
    where
        \begin{align*}
            E'=T^{-1}\|\Kbf\|^{-1}\sum_{t\in (0,T]}\ \sideset{}{^*}\sum_{\ybf\in (\mathbf{0},\Kbf]}\sum_{\xbf\in (\Nbf,\Nbf+\Hbf]}\prod_{i=1}^s \psi_i (\lambda_i(\xbf) +t\lambda_i(\ybf) ),
        \end{align*}
    and the starred sum runs over $\ybf$'s for which there exists $1\leq i\leq s$ such that $U_i \ybf \equiv \mathbf{0}$. Since $U_i$ has full rank $\min\{k_i,n\}$ by Proposition \ref{prop:gil_main0}, its kernel has dimension $m_i:=\max\{n-k_i,0\}$, depending on whether $n\leq k_i$ or $n>k_i$. 
    So there are $\ll_{n} K^{m_i}$ many choices of $\ybf$ (we prove this fact rigorously in \S \ref{sec:energy_prelim}),
    and hence
            \begin{align}\label{eqn:burgess_E'}
            E'\ll_{n,k} K^{-n}\|\Hbf\| \max_{1\leq i \leq s} K^{m_i} 
            \ll_{n,k} \|\Hbf\| K^{-1}.
        \end{align}
Similarly, we separate the $\xbf$'s for which there exists $1\leq i\leq s$ such that $ \lambda_i(\xbf)=0.$
    Write
        \begin{align*}
         M' &=M'' + E'',
        \end{align*}
    where
        \begin{align*}
        E''= T^{-1}\|\Kbf\|^{-1}\sum_{t\in (0,T]}\sum_{\ybf\in (\mathbf{0},\Kbf]}\ \sideset{}{^*}\sum_{\xbf\in (\Nbf,\Nbf+\Hbf]}\prod_{i=1}^s \psi_i (\lambda_i(\xbf) +t\lambda_i(\ybf) ),
        \end{align*}
    and the starred sum runs over $\xbf$'s for which there exists $1\leq i \leq s$ such that $U_i \xbf \equiv \mathbf{0}$. Again, the kernel of $U_i$ has dimension $m_i$, and hence
        \begin{align*}
            E''\ll_{n,k}  \max_{1\leq i \leq s} H^{m_i}  
            \ll_{n,k}  \|\Hbf\|H^{-1}.
        \end{align*}
Since $TK\leq H/2$, and in particular, $K\leq H$, this bound on $E''$ is subsumed by the bound on $E'$ in \eqref{eqn:burgess_E'}. 

Combining the estimates so far, we have 
        \begin{align}\label{eqn:SMEE}
             S(F;\Nbf,\Hbf) = M''  +E'+E,
        \end{align}
where, by \eqref{eqn:burgess_E} and \eqref{eqn:burgess_E'},
    \begin{align*}
   E\ll n \|\Hbf\| (\frac{H'}{H}) \max_{\substack{\Nbf,\Hbf'\\ H_i'\in [H',2H']}} \|\Hbf'\|^{-1} |S(F;\Nbf,\Hbf')|
   \quad \text{and} \quad    E'\ll_{n,k} \|\Hbf\| K^{-1};
    \end{align*}
here, we make explicit the dependence of the implied constant on $n$ for $E$ but not $E'$, because we will later (in \S\ref{sec:recursive_kon}) iterate the bound on $E$. The main term of \eqref{eqn:SMEE} is
    \begin{align*}
        M'' 
        &=T^{-1}\|\Kbf\|^{-1}\sum_{t\in (0,T]}\ \sideset{}{^\dagger}\sum_{\ybf\in (\mathbf{0},\Kbf]}\ \sideset{}{^\dagger}\sum_{\xbf\in (\Nbf,\Nbf+\Hbf]} \prod_{i=1}^s \psi_i (\lambda_i(\xbf)+t\lambda_i(\ybf) ),
    \end{align*}
where the first daggered sum runs over $\ybf$'s such that $\lambda_j(\ybf)\neq 0$ for all $i$, and similarly for the second daggered sum over $\xbf$. 
By swapping the sums and then applying the multiplicativity of $\psi_i$,
    \begin{align*}
         |M''|
        \leq T^{-1}\|\Kbf\|^{-1}\ \sideset{}{^\dagger}\sum_{\ybf\in (\mathbf{0},\Kbf]}\ \sideset{}{^\dagger}\sum_{\xbf\in (\Nbf,\Nbf+\Hbf]}| \sum_{t\in (0,T]} \prod_{i=1}^s \psi_i (t+\lambda_i(\xbf) \lambda_i(\ybf)^{-1})|,
    \end{align*}
where $\lambda_i(\ybf)^{-1}$ is the multiplicative inverse of $\lam_i(\ybf)$ in $\bF_{p^{k_i}}$, which exists since $\lam_i(\ybf)\neq 0$.
To capture the redundancies of the term $\lambda_i(\xbf) \lambda_i(\ybf)^{-1}$ and to remove the dependence on $\xbf,\ybf$ from inner sum, we introduce new variables $z_1,...,z_s$, and rewrite the above bound as
    \begin{align*}
        |M''|
        \leq T^{-1}\|\Kbf\|^{-1}\sum_{\substack{z_1,...,z_s\\z_i\in \bF_{p^{k_i}}\setminus \{0\}}}\eta(z_1,...,z_s)|\sum_{t\in (0,T]} \prod_{i=1}^s \psi_i (t+z_i)|,
    \end{align*}
where
    \begin{align}\label{eqn:def_eta}
        \eta(z_1,...,z_s):
        =|\{\xbf\in (\Nbf,\Nbf+\Hbf],\ybf\in (\mathbf{0},\Kbf]: z_i = 
        \lambda_i(\xbf)\lambda_i(\ybf)^{-1}
        \in \bF_{p^{k_i}}, \lambda_i(\xbf),\lambda_i(\ybf)\neq 0 \text{ for all } i\}|,
    \end{align}
    and $z_i$ is well-defined since $\lambda_i(\ybf)\neq 0$ and nonzero since $\lambda_i(\xbf)\neq 0$. 
This step extracts information relating to the form $F$ (encoded in the $\lam_i$'s) from the inner character sum to the redundancy count $\eta$. To separate these two different types of estimates, apply H\"{o}lder's inequality twice to get
    \begin{multline}\label{eqn:burgess_M''}
        |M'' |
        \leq T^{-1}\|\Kbf\|^{-1}(\sum_{\substack{z_1,...,z_s\\z_i\in \bF_{p^{k_i}}\setminus \{0\}}}\eta(z_1,...,z_s))^{1-\frac{1}{r}}\\
        \cdot (\sum_{\substack{z_1,...,z_s\\z_i\in \bF_{p^{k_i}}\setminus \{0\}}}\eta(z_1,...,z_s)^{2})^{\frac{1}{2r}}(\sum_{\substack{z_1,...,z_s\\z_i\in \bF_{p^{k_i}}\setminus \{0\}}}|\sum_{t\in (0,T]} \prod_{i=1}^s \psi_i (t+z_i)|^{2r})^{\frac{1}{2r}}.
    \end{multline}
The first sum is trivially bounded by $\|\Hbf\| K^n$, so it remains to bound the last two sums, which we denote by $S_1$ and $S_2$. We will first bound $S_1$, conditional on the truth of Theorem \ref{cor:main_energy}, and then bound $S_2$, which is unconditional.

\subsubsection{Bounding the redundancies $S_1$ conditional on Theorem \ref{cor:main_energy}}\label{sec:burgess_redundancies}
Rewrite $S_1$ as
    \begin{align*}
    S_1=|\{(\xbf,\xbf',\ybf,\ybf')\in (\Nbf,\Nbf+\Hbf]^2\times (\mathbf{0},\Kbf]^2: \lambda_i(\xbf)\lambda_i(\ybf')=\lambda_i(\xbf')\lambda_i(\ybf),
    \lambda_i(\xbf)\lambda_i(\ybf') \neq 0 \text{ for all } i\}|.
    \end{align*}
To apply Theorem \ref{cor:main_energy}, we need the coordinates $\xbf,\xbf',\ybf,\ybf'$ to lie in the same box. 
We adjust the current asymmetry by first rewriting the equations $\lambda_i(\xbf)\lambda_i(\ybf')=\lambda_i(\xbf')\lambda_i(\ybf)$ as $\lambda_i(\xbf) \lam_i(\xbf')^{-1} =\lambda_i(\ybf)\lambda_i(\ybf')^{-1}$, which is allowed since $\lam_i(\xbf'), \lam_i(\ybf')\neq 0$.
Hence
    \begin{multline*}
    S_1=\sum_{\substack{z_1,...,z_s\\z_i\in \bF_{p^{k_i}}\setminus\{0\}}} |\{(\xbf,\xbf')\in (\Nbf,\Nbf+\Hbf]^2: \lambda_i(\xbf)= z_i \lambda_i(\xbf'),\ \lambda_i(\xbf) \lambda_i(\xbf')\neq 0\}| \\
        \cdot |\{(\ybf,\ybf')\in (\mathbf{0},\Kbf]^2: \lambda_i(\ybf)= z_i \lambda_i(\ybf'),\ \lambda_i(\ybf) \lambda_i(\ybf')\neq 0\}|.
    \end{multline*}
Then apply Cauchy-Schwarz to get
    \begin{align*}
    S_1 \leq E_{n,k}((\Nbf,\Nbf+\Hbf])^{1/2} E_{n,k}((\mathbf{0},\Kbf])^{1/2}.
    \end{align*}
Assuming the truth of Theorem \ref{cor:main_energy} and imposing the condition $H,K\leq p^{1/2}$, we obtain
    \begin{align}\label{eqn:burgess_S1}
        S_1=\sum_{\substack{z_1,...,z_s\\z_i\in \bF_{p^{k_i}}\setminus \{0\}}}\eta(z_1,...,z_s)^{2} 
        \ll_{n,k,\varepsilon} H^{k'} K^{k'} p^\varepsilon
    \end{align}
for all $\varepsilon>0$, recalling that $k'=\max\{k,2n-k\}$.

\subsubsection{Bounding the $2r$-th moment $S_2$}\label{sec:burgess_2r}
We record here Theorem 2C' of \cite{Sch76} (also stated as Theorem 11.23 of \cite{IK04}) which is a version of the Weil bound on complete multiplicative character sums over finite fields.

\begin{thm}\label{thm:weil1} Let $\bF_q$ be a finite field with $q$ elements. Let $\psi$ be a nonprincipal multiplicative character of $\bF_q$ of order $d$. Suppose $f\in \bF_q[X]$ has $m$ distinct roots and is not a $d$-th power. Then
    \begin{align*}
        |\sum_{x\in \bF_{q}} \psi(f(x))| \leq (m-1)q^{1/2}.
    \end{align*}
\end{thm}

Bound $S_2$ by
\begin{align*}
    S_2 &\leq\sum_{\substack{z_1,...,z_s\\z_i\in \bF_{p^{k_i}}}}|\sum_{t\in (0,T]} \prod_{i=1}^s \psi_i (t+z_i)|^{2r}
    =\sum_{\substack{t_1,...,t_{2r}\\t_i\in (0,T]}} \sum_{\substack{z_1,...,z_s\\z_i\in \bF_{p^{k_i}}}} [\prod_{j=1}^r \prod_{i=1}^s \psi_i(t_j+z_i)][\prod_{j=r+1}^{2r} \prod_{i=1}^s \psi_i(t_j+z_i)^{d_i-1}],
\end{align*}
where $d_i$ is the order of $\psi_i$. After further rearranging, we see that
    \begin{align*}
        S_2 \leq \sum_{\substack{t_1,...,t_{2r}\\t_i\in (0,T]}} \prod_{i=1}^s \sum_{z\in \bF_{p^{k_i}}}  \psi_i(\prod_{j=1}^r(z+t_j)(z+t_{r+j})^{d_i-1}).
    \end{align*}
For each $1\leq i \leq s$ and $\tbf \in (0,T]^{2r}$, define $h_i(X;\tbf)=\prod_{j=1}^r(X+t_j)(X+t_{r+j})^{d_i-1}\in \bZ[X]$. If $h_i(X;\tbf)$ is a perfect $d_i$-th power in $\bF_{p^{k_i}}[X]$, then we apply the trivial bound
    \begin{align*}
        |\sum_{z\in \bF_{p^{k_i}}}  \psi_i(h_i(z;\tbf))| \leq p^{k_i}.
    \end{align*}
Otherwise, we apply the Weil bound in Theorem \ref{thm:weil1} to get
     \begin{align*}
        |\sum_{z\in \bF_{p^{k_i}}}  \psi_i(h_i(z;\tbf))| \leq 2r p^{k_i/2}.
    \end{align*}
Since $\tbf \in (0,T]^{2r}$,
$h_i(X;\tbf)$ is a perfect $d_i$-th power in $\bF_{p^{k_i}}[X]$ if and only if it is a perfect $d_i$-th power in $\bF_p[X]$.
Furthermore, if $h_i(z;\tbf)$ is a perfect $d_i$-th power in $\bF_p[X]$, then $h_j(z;\tbf)$ is a perfect $d_j$-th power for every $j$. In other words, either $h_i(z;\tbf)$ is a $d_i$-th power for all $i$, or $h_i(z;\tbf)$ is not a $d_i$-th power for all $i$. So it suffices to count the number of $\tbf$ for which the polynomial
    \begin{align*}
        h(X;\tbf) = (X +t_1)\cdots (X+t_r)(X+t_{r+1})^{d-1}\cdots (X+t_{2r})^{d-1}
    \end{align*}
is a $d$-th power over $\bF_p$. We will count a larger set i.e. the set of $\tbf\in (0,T]^{2r}$ for which every root $-t_j$ of $h(X;\tbf)$ in $\bF_{p^{k_i}}$ appears at least twice. Let $\mathrm{Bad}(T)$ denote this set of $\tbf$, and let $\mathrm{Good}(T)$ denote the complement. Then
    \begin{align*}
        |S_2| \leq \sum_{\tbf\in \mathrm{Good}(T)} 2r p^{k/2} +  \sum_{\tbf\in \mathrm{Bad}(T)}p^k
        \ll_r |\mathrm{Good}(T)| p^{k/2} + |\mathrm{Bad}(T)| p^k.
    \end{align*}
We bound $|\mathrm{Good}(T)|$ trivially by $T^{2r}$. To estimate $|\mathrm{Bad}(T)|$, we follow an argument in \cite{Bur62a}.
Let $\tbf\in \mathrm{Bad}(T)$ and let $m$ denote the number of distinct roots of $h(X;\tbf)$. Then $1\leq m \leq r$ by definition of $\mathrm{Bad}(T)$. With $m$ fixed, let $x_1,...,x_m$ denote the distinct roots.
For $1\leq j \leq m$, let $i_j$ denote the smallest index $i$ such that $t_i=x_j$; suppose without loss of generality that $i_1<\cdots < i_m$.
Note that $i_1=1$ and $i_m\leq 2r-1$. The number of choices of the positions of these indices is $\binom{2r-2}{m-1} \leq 2^{2r-2}$. 
With a position $i_1,...,i_m$ fixed, there are $T^m$ choices for $t_{i_1},...,t_{i_m}$.
With these $m$ roots fixed, the remaining roots can only take on any of these $m$ values, so there are $m^{2r-m}$ choices. Hence $|\mathrm{Bad}(T)|$ is bounded above by
    \begin{align*}
        \sum_{m=1}^r 2^{2r-2}T^m m^{2r-m} \leq 2^{2r-2}T^r r^{2r} r \ll_r T^r,
    \end{align*}
and so $S_2 \ll_r T^{2r} p^{k/2} + T^r p^k$. Upon balancing these two terms, we learn that $T=p^{k/2r}$. Thus,
    \begin{align}\label{eqn:burgess_S2}
        S_2\leq \sum_{\substack{z_1,...,z_s\\z_i\in \bF_{p^{k_i}}}}|\sum_{t\in (0,T]} \prod_{i=1}^s \psi_i (t+z_i)|^{2r} \ll_r p^{3k/2}.
    \end{align}

\subsubsection{Combining estimates}

By \eqref{eqn:burgess_M''}, \eqref{eqn:burgess_S1}, \eqref{eqn:burgess_S2}, and the dyadic variation of $\Kbf$, we obtain
\begin{align*}
        M'' \ll_{n,k,r,\varepsilon} T^{-1}K^{-n} (\|\Hbf\| K^n )^{1-\frac{1}{r}}(H^{k'} K^{k'}p^\varepsilon)^{\frac{1}{2r}}p^{\frac{3k}{4r}} 
        \ll_{n,k,r,\varepsilon} \|\Hbf\| (HK)^{-\frac{(2n-k')}{2r}}p^{\frac{k}{4r} + \frac{\varepsilon}{2r}}.
    \end{align*}
We remark that the conditions $T=p^{k/2r}$ (from \S \ref{sec:burgess_2r}), $TK\leq H/2$, and $K\geq 1/2$ imply $H\geq p^{k/2r}$. The energy bounds (from \S \ref{sec:burgess_redundancies}) require $H\leq p^{1/2}$ so for the range of allowable $H$ to be nonempty, we impose that $r\geq k$.
By \eqref{eqn:SMEE} with $H'=2TK$ and $T=p^{k/2r}$, we have proved the following recursive upper bound on the character sum, assuming Theorem \ref{cor:main_energy}.

  \begin{prop}\label{prop:burgess_recursive}
Fix integers $n,k, r\geq 1$ with $k< 2n$ and $r\geq k$. Let $H$ be an integer such that $p^{k/2r}\leq H\leq p^{1/2}$, and let $\Hbf=(H_1,...,H_n)$ with $H\leq H_i\leq 2H$ for all $i$. Let $H'$ be an integer such that $p^{k/2r}\leq H'\leq H$. 
    Then for all  $\varepsilon>0$,
    \begin{multline*}
        |S(F;\Nbf,\Hbf)|
         \leq C_1  \|\Hbf\| (H H'p^{-\frac{k}{2r}})^{-\frac{(2n-k')}{2r}} p^{\frac{k}{4r}+\frac{\varepsilon}{2r}} 
         +C_2  \|\Hbf\|(H'p^{-\frac{k}{2r}})^{-1}\\
         +C_3 n \|\Hbf\| (\frac{H'}{H}) \max_{\substack{\Nbf',\Hbf'\in \bZ^n\\ H'_i\in [H',2H']}} \|\Hbf'\|^{-1}|S(F;\Nbf',\Hbf')|,
    \end{multline*}
where $C_1=C_1(n,k,r,\varepsilon), C_2=C_2(n,k)$ and $C_3$ is an absolute constant.
\end{prop}

We remark that the factor of $n$ in the last term is written out explicitly because it contributes growth after iterating this bound many times, so we keep track of it carefully. For this reason, it is crucial that $C_3$ does not depend on $n$.

\begin{remark}
Before proceeding, we note that if we bound the max term in the proposition trivially, then we can a priori obtain a nontrivial bound on $|S(F;\Nbf,\Hbf)|$. Indeed, this gives
    \begin{align*}
         S(F;\Nbf,\Hbf)
         \ll_{n,k,r,\varepsilon}  H^n (H H'p^{-\frac{k}{2r}})^{-\frac{(2n-k')}{2r}} p^{\frac{k}{4r}+\frac{\varepsilon}{2r}} 
         + H^n(H'p^{-\frac{k}{2r}})^{-1}
         + H^{n-1}H'.
    \end{align*}
Some computation shows that the right-hand side is minimized when $H'$ is chosen to be 
        \begin{align*}
            H'=H^{1-\frac{2(2n-k')}{2r+2n-k'}} p^{\frac{k(r+2n-k')}{2r(2r+2n-k')}+\varepsilon},
        \end{align*}
        as long as $r\geq 3(2n-k')/2$.
    With this choice, for $\Hbf=(H,...,H)$,
        \begin{align*}
           S(F;\Nbf,\Hbf)
           \ll_{n,k,r,\varepsilon} H^{n-\frac{2(2n-k')}{2r+2n-k'}} p^{\frac{k(r+2n-k')}{2r(2r+2n-k')}+\varepsilon},
        \end{align*}
        which gives a nontrivial bound as long as $H\geq p^{\frac{k}{4(2n-k')}+\kappa}$ with  $\kappa>0$.
   Notice this produces the same threshold as Theorem \ref{thm:main2}; however, the upper bound here on $|S(F;\Nbf,\Hbf)|$ is worse for all $H\geq p^{\frac{k}{4(2n-k')}+\kappa}$. In particular, when $n=1$ (and $k=1$), the bound is $H^{1-1/(r+1/2)}p^{(r+1)/(4r^2+2r)}$, which is worse than the classic Burgess bound \eqref{eqn:burgess} and falls short of proving the subconvexity result of $L(1/2+it,\chi)\ll_{t,\varepsilon} p^{3/16+\varepsilon}$. For this reason, we are interested in proving a bound that not only recovers the Burgess threshold of $p^{1/4+\kappa}$, but also the Burgess bound of \eqref{eqn:burgess}, when $n=1$. This is achieved by Theorem \ref{thm:main2}, which we deduce by iterating Proposition \ref{prop:burgess_recursive}.

   We further mention other techniques in literature, deriving from alternative proofs of the classic Burgess bound, that handle the ``error'' from shifting. In \cite[Chapter 12]{IK04}, the classic bound is achieved by induction on $H$; this is challenging in our setting since we would a priori need to know the shape of the main bound, whose dependence on $n,k,r$ is difficult to predict. Alternatively, Friedlander and Iwaniec \cite{FI93} followed a different approach that utilized a weight function and applied Fourier analysis.
\end{remark}

\subsection{Deduction of Theorem \ref{thm:main2} from Proposition \ref{prop:burgess_recursive}}\label{sec:recursive_kon}

Fix $0<\alpha\leq 1/2$ that we will choose later (to depend on $C_3, n,k,r$); in particular, assume $\alpha \asymp_{n,k,r} 1$. For every integer $\ell\geq 0$, define
    \begin{align*}
        H^{(\ell)} = \lfloor \alpha^\ell H \rfloor,
    \end{align*}
so $H^{(0)}=H$. 
For each $\ell$, further denote $\Hbf^{(\ell)}=(H_1^{(\ell)},...,H_n^{(\ell)})$, where $H^{(\ell)}\leq H_i^{(\ell)}\leq 2H^{(\ell)}$. Let $L$ be such that $H^{(L)}\geq p^{k/2r}> H^{(L+1)}$; we will iterate until we reach this threshold. Note that this choice of $L$ implies 
    \begin{align}\label{eqn:cond_L}
       \frac{\log H - \log(p^{k/2r}+1)}{\log \alpha^{-1}} -1 < L \leq \frac{\log H - \log(p^{k/2r})}{\log \alpha^{-1}}.
    \end{align}
For all $0\leq \ell \leq L-1$, apply the proposition to $\Hbf^{(\ell)}$ to get
    \begin{multline*}
        | S(F;\Nbf,\Hbf^{(\ell)})| \leq C_1 \|\Hbf^{(\ell)}\| ( H^{(\ell)}H^{(\ell+1)}p^{-\frac{k}{2r}})^{-\frac{(2n-k')}{2r}} p^{\frac{k}{4r}+\frac{\varepsilon}{2r}} 
         +C_2 \|\Hbf^{(\ell)}\| (H^{(\ell+1)}p^{-\frac{k}{2r}})^{-1}
        \\
        +C_3 n \|\Hbf^{(\ell)}\| (\frac{H^{(\ell+1)}}{H^{(\ell)}}) \max_{\substack{\Nbf',\Hbf^{(\ell+1)}\in \bZ^n\\ H_i^{(\ell+1)}\in [H^{(\ell+1)}, 2H^{(\ell+1)}]}} \|\Hbf^{(\ell+1)}\|^{-1}|S(F;\Nbf',\Hbf^{(\ell+1)})|,
    \end{multline*}
so that after iteration,
    \begin{multline*}
        |S(F;\Nbf,\Hbf)| \leq C_1 \|\Hbf\| p^{\frac{k}{4r}+\frac{\varepsilon}{2r}} \sum_{\ell=0}^{L-1}(C_3n)^\ell(\frac{H^{(\ell)}}{H})( H^{(\ell)}H^{(\ell+1)}p^{-\frac{k}{2r}})^{-\frac{(2n-k')}{2r}}  \\
        + C_2 \|\Hbf\| \sum_{\ell=0}^{L-1}(C_3n)^\ell (\frac{H^{(\ell)}}{H})(H^{(\ell+1)}p^{-\frac{k}{2r}})^{-1}\\
        + (C_3n)^L \|\Hbf\| (\frac{H^{(L)}}{H})  \max_{\substack{\Nbf',\Hbf^{(L)}\in \bZ^n\\ H_i^{(L)}\in [H^{(L)}, 2H^{(L)}]}} \|\Hbf^{(L)}\|^{-1}|S(F;\Nbf',\Hbf^{(L)})|.
    \end{multline*}
    Denote the three terms on the right-hand side respectively by  $M_1,M_2, M_3$.
To bound $M_1$, note first that $\lfloor x \rfloor \geq x/2$ as long as $x\geq 1$. Since by definition $\alpha^{\ell}H \geq H^{(\ell)}\geq p^{k/2r}\geq 1$ for all $\ell \leq L$,
 we have $H^{(\ell)}H^{(\ell+1)} \gg \alpha^{2\ell +1}H^2/4$ for $\ell,\ell+1 \leq L$. Then
    \begin{align*}
    |M_1|    &\leq  C_1 \|\Hbf\| p^{\frac{kr+k(2n-k')}{4r^2}+\frac{\varepsilon}{2r}}\sum_{\ell=0}^{L-1}(C_3n)^\ell\alpha^{\ell}(\frac{\alpha^{2\ell+1} H^2}{4})^{-\frac{(2n-k')}{2r}} \\
        &\leq C_1 \|\Hbf\| H^{-\frac{(2n-k')}{r}} p^{\frac{k(r+2n-k')}{4r^2}+\frac{\varepsilon}{2r}}
        \alpha^{-\frac{(2n-k')}{2r}}
        \sum_{\ell=0}^{L-1}(C_3n)^\ell\alpha^{\ell(1-\frac{2n-k'}{r})}.
    \end{align*}
We bound the sum by the series $\sum_{\ell=0}^\infty (C_3n)^\ell\alpha^{\ell(1-\frac{2n-k'}{r})}$, which converges if $C_3 n \alpha^{1-\frac{(2n-k')}{r}}<1$. 
By imposing $r> k$, we have $r>2n-k'$. Hence this forces the condition
    \begin{align}\label{eqn:cond_alpha}
        \alpha < (C_3n)^{-\frac{r}{r-(2n-k')}}.
    \end{align}
Then (under this assumption on $\alpha$),
    \begin{align}\label{eqn:burgess_M1}
    M_1\ll_{n,k,r,\varepsilon}  H^{n-\frac{(2n-k')}{r}} p^{\frac{k(r+2n-k')}{4r^2}+\frac{\varepsilon}{2r}},
    \end{align}
which is nontrivial when
    \begin{align}\label{eqn:burgess_range}
        H \gg_{n,k,r,\varepsilon}  p^{\frac{k}{4(2n-k')} + \frac{k}{4r}+\varepsilon}.
    \end{align}
For the second term, we have
    \begin{align}\label{eqn:burgess_M2}
        |M_2| \leq C_2 \|\Hbf\| \sum_{\ell=0}^{L-1}(C_3n)^\ell\alpha^{\ell}(\frac{\alpha^{\ell+1}H}{2}p^{-\frac{k}{2r}})^{-1}
        \ll C_2 \|\Hbf\| H^{-1}p^{\frac{k}{2r}}\sum_{\ell=0}^{L-1}(C_3n)^\ell \alpha^{-1}
        \ll_{n,k,r} H^{n-1}p^{\frac{k}{2r}} n^L L,
    \end{align}
    where the last step holds by $\alpha \asymp_{n,k,r} 1$.
We estimate the third term $M_3$ trivially to yield
    \begin{align*}
    |M_3| \leq \|\Hbf\| H^{-1}(C_3n)^L H^{(L)}.
    \end{align*}
By the definition of $L$ and the assumption $\alpha \asymp_{n,k,r} 1$, we have $H^{(L)}\asymp_{n,k,r} p^{k/2r}$, and so the bound for $M_3$ is absorbed by the right-hand side of \eqref{eqn:burgess_M2}. By \eqref{eqn:cond_L},
$ n^L \ll n^{\frac{\log H}{-\log \alpha}}
 \ll H^{\frac{\log n}{-\log \alpha}}$,
so that
    \begin{align*}
        M_2+M_3 \ll_{n,k,r} H^{n-1 +\frac{\log n}{-\log \alpha}} p^{\frac{k}{2r}} \log H.
    \end{align*}
This bound is in turn absorbed by the right-hand side of \eqref{eqn:burgess_M1}, as long as
    \begin{align*}
       H^{n-1 +\frac{\log n}{-\log \alpha}} p^{\frac{k}{2r}} \log H
        \ll_{n,k,r,\varepsilon}  H^{n-\frac{(2n-k')}{r}} p^{\frac{k(r+2n-k')}{4r^2}+\frac{\varepsilon}{2r}},
    \end{align*}
which holds if and only if
    \begin{align}\label{eqn:range_M2}
        p^{\frac{k(r-2n+k')}{4r^2}-\frac{\varepsilon}{2r}} \ll_{n,k,r,\varepsilon}  H^{1-\frac{(2n-k')}{r} + \frac{\log n}{\log \alpha}}.
    \end{align}
Let $\beta = \beta(n,k,r,\alpha)$ denote the exponent of $H$ in \eqref{eqn:range_M2}. A short computation shows that $\beta>0$ when
    $\alpha$ satisfies \eqref{eqn:cond_alpha}.
Hence \eqref{eqn:range_M2} covers the range \eqref{eqn:burgess_range} when
$(r-2n+k')(2n-k') < \beta r(r+2n-k')$,
which holds if and only if
    \begin{align*}
        \alpha < n^{-\frac{r+2n-k'}{r-2n+k'}}.
    \end{align*}
For this condition to be compatible with \eqref{eqn:cond_alpha},
choose $\alpha=\min\{(C_3n)^{-\frac{r}{r-(2n-k')}},n^{-\frac{r+2n-k'}{r-2n+k'
}}\}/2$.
We conclude that for $\Hbf$ with $H\leq H_i\leq 2H$,
    \begin{align*}
        S(F;\Nbf,\Hbf)\ll_{n,k,r,\varepsilon} 
        H^{n-\frac{2n-k'}{r}} p^{\frac{k(r+2n-k')}{4r^2}+\varepsilon},
    \end{align*}
and so for arbitrary $\Hbf$, by \eqref{eqn:comparable_length} we obtain Theorem \ref{thm:main2}.
This concludes the deduction of Theorem \ref{thm:main2} from Proposition \ref{prop:burgess_recursive}, which is conditional on Theorem \ref{cor:main_energy}.

\section{Preliminary energy estimates: proof of Lemma \ref{lem:energy_elementary}}\label{sec:energy_prelim}

We begin with some estimates on bounded solutions to systems of linear equations.
 \begin{lemma}\label{lem:linear_algebra0}
    Fix an $m'\times m$ matrix $M$ mod $p$. Fix an integer $1\leq L \leq p$, and let $\Lbf=(L_1,...,L_m)$ with $L\leq L_i\leq 2L$ for all $i$. Then for a vector $\bbf\in \bZ^{m'}$ and a set $D\subseteq \bZ^m$ of integral tuples, define $S(\bbf;D) := \{\xbf\in D: M\xbf \equiv \bbf \modd p\}$.
    Then for any $\Nbf=(N_1,...,N_m)$,
        \begin{align}\label{eqn:LA01}
         |S(\bbf; [\Nbf,\Nbf+\Lbf])| \leq |S(\mathbf{0}; [-\Lbf,\Lbf])|,    
        \end{align}
    and if $M$ has full rank $\min\{m',m\}$, then
        \begin{align}\label{eqn:LA02}
            |S(\mathbf{0}; [-\Lbf,\Lbf])| \ll_m L^{\dim \ker M}.
        \end{align}
\end{lemma}
\begin{proof} If $S(\bbf; [\Nbf,\Nbf+\Lbf])$ is empty, then \eqref{eqn:LA01} holds. Suppose otherwise, and fix $\ybf\in S(\bbf; [\Nbf,\Nbf+\Lbf])$.
    Define $f:S(\bbf; [\Nbf,\Nbf+\Lbf]) \rightarrow S(\mathbf{0}; [-\Lbf,\Lbf])$ by $f(\xbf) = \xbf - \ybf$. This is well-defined, since $|x_i-y_i|\leq L_i$ for all $i$ and $M(\xbf-\ybf)\equiv \mathbf{0}$. Clearly, $f$ is injective, so \eqref{eqn:LA01} is proved.
    
        We prove \eqref{eqn:LA02} by considering two cases. If $m'>m$, then $\rank M = m$ and $\dim \ker M=0$. Since $L\leq p$, we deduce that $|S(\mathbf{0};[-\Lbf,\Lbf])|\ll_m 1$, which agrees with \eqref{eqn:LA02}. On the other hand, if $m'<m$, then $\rank M=m'$ and $\dim\ker M = m-m'$. Let $\hat{M}$ be an $m'\times m'$ submatrix of $M$ whose determinant is nonzero mod $p$, and write $M\xbf \equiv \hat{M}\hat{\xbf} -\check{M}\check{\xbf}$, where $\hat{\xbf}$ denotes the $m'$ coordinates of $\xbf$ that correspond to the columns of $\hat{M}$, while $\check{\xbf}$ denotes the remaining $m-m'$ coordinates; similarly define $\hat{\Lbf}$ and $\check{\Lbf}$. Then $M\xbf \equiv \mathbf{0}$ if and only if $\hat{M}\hat{\xbf} \equiv\check{M}\check{\xbf}$, and so
        \begin{align*}
            |S(\mathbf{0};[-\Lbf,\Lbf])|
            &=\sum_{\check{\xbf}\in [-\check{\Lbf},\check{\Lbf}]} |\{\hat{\xbf}\in [-\hat{\Lbf},\hat{\Lbf}]:\hat{M}\hat{\xbf} \equiv \check{M}\check{\xbf}\}|
            \leq \sum_{\check{\xbf}\in [-\check{\Lbf},\check{\Lbf}]} |\{\hat{\xbf}\in [-2\hat{\Lbf},2\hat{\Lbf}]:\hat{M}\hat{\xbf} \equiv \mathbf{0}\}|,
        \end{align*}
        where the last inequality holds by applying \eqref{eqn:LA01} to each summand on the left with $\bbf=\bbf(\check{\xbf})=\check{M}\check{\xbf}$.
        Since $\hat{M}$ is nonsingular mod $p$, $\hat{M}\hat{\xbf}\equiv 0$ if and only if $\hat{\xbf}\equiv 0$, and hence each summand on the right is $\ll 1$. 
        Then a trivial count of the $\check{\xbf}$'s gives \eqref{eqn:LA02}.
\end{proof}

Now we prove Lemma \ref{lem:energy_elementary}.
We work in the product ring $\bF_{p^{n_1}}\times \cdots \times \bF_{p^{n_s}}$, where multiplication is pointwise.   
Note that this ring has nontrivial zero divisors (recall an element $\underline{\alpha}=(\al_1,...,\al_s)$ is a zero divisor if there exists a nonzero $\underline{\beta}=(\beta_1,...,\beta_s)$ such that $\underline{\alpha}\underline{\beta}=0$). Hence the proof is slightly more involved than usual (e.g. in $\bZ$ or $\bF_q$).
   The lower bound is established by the diagonal solutions, where $\xbf'=\xbf$ and $\ybf'=\ybf$, which contribute $\gg_n H^{2n}$. 
   
   For the upper bound, first choose $\xbf',\ybf$ freely, contributing $\ll_n H^{2n}$. We consider two cases: $\xbf$ such that $\underline{\lam}(\xbf)=(\lam_1(\xbf),...,\lam_s(\xbf))$ is a zero divisor and $\xbf$ such that $\underline{\lam}(\xbf)$ is not a zero divisor. In the second case, for fixed $\xbf$, we get a unique element $\underline{\lam}(\ybf')=\underline{\lam}(\xbf')\underline{\lam}(\ybf)\underline{\lam}(\xbf)^{-1}$ and hence a unique $\ybf' \modd p$. Since $H\leq p$ and $H_i\leq 2H$, there are $\ll_n 1$ choices of $\ybf'\in (\Nbf,\Nbf+\Hbf]$. We bound the number of $\xbf$'s trivially by $H^n$, so together we have the upper bound $\ll_n H^{3n}$.
    
    Now suppose $\underline{\lam}(\xbf)$ is a zero divisor. Then $\lam_i(\xbf)=0$ for at least one $i$. Fix $1\leq r \leq s$, and suppose without loss of generality that $\lam_i(\xbf)=0$ for $1\leq i \leq r$ and nonzero otherwise. By definition, $\lam_i(\xbf)=0$ if and only if $L_{i,j}(\xbf)\equiv 0$ for all  $1\leq j \leq n_i$. In matrix form, this is
    $\xbf\in \ker U_i$ for $1\leq i \leq r$,
    where $U_i$ is as defined in \eqref{eqn:def_Ui} with $k_i=n_i$, and so
        \begin{align*}
        [U_i]_{1\leq i \leq r}
            \xbf \equiv \mathbf{0}.
        \end{align*}
        By \eqref{eqn:rank_Ui}, this matrix of $U_i$'s has full rank $n_1+\cdots +n_r$. Hence
    by \eqref{eqn:LA02}, the number of choices of $\xbf$ is $\ll_n H^{n-(n_1+\cdots + n_r)}$.
    For $r+1\leq i \leq s$, $\lam_i(\xbf)\neq 0$, so $\lam_i(\ybf')=\lam_i(\xbf')\lam_i(\ybf)\lam_i(\xbf)^{-1}$. With $\xbf',\ybf,$ and now $\xbf$, fixed, $\lam_i(\ybf')$ is uniquely determined for $r+1\leq i\leq s$, so $\ybf'$ satisfies
        \begin{align*}
        [U_i]_{r+1\leq i \leq s}
            \ybf'
            \equiv \bbf,
        \end{align*}
    where $\bbf = \bbf(\xbf,\xbf',\ybf)$.
    Again by \eqref{eqn:rank_Ui}, the matrix of $U_i$'s has full rank $n_{r+1}+\cdots + n_s$, so by \eqref{eqn:LA02}, the number of choices for $\ybf'$ is $\ll_n H^{n-(n_{r+1}+\cdots + n_s)}$. Altogether, we have an upper bound of $\ll_n H^{3n}$, which concludes the proof of Lemma \ref{lem:energy_elementary}.

\begin{remark}
    We will see later in the proof of Theorem \ref{thm:main_energy} a recursive relation that depends on the bounds of two constrained energies $E'(\Hbf,\Kbf)$ and $E''(\Hbf,\Kbf)$. These correspond, essentially, to the two cases we considered above, where $\underline{\lam}(\xbf')$ is not a zero divisor and is a zero divisor, respectively. 
    While the proof of Theorem \ref{thm:main_energy} is considerably more involved, we see traces of its key ingredients in this preliminary demonstration. 
\end{remark}

\section{Results on lattices}\label{sec:lattices}

In this section, we prove results on lattices that prepare us for bounding various shapes of multiplicative energy in \S \ref{sec:energy}.
Fix an integer $n\geq 1$ and a partition $(n_1,...,n_s)$ of $n$.
Fix $\zbf=(z_1,...,z_s)$ with $z_i\in \bF_{p^{n_i}}\setminus \{0\}$. Let $\Hbf=(H_1,...,H_n)$ with $H\leq H_i \leq 2H$ for some fixed $H$, and recall the notation $D_\Hbf:=[-\Hbf,\Hbf]$.
Later in \S \ref{sec:energy}, we will encounter quantities of the form
    \begin{align}\label{eqn:def_eta_2}
        |\{(\xbf,\ybf)\in D_{\Hbf}^2: \lambda_i(\xbf) = z_i \lambda_i'(\ybf) \in \bF_{p^{n_i}},\lambda_i(\xbf)\lambda_i'(\ybf)\neq 0 \text{ for } 1 \leq i\leq s  \}|,
    \end{align}
where $\lam_i(\xbf)=L_{i,1}(\xbf)+L_{i,2}(\xbf)+\cdots + L_{i,n_i}(\xbf)\omega_i^{n_i-1}$ and $\lam_i'(\ybf)=L'_{i,1}(\ybf)+L'_{i,2}(\ybf)+\cdots + L'_{i,n_i}(\ybf)\omega_i^{n_i-1}$, 
and the matrices $A,A'$
defined by
    \begin{align}\label{eqn:A_L}
        A\xbf:=\begin{bmatrix}
            L_{1,1}(\xbf)\\
            \vdots\\
            L_{s,n_s}(\xbf)
        \end{bmatrix}, \qquad
        A'\ybf:=\begin{bmatrix}
            L'_{1,1}(\ybf)\\
            \vdots\\
            L'_{s,n_s}(\ybf)
        \end{bmatrix}
    \end{align}
are nonsingular mod $p$.
We view this as a shifted and asymmetric version of $\eta(z_1,...,z_s)$ defined in \eqref{eqn:def_eta}; it has been shifted so the domain of $\xbf,\ybf$ is now centered at the origin, and it is asymmetric in the sense that the linear forms in the definitions of $\lam_i$ and $\lam'_i$ differ. We now consider another interpretation (originating in \cite{Kon10}) of this quantity. Define the lattice
    \begin{align}\label{eqn:lattice_def}
        \Lcal_\zbf = \{(\xbf,\ybf)\in \bZ^{2n}: \lambda_i(\xbf) = z_i \lambda_i'(\ybf)\in \bF_{p^{n_i}} \text{ for }1\leq i \leq s \}.
    \end{align}
For $\Hbf = (H_1,...,H_n)$ with $H_i>0$, recall we let $  B_{\Hbf} := \{\xbf\in \bR^n: -H_i\leq x_i\leq H_i\}$ denote the real-variable box.
Then \eqref{eqn:def_eta_2} is bounded by $|B_\Hbf^{2}\cap \Lcal_\zbf|$, hence transforming the problem of estimating \eqref{eqn:def_eta_2} to studying the size of the intersection of a convex set with a lattice.
This allows us to access advantageous properties of lattices, which is our focus in the present section (culminating in Propositions \ref{prop:keyfacts} and \ref{prop:keyfact_dual}). In particular, this requires new ideas beyond \cite{BC10}.

\subsection{Preliminary definitions}
We begin with some definitions. Fix an integer $m\geq 1$. A lattice $\Lcal$ in $\bR^{m}$ is a discrete additive subgroup of $\bR^{m}$, and the rank of $\Lcal$ is the dimension of the subspace spanned by the vectors in $\Lcal$. We will only consider lattices of full rank, that is, rank $m$. For a lattice of full rank, there exist linearly independent vectors $\abf_1,...,\abf_{m}\in \bR^m$ such that $\Lcal = \abf_1\bZ + \cdots + \abf_{m}\bZ$; this is then a basis for $\Lcal$. The determinant $\det(\Lcal)$ is the absolute value of the determinant of the matrix formed by the basis vectors; equivalently, it is the volume of the fundamental parallelepiped defined by $\{\abf_1t_1+\cdots + \abf_mt_m: 0\leq t_i\leq 1\}\subseteq \bR^m$.

The dual (or polar) lattice $\Lcal^\ast$ is generated by the dual basis $\abf_1^\ast,...,\abf_{m}^\ast$ defined by $\langle\abf_i^\ast, \abf_j\rangle=1$ if $i=j$ and 0 otherwise. Equivalently,
    \begin{align}\label{eqn:def_dual}
        \Lcal^\ast = \{\ubf\in \bR^{m}: \langle \ubf,\xbf\rangle \in \bZ \text{ for all } \xbf\in \Lcal\}
    \end{align}
by e.g. \cite[\S 1.5, Lemma 5]{Cas97}.
A set $B\subseteq \bR^{m}$ is convex if $(1-t)\xbf+t \ybf\in B$ whenever $\xbf,\ybf\in B$ and $0\leq t \leq 1$; a convex set $B$ is called a convex body if it is compact and has a nonempty interior. 
Finally $B$ is symmetric (about the origin) if $B=-B$. 
(In this paper, we will only consider cases where $B$ is a box of the form $B_\Lbf$, for some $\Lbf$.)

Given a lattice $\Lcal$ and a convex body $B$, we define the successive minima $\lam_1,...,\lam_{m}$ as follows. For each $i$, let $\lam_i$ be the infimum of all $\lam$ for which the dilated body $\lam B$ contains $i$ linearly independent vectors in $\Lcal$. Clearly, $\lam_1\leq \cdots \leq \lam_{m}$. Define $s=s(B,\Lcal)$ by
    \begin{align}\label{eqn:def_s}
        s = \max_{1\leq i \leq m} \{i: \lambda_i\leq 1\}
    \end{align}
and note that $1\leq s \leq m$. In particular, $s$ is the number of linearly independent vectors in $B\cap \Lcal$.

\subsection{Key results}\label{sec:key_results} The first result gives relative bounds on the cardinality of the intersection of a convex set with a lattice; the bounds are relative in the sense that they are in terms of the cardinality of a different convex set or lattice. 

\begin{prop}\label{prop:keyfacts} Fix integers $m, H\geq 1$, and let $\Hbf=(H_1,...,H_m)$ with $H\leq H_i\leq 2H$. Let $\Lcal$ be a lattice in $\bR^{m}$ and $\Lcal^\ast$ its dual. Let $\lambda_1,...,\lam_{m}$ be the successive minima of $B_{\Hbf}$ with respect to $\Lcal$, and let $s$ be as defined in \eqref{eqn:def_s}. Then for $H'\leq H$ and $\Hbf'=(H'_1,...,H'_m)$ with $H'\leq H'_i\leq 2H'$,
\begin{align}\label{eqn:keyfact2}
        |B_{\Hbf} \cap \Lcal| \ll_m (\frac{H}{H'})^{s}  |B_{\Hbf'} \cap \Lcal|.
    \end{align}
 For a fixed real $a>0$, let $H''\leq a/H$ and $\Hbf''=(H_1'',...,H_m'')$ with $H''\leq H_i''\leq 2H''$. Then
    \begin{align}\label{eqn:keyfact25}
         |B_{a/\Hbf} \cap a\Lcal^\ast| \ll_m (\frac{a/H}{H''})^{m-s}  |B_{\Hbf''} \cap a\Lcal^\ast|
    \end{align}
    and
      \begin{align}\label{eqn:keyfact3}
       |B_{\Hbf} \cap \Lcal| \asymp_m \frac{H^{m}}{\det(\Lcal)}|B_{a/\Hbf} \cap a\Lcal^\ast|,
    \end{align}
    where $a/\Hbf = (a/H_1,...,a/H_m)$.
\end{prop}
We prove this in \S \ref{sec:proof_keyfacts} and later
apply it in \S \ref{sec:recursive_formula}. 
The first inequality \eqref{eqn:keyfact2} allows us to reduce to a case with a smaller convex set. To control the bound on the right-hand side (that is, to guarantee that we always achieve square-root cancellation on the number of variables), we distinguish between the case $s$ ``small'' ($s\leq m/2$) and $s$ ``large'' ($m/2<s\leq m$). When $s$ is small, we employ \eqref{eqn:keyfact2} directly. For $s$ large, we will apply \eqref{eqn:keyfact3} followed by \eqref{eqn:keyfact25}, since in this case, the dual lattice will have more advantageous properties. 
For this passage to be profitable,
we require the dual lattice to have the same ``structure" as the original lattice defined in \eqref{eqn:lattice_def}. This difficulty did not appear in \cite{BC10}, where the original lattice \eqref{eqn:lattice_def} consisted of only equations over $\bF_p$, which conveniently implied, without significant work, that the dual lattice shared the same structure as the original. Our new setting consists of general cases that lack this convenient property (which we explain explicitly in \S \ref{sec:proof_dual}), and so we require the next new result.

\begin{prop}\label{prop:keyfact_dual}
Fix a prime $p$ and an integer $n\geq 1$. Let $(n_1,...,n_s)$ be a partition of $n$. Let $\zbf=(z_1,...,z_s)$ where $z_i\in \bF_{p^{n_i}}\setminus \{0\}$, and define the lattice
        \begin{align*}
        \Lcal_\zbf = \{(\xbf,\ybf)\in \bZ^{2n}: \lambda_i(\xbf) = z_i \lambda_i'(\ybf)\in \bF_{p^{n_i}} \text{ for }1\leq i \leq s \},
    \end{align*}
where $\lam_i(\xbf)=L_{i,1}(\xbf)+L_{i,2}(\xbf)\omega_i+\cdots + L_{i,n_i}(\xbf)\omega_i^{n_i-1}$ and $\lam_i'(\ybf)=L'_{i,1}(\ybf)+L'_{i,2}(\ybf)\omega_i+\cdots + L'_{i,n_i}(\ybf)\omega_i^{n_i-1}$, and the matrices $A, A'$, defined respectively by the linear forms $L_{i,j}$ and $L_{i,j}'$ as in \eqref{eqn:A_L}, are nonsingular mod $p$.
    Then the dual lattice (scaled by $p$) can be written in the form
        \begin{align}\label{eqn:dual_lattice_want}
            p\Lcal_\zbf^\ast = \{(\ubf,\vbf)\in \bZ^{2n}: \lam''_i(\vbf) = z_i \lam_i'''(\ubf) \in \bF_{p^{n_i}} \text{ for }1\leq i \leq s\},
        \end{align}
    where $\lam''_i(\vbf) = L''_{i,1}(\vbf)+L''_{i,2}(\vbf)\omega_i+\cdots + L''_{i,n_i}(\vbf)\omega_i^{n_i-1}$ and $\lam'''_i(\ubf) = L'''_{i,1}(\ubf)+L'''_{i,2}(\ubf)\omega_i+\cdots + L'''_{i,n_i}(\ubf)\omega_i^{n_i-1}$, and the matrices $A'', A'''$, defined respectively by the linear forms $L_{i,j}''$ and $L_{i,j}'''$ as in \eqref{eqn:A_L}, are nonsingular mod $p$.
\end{prop}

Note that the roles of $\ubf,\vbf$ in $p\Lcal_\zbf^\ast$ are reversed compared to $\xbf, \ybf$ in $\Lcal_\zbf$; this does not affect our computations. We prove this result in \S \ref{sec:proof_dual} and later apply this in \S \ref{sec:s(z)>n}.

As a consequence of \eqref{eqn:keyfact2}, we show in Lemma \ref{lem:linear_algebra} a bounded version of the following linear algebra fact on the size of cosets of a subspace. Let $M$ denote an $m'\times m$ matrix mod $p$, and let $\bbf\in \bZ^{m'}$. If $M \xbf\equiv \bbf$ has solutions mod $p$, then the number of solutions mod $p$ equals the size of the kernel mod $p$, i.e. $|\{\xbf\in (\bF_p)^{m}:M \xbf\equiv \bbf\}| = |\{\xbf\in (\bF_p)^{m}:M \xbf\equiv \mathbf{0}\}|$. We require a bounded version of this, since we will consider only mod $p$ solutions in a discrete box of the form $D_{\Lbf}=[-\Lbf,\Lbf]$. 
In this case, we no longer have equality, but an upper bound suffices. Recall the definition
    \begin{align*}
        S(\bbf; D_{\Lbf}) := \{\xbf\in D_{\Lbf}: M\xbf \equiv \bbf \modd p\}
    \end{align*}
from Lemma \ref{lem:linear_algebra0}.
Then we have the following ``extension'' of \eqref{eqn:LA01}.

\begin{lemma}\label{lem:linear_algebra}
    Fix an $m'\times m$ matrix $M$ mod $p$. Fix an integer $1\leq L \leq p$, and let $\Lbf=(L_1,...,L_m)$ with $L \leq L_i\leq 2L$ for all $i$. Then for any $\bbf\in \bZ^{m'}$,
    $   |S(\bbf; D_{\Lbf})| \ll_m |S(\mathbf{0};D_{\Lbf})|.$
\end{lemma}
    \begin{proof}
    By \eqref{eqn:LA01}
    we have $|S(\bbf; D_{\Lbf})| \leq   |S(\mathbf{0}; D_{2\Lbf})| $. 
    Define the lattice $ \Lcal = \{\xbf\in \bZ^m: M\xbf \equiv \mathbf{0} \modd p\}$, so
     $S(\mathbf{0};D_{\Lbf}) = B_\Lbf\cap \Lcal$. By \eqref{eqn:keyfact2}, $|B_{2\Lbf}\cap \Lcal| \ll_m (2L/L)^s |B_{\Lbf}\cap \Lcal|$, where $s$ is as defined in \eqref{eqn:def_s} and in particular, $s \leq m$. Then $|S(\bbf; D_{\Lbf})| \leq |S(\mathbf{0};D_{2\Lbf})| = |B_{2\Lbf}\cap \Lcal| \ll_m |B_{\Lbf}\cap \Lcal| = |S(\mathbf{0};D_{\Lbf})|$,
    which proves the lemma.
    \end{proof}

The purpose of this lemma is to bound $|S(\bbf;D_\Lbf)|$ by a count of $\xbf$'s lying in $D_\Lbf$ instead of $D_{2\Lbf}$, which a priori is attained by \eqref{eqn:LA01}. We require the domain $D_\Lbf$ to stay the same, in order to prove Proposition \ref{prop:keyestimates}, which is a key ingredient in the main recursion relation in Proposition \ref{prop:recursive}.

\subsection{Proof of Proposition \ref{prop:keyfacts}}\label{sec:proof_keyfacts}
We require some classical facts from the geometry of numbers. 

\begin{lemma}\label{prop:minkowski}
Let $\Lcal$ be a full rank lattice in $\bR^{m}$, and let $B\subseteq \bR^{m}$ be a symmetric convex body. Let $\lambda_1,...,\lambda_{m}$ denote the successive minima of $B$ with respect to $\Lcal$. Then 
\begin{align}\label{eqn:minkowski}
    \prod_{i=1}^m\lam_i \asymp_m \det(\Lcal)/\mathrm{vol}(B),
\end{align}
    and, recalling the definition of $s$ in \eqref{eqn:def_s},
        \begin{align}\label{eqn:lattices_fact}
            |B\cap \Lcal|\asymp_m \prod_{i=1}^{s}\lambda_i^{-1}.
        \end{align}
\end{lemma}

\begin{proof}
The first claim \eqref{eqn:minkowski} is Minkowski's second theorem, originating in \cite{Min68}; see also \cite[Theorem 3.30]{TV06}. For the second claim \eqref{eqn:lattices_fact}, we cite the following result from \cite[\S 3]{TV06}: if $B$ is such that $B\cap \Lcal$ contains $m$ linearly independent vectors, then
    \begin{align}\label{eqn:lattices_fact2}
        |B\cap\Lcal|\asymp_m \vol(B) / \det(\Lcal).
    \end{align}
 To deduce \eqref{eqn:lattices_fact}, recall $s$ is the number of linearly independent vectors in $B\cap \Lcal$. 
    Let $\Lcal'$ be the sub-lattice of $\Lcal$ generated by the vectors in $B\cap \Lcal$ over $\bZ$, and let $\Rcal$ be the subspace of $\bR^m$ generated by the vectors in $B\cap \Lcal$ over $\bR$.
    Let $B'=B\cap \Rcal$. Then $B'$ is symmetric and convex in $\Rcal\cong \bR^s$.
    Let $\mu_1,...,\mu_s$ denote the successive minima of $B'$ with respect to $\Lcal'$. Then $\mu_i=\lam_i$ for $1\leq i \leq s$ and $|B\cap \Lcal| = |B'\cap \Lcal'|$. By \eqref{eqn:lattices_fact2} and \eqref{eqn:minkowski}, $|B'\cap\Lcal'|\asymp_s \vol(B') / \det(\Lcal') \asymp_s \prod_{i=1}^s \lam_i^{-1}$, as claimed.
\end{proof}

The last fact we need is a relation between the successive minima of a convex body with respect to a lattice and the successive minima of the dual body with respect to the dual lattice. First we state some definitions.
The dual of a convex body $B$ is 
    \begin{align*}
        B^\ast = \{\ubf\in \bR^{m}: \langle \ubf,\xbf\rangle \leq 1 \text{ for all } \xbf\in B\}.
    \end{align*}
For example, the dual of $B_H^m=\{\xbf\in \bR^m: |x_i|\leq H\}$ is $(B_H^m)^\ast = \{\xbf \in \bR^m: \sum_{i=1}^m |x_i|\leq 1/H\}$, where note that $(B_H^m)^\ast \subseteq B_{1/H}^m$.
Let $\lam_1,...,\lam_m$ denote the successive minima of a convex body $B$ with respect to a lattice $\Lcal$.
We can similarly define the successive minima $\lambda_1^\ast,...,\lambda_{m}^\ast$ of $B^\ast$ with respect to the dual lattice $\Lcal^\ast$. The next lemma states a relation between $\lam_1,...,\lam_m$ and $\lambda_1^\ast,...,\lambda_{m}^\ast$. This is a result of Mahler \cite{Mah39}, also recorded in e.g. \cite[Theorem 14.5]{GL87} and \cite[p.~221]{Cas97}.

\begin{lemma}\label{lem:lattice2} For all $1\leq i\leq m$, we have
    $1\leq  \lambda_i \lambda_{m+1-i}^\ast \leq (m!)^2$.
In particular, $\lambda_i \lambda_{m+1-i}^\ast \asymp_m 1$.
\end{lemma}
\begin{remark}
    There is a stronger upper bound of $m!$, for example, recorded in \cite[\S VIII, Theorem VI]{Cas97} but we need only the qualitative result of $\lambda_i \lambda_{m+1-i}^\ast \asymp_m 1$.
\end{remark}
  
Finally, we are ready to prove Proposition \ref{prop:keyfacts}, beginning with \eqref{eqn:keyfact2}. Let $\lambda_1,...,\lam_{m}$ and $\lam_1',...,\lam_m'$ denote, respectively, the successive minima of $B_{\Hbf}$ and $B_{\Hbf'}$ with respect to $\Lcal$.
By definition, $\lam_i$ is the infimum of $\lambda$'s for which $\lambda B_{\Hbf}$ contains $i$ many linearly independent vectors in $\Lcal$. Note that
    \begin{align*}
        \lam B_\Hbf = \lam B_{H_1}\times \cdots \times \lam B_{H_m} = \lam (\frac{H_1}{H_1'})B_{H_1'}\times \cdots \times \lam (\frac{H_m}{H_m'}) B_{H_m'},
    \end{align*}
and so
    \begin{align*}
        (\min_i \frac{H_i}{H_i'}) \lam_i\leq \lam_i' \leq (\max_i \frac{H_i}{H_i'}) \lam_i.
    \end{align*}
In particular, the assumption of dyadic variation for $H_i$ and $H_i'$ implies that $\lam_i' \asymp (H/H')\lam_i$.
Let $s$ and $s'$ be defined as in \eqref{eqn:def_s} respectively for $\lam_1,...,\lam_{m}$ and $\lam'_1,...,\lam'_{m}$. 
Then by applying \eqref{eqn:lattices_fact} twice and treating the cases $s'<s$ and $s'>s$ separately, we get
    \begin{align*}
        |B_\Hbf\cap \Lcal| \asymp_m \prod_{i=1}^s \lam_i^{-1} 
        \asymp (\frac{H}{H'})^s \prod_{i=1}^s (\lam'_i)^{-1} 
        \ll (\frac{H}{H'})^s (\prod_{i=1}^{s'} (\lam'_i)^{-1}) \asymp_m  (\frac{H}{H'})^s |B_{\Hbf'}\cap \Lcal|,
    \end{align*}
and the first claim \eqref{eqn:keyfact2} is verified. 

For the second claim \eqref{eqn:keyfact25},
note first that $|B_{a/\Hbf} \cap a\Lcal^\ast| =   |B_{1/\Hbf} \cap \Lcal^\ast|$ for any $a>0$, so it suffices to assume $a=1$.
Let $\mu^\ast_1,...,\mu^\ast_m$ and $\mu_1,...,\mu_m$ denote, respectively, the successive minima of $B_{1/\Hbf}$ with respect to $\Lcal^\ast$ and $(B_{1/\Hbf})^\ast$ with respect to $\Lcal$. Let $t^\ast$ and $t$ be defined as in \eqref{eqn:def_s} respectively for $\mu^\ast_1,...,\mu^\ast_m$ and $\mu_1,...,\mu_m$. 
By hypothesis, $\lam_1,...,\lam_m$ is the successive minima of $B_\Hbf$ with respect to $\Lcal$, and $s$ is as defined in \eqref{eqn:def_s} for $\lam_1,...,\lam_m$.
We must first show the relation
    \begin{align}\label{eqn:sstar_m-s}
        |B_{1/\Hbf}\cap \Lcal^\ast| \asymp_m \prod_{i=1}^{m-s} (\mu_i^\ast)^{-1}.
    \end{align}
By \eqref{eqn:lattices_fact}, $|B_{1/\Hbf}\cap \Lcal^\ast| \asymp_m \prod_{i=1}^{t^\ast} (\mu_i^\ast)^{-1}$, and by Lemma \ref{lem:lattice2}, $\mu_i^\ast \ll 1$ for $1\leq i\leq m-t$ and $\mu_i^\ast\gg 1$ for $m-t<i\leq m$. (We remark that this does not guarantee $t^\ast=m-t$.) Since we further require estimates on $\mu_i^\ast$ when $1\leq i \leq m-s$ and $m-s<i\leq m$, we must compare $(B_{1/\Hbf})^\ast$ to $B_\Hbf$. Notice that 
    \begin{align*}
        (B^m_{1/H})^\ast\subseteq (B_{1/\Hbf})^\ast  \subseteq (B^m_{1/cH})^\ast \quad \text{and} \quad (B^m_{1/H})^\ast\subseteq B_\Hbf \subseteq (B^m_{1/cH})^\ast
    \end{align*}
for some absolute constant $c>1$.
The inclusions on the left are clear since $B^m_{1/cH}\subseteq B_{1/\Hbf} \subseteq B^m_{1/H}$ with $c\geq 2$ for example. The inclusions on the right hold since $B^m_H\subseteq B_\Hbf\subseteq B^m_{2H}$, where $(B^m_{1/H})^\ast \subseteq B^m_H$, by the example before Lemma \ref{lem:lattice2}, and $(B_{2H}^m)^\ast \supseteq B^m_{1/cH}$ for some absolute $c$ which can be computed. With $(B_{1/\Hbf})^\ast$ and $B_\Hbf$ sandwiched between $ (B^m_{1/H})^\ast$ and $ (B^m_{1/cH})^\ast$, we must now compare the latter pair.

Let $\lam_{H,1},...,\lam_{H,m}$ and $\lam_{cH,1},...,\lam_{cH,m}$ denote, respectively, the successive minima of $(B^m_{1/H})^\ast$ and $(B^m_{1/cH})^\ast$ with respect to $\Lcal$. Since $(B_{1/cH}^m)^\ast = c(B_{1/H}^m)^\ast$, we learn that $\lam_{H,i}\asymp \lam_{cH,i}$, and hence $\mu_i\asymp \lam_i$. This further implies by Lemma \ref{lem:lattice2} that $\mu_i^\ast \ll 1$ for $1\leq i\leq m-s$ and $\mu_i^\ast\gg 1$ for $m-s<i\leq m$.
Then treating the cases $t^\ast \leq m-s$ and $t^\ast >m-s$ separately verifies \eqref{eqn:sstar_m-s}.
With this in hand, the remaining steps are analogous to the
deduction of \eqref{eqn:keyfact2}.

To show \eqref{eqn:keyfact3}, 
again it suffices to assume $a=1$.
We resume the notation that $\lambda_1,...,\lambda_{m}$ is the successive minima of $B_{\Hbf}$ with respect to $\Lcal$. By \eqref{eqn:lattices_fact}, \eqref{eqn:minkowski}, and the dyadic variation of $H_i$,
    \begin{align*}
        |B_{\Hbf}\cap \Lcal|
        \asymp_m \prod_{i=1}^{s} \lambda_i^{-1}
       \asymp_m \frac{\vol(B_{\Hbf})}{\det(\Lcal)}\prod_{i=s+1}^{m} \lambda_i
       \asymp_m \frac{H^m}{\det(\Lcal)}\prod_{i=s+1}^{m} \lambda_i.
    \end{align*}
It remains to show $ |B_{1/\Hbf} \cap \Lcal^\ast|
        \asymp_m \prod_{i=s+1}^{m} \lambda_i$,
  which is given by Lemma \ref{lem:lattice2} and \eqref{eqn:sstar_m-s}, since $\lam_i\asymp \mu_i$. This concludes the proof of \eqref{eqn:keyfact3} and hence Proposition \ref{prop:keyfacts}.

\subsection{Proof of Proposition \ref{prop:keyfact_dual}}\label{sec:proof_dual}
Let $\zbf=(z_1,...,z_s)$ with $z_i\in \bF_{p^{n_i}}\setminus \{0\}$, and
\begin{align}\label{eqn:lattice_def_1}
        \Lcal_\zbf := \{(\xbf,\ybf)\in \bZ^{2n}: \lambda_i(\xbf) = z_i \lambda_i'(\ybf)\in \bF_{p^{n_i}} \text{ for }1 \leq i \leq s \}
    \end{align}
by hypothesis.
We split the proof into three main steps, which we summarize now.

(1) In the first step, we rewrite the $s$ equations in the above definition of $\Lcal_\zbf$ in terms of matrix equations mod $p$. More precisely, we will show
    \begin{align}\label{eqn:lattice_matrices}
        \Lcal_\zbf = \{(\xbf,\ybf)\in \bZ^{2n}: A \xbf \equiv  M_\zbf A' \ybf \modd p\},
    \end{align}
where recall that $A, A' \in \GL_n(\bF_p)$ are matrices associated to the linear forms in the definition of $\lam_i$ and $\lam_i'$, respectively. We will see that $M_\zbf$ is an $n\times n$ nonsingular matrix mod $p$ that carries a multiplication rule for the finite fields $\bF_{p^{n_i}}$ for $1\leq i \leq s$ with respect to the power bases $1,\omega_i,...,\omega_i^{n_i-1}$.
This step is more technical than the corresponding step in \cite{BC10}, where $s=n$ and $n_i=1$ for all $i$, so that $M_\zbf$ is simply the diagonal matrix with entries $z_1,...,z_n$. Here, we require relations between equations over different finite fields and linear equations over $\bF_p$.

(2) The advantage of writing the lattice $\Lcal_\zbf$ in the form \eqref{eqn:lattice_matrices} is so that we can compute the (scaled) dual lattice $p\Lcal_\zbf^\ast$ in a straightforward manner. Indeed, in this second step, we apply properties of inner products to obtain
 \begin{align}\label{eqn:dual_lattice0}
        p\Lcal^\ast_\zbf = \{(\ubf,\vbf)\in \bZ^{2n}: M_\zbf^T(A^{-1})^T \ubf \equiv   (-A'^{-1})^T \vbf \modd p\}.
    \end{align}
    
  (3)  In the third step, we convert this matrix equation back into simultaneous equations over finite fields, similar to \eqref{eqn:lattice_def_1}. In the case considered by \cite{BC10}, $M_\zbf\equiv M_\zbf^T$ since $M_\zbf$ is diagonal, so this step is completely analogous (in the opposite direction) to the first step. However, in our general setting, $M_\zbf$ is not diagonal, so we require new tools to achieve the desired form in \eqref{eqn:dual_lattice_want}.

Since the first and third steps are more involved and similar in nature, we discuss them together in \S \ref{sec:encoding_multiplication}, and now carry out the easier, second step first.

\subsubsection{Computing the dual lattice} 
In this step, given a lattice defined by matrix equations mod $p$, we compute its determinant as well as its dual.
    \begin{lemma}\label{lem:keyfact_dual0}
        Let $A,A', M\in \GL_n(\bF_p)$, and consider the lattice
            \begin{align*}
                  \Lcal &= \{(\xbf,\ybf)\in \bZ^{2n}: A \xbf \equiv  M A' \ybf \modd p\}.
            \end{align*}
        Then $ \det(\Lcal) = p^n$
        and 
            \begin{align}
        p\Lcal^\ast &= \{(\ubf,\vbf)\in \bZ^{2n}: M^T(A^{-1})^T \ubf \equiv   (-A'^{-1})^T \vbf \modd p\}
        \label{eqn:dual_lattice}.
    \end{align}
    \end{lemma}
    
        \begin{proof}
To compute $\det(\Lcal)$, recall this is the volume of the fundamental parallelepiped. Fix a large $R>0$ (relative to $p$) that we will take to infinity. Recall $B_R = \{x\in \bR: -R\leq x \leq R\}$. Then
    \begin{align*}
        \det(\Lcal) = (1+O(1/R))\frac{\vol(B_R^{2n})}{|B_R^{2n}\cap \Lcal|}
    \end{align*}
by e.g. \cite[Lemma 3.22]{TV06}. We have $\vol(B_R^{2n}) = (2R)^{2n}$. To compute $|B_R^{2n}\cap \Lcal|$, note that we can first choose $\xbf$ freely in $B_R^n$; the number of choices for each coordinate $x_i$ is $R_1:=
2\lfloor R \rfloor+1 = 2R + O(1)$. Then $\ybf\equiv (MA')^{-1}A\xbf$ is uniquely determined mod $p$, and hence each coordinate has $\lfloor R_1/p \rfloor = R_1/p + O(1)=2R/p + O(1)$ choices. Thus, 
    \begin{align*}
        |B_R^{2n}\cap \Lcal| = (2R + O(1))^n(2R/p+O(1))^n = (2R)^n(2R/p)^n(1+O(p/R)) = [(2R)^{2n}/p^n](1+O(p/R))
    \end{align*}
and so
    \begin{align*}
        \det(\Lcal) = (1+O(1/R))\frac{(2R)^{2n}}{(1+O(p/R))(2R)^{2n}/p^n} = (1+O(1/R))\frac{p^n}{(1+O(p/R))}. 
    \end{align*}
Finally take $R$ to infinity to conclude.

To prove \eqref{eqn:dual_lattice}, let us first show that $p\Lcal^\ast \subseteq \bZ^{2n}$. Indeed, $p\bZ^{2n}\subseteq \Lcal$, and so if $(\ubf,\vbf)\in \Lcal^\ast$, then by definition of the dual lattice in \eqref{eqn:def_dual}, $\langle (\ubf,\vbf),(\xbf,\ybf) \rangle \in \bZ$ for all $(\xbf,\ybf)\in \Lcal$, and in particular for all $(\xbf,\ybf)\in p\bZ^{2n}$. Considering the 
$2n$ choices of $(\xbf,\ybf)\in p\bZ^{2n}$ where one coordinate is $p$ and otherwise 0, we deduce that $(\ubf,\vbf)\in p^{-1}\bZ^{2n}$ and hence $\Lcal^\ast \subseteq p^{-1}\bZ^{2n}$. Thus by \eqref{eqn:def_dual},
    \begin{align*}
        p\Lcal^\ast = \{(\ubf,\vbf)\in \bZ^{2n}:\langle (\ubf,\vbf),(\xbf,\ybf) \rangle \equiv 0 \modd p \text{ for all }(\xbf,\ybf)\in \Lcal\}.
    \end{align*}
Now we show \eqref{eqn:dual_lattice}. Let $(\ubf,\vbf)\in \bZ^{2n}$. It suffices to show that $(\ubf,\vbf)$ satisfies $\langle (\ubf,\vbf),(\xbf,\ybf) \rangle \equiv 0 \modd p$
for all $(\xbf,\ybf)\in \Lcal$ if and only if it satisfies $M^T(A^{-1})^T \ubf \equiv   (-A'^{-1})^T \vbf \modd p$. By definition, $(\xbf,\ybf)\in \Lcal$ if and only if $\xbf \equiv A^{-1} M A'\ybf$, so
    \begin{align*}
    \langle \ubf,\xbf \rangle + \langle \vbf, \ybf \rangle 
    \equiv \langle \ubf,A^{-1} M A'\ybf \rangle + \langle \vbf, \ybf \rangle
    \equiv  \langle (A^{-1} M A')^\ast \ubf,\ybf \rangle + \langle \vbf, \ybf \rangle 
    \equiv \langle (A^{-1} M A')^\ast \ubf + \vbf, \ybf\rangle,
    \end{align*}
where the adjoint of a matrix $C$ mod $p$ is $C^\ast \equiv C^T$. Thus  $\langle (\ubf,\vbf),(\xbf,\ybf) \rangle \equiv 0$
for all $(\xbf,\ybf)\in \Lcal$ if and only if $\langle (A^{-1} M A')^\ast \ubf + \vbf, \ybf\rangle\equiv 0$ for all $\ybf$. This is equivalent to $\mathbf{0}\equiv (A^{-1} M A')^\ast \ubf + \vbf = (MA')^\ast (A^{-1})^\ast \ubf +\vbf$, or in other words $ (MA')^\ast (A^{-1})^\ast \ubf \equiv -\vbf$, so
    \begin{align*}
      (A^{-1})^T \ubf \equiv    -((MA')^{-1})^T \vbf \equiv - (M^{-1})^T  (A'^{-1})^T \vbf,
    \end{align*}
as desired.
        \end{proof}

\subsubsection{Encoding multiplication in finite fields}\label{sec:encoding_multiplication}
We begin by expressing the equations in \eqref{eqn:lattice_def_1} as a matrix equation over $\bF_p$. Let $1,\omega,...,\omega^{m-1}$ be a basis of $\bF_{p^m}$, where $\omega$ is a root of a fixed polynomial $f(x) = x^m-c_{m-1}x^{m-1}-\cdots - c_1x-c_0$ that is irreducible over $\bF_p$; in particular, we have the relation $\omega^m = c_{m-1}\omega^{m-1}+\cdots + c_1\omega + c_0$. Then multiplication by $\omega$ in $\bF_{p^m}$ can be described by a linear map over $\bF_p$ represented by the matrix 
 \begin{align*}
      M_f:=  \begin{bmatrix}
           0& & &\cdots &&c_{0}\\
           1&0& &\cdots &&c_{1}\\
           0&1& &\cdots &&c_{2}\\
           0&0&1&\cdots &&c_{3}\\
           \vdots  &\vdots&&\ddots&&\vdots\\
           0&0&0&\cdots &1&c_{m-1}
        \end{bmatrix}_{m\times m}.
    \end{align*}
Note that this is the so-called companion matrix of $f$. 
Then $M_f$ encodes multiplication by $\omega$, in the sense that for $b \in \bF_{p^m}$ with $b = b_0 + b_1\omega + \cdots + b_{m-1}\omega^{m-1}$ and $\bbf=(b_0,...,b_{m-1})^T$, we have $\omega b = d_0 + d_1\omega + \cdots + d_{m-1}\omega^{m-1}$, where $\dbf = (d_0,...,d_{m-1})^T$ satisfies $\dbf \equiv M_f \bbf$. We can then generalize this for multiplication by any element in $\bF_{p^m}$.
Indeed, for $a\in \bF_{p^m}$, write $a=a_0 + a_1\omega + \cdots + a_{m-1}\omega^{m-1}$ and let $\abf=(a_0,...,a_{m-1})^T$. Define the matrix
    \begin{align}\label{eqn:M_multiplication_matrix}
        M(m;\abf;\omega) = a_{m-1}M_f^{m-1} + \cdots + a_1 M_f + a_0I_m,
    \end{align}
    where $I_m$ is the $m\times m$ identity matrix.
Then we have the following ``multiplication rule".
    \begin{lemma}\label{lem:B_multiplication} Fix an integer $m\geq 1$ and $a,b,d\in \bF_{p^m}$.
    Let $1,\omega,...,\omega^{m-1}$ be a basis for $\bF_{p^m}$, and let $\underline{\omega}=(1,\omega,...,\omega^{m-1})^T$. Define $\abf=(a_0,...,a_{m-1})^T,\bbf=(b_0,...,b_{m-1})^T,\dbf=(d_0,...,d_{m-1})^T \in (\bF_p)^{m}$ by
    $a=\abf^T\underline{\omega}$, $b=\bbf^T \underline{\omega}$, and $d=\dbf^T \underline{\omega}$, respectively.
    Then $d=ab$ if and only if $\dbf \equiv M(m;\abf;\omega)\bbf$.
    \end{lemma}
\begin{proof}
    We show by induction that for all $i\geq 1$,
    $d=\omega^i b$ if and only if $\dbf \equiv M_f^i \bbf$. The base case holds by our above discussion. Let $d'=\omega^{i-1}b$ so $d=\omega d'$. By the inductive hypothesis $\dbf' \equiv M_f^{i-1}\bbf$, so $\dbf \equiv M_f (M_f^{i-1}\bbf) \equiv M_f^i \bbf$. The lemma then follows by linearity.
\end{proof}

Apply Lemma \ref{lem:B_multiplication} to the $s$ equations $ \lam_i(\xbf) = z_i\lam'_i(\ybf)$ on the right-hand side of \eqref{eqn:lattice_def_1},
    where recall $\lam_i(\xbf) =L_{i,1}(\xbf) + L_{i,2}(\xbf)\omega_i + \cdots + L_{i,n_i}(\xbf)\omega_i^{n_i-1}$ and write $\zbf_i =  (z_{i,1},...,z_{i,n_i})$ where $z_i = z_{i,1}+z_{i,2}\omega_i + \cdots + z_{i,n_i}\omega_i^{n_i-1}$,
to get the system of equations
    \begin{align*}
        \begin{bmatrix}
            L_{1,1}(\xbf)\\
            \vdots\\
            L_{1,n_1}(\xbf)\\
            \vdots\\
            L_{s,1}(\xbf)\\
            \vdots\\
            L_{s,n_s}(\xbf)\\
        \end{bmatrix}
        \equiv
        \begin{bmatrix}
            M(n_1;\zbf_1;\omega_1)\\
            &\ddots\\
            && M(n_s;\zbf_s;\omega_s)
        \end{bmatrix}_{n\times n}
        \begin{bmatrix}
            L'_{1,1}(\ybf)\\
            \vdots\\
            L'_{1,n_1}(\ybf)\\
            \vdots\\
            L'_{s,1}(\ybf)\\
            \vdots\\
            L'_{s,n_s}(\ybf)\\
        \end{bmatrix}.
    \end{align*}
The matrix of $M(n_i;\zbf_i;\omega_i)$'s is block diagonal. We further consolidate notation and write this as $A\xbf\equiv M_\zbf A'\ybf$, so
    \begin{align}\label{eqn:lattice_structure}
        \Lcal_\zbf=\{(\xbf,\ybf)\in \bZ^{2n}: A\xbf\equiv M_\zbf A' \ybf \modd p\},
    \end{align}
which proves \eqref{eqn:lattice_matrices} and completes the first step.

In the second step, we compute the dual lattice of \eqref{eqn:lattice_structure}.
To apply Lemma \ref{lem:keyfact_dual0}, we require $M_\zbf$ to be nonsingular mod $p$. The next lemma provides a necessary and sufficient condition.

\begin{lemma}\label{lem:Bmatrix} Let $a = a_0+a_1\omega +\cdots+ a_{m-1} \omega^{m-1} \in \bF_{p^m}$ and $\abf=(a_0,...,a_{m-1})^T$. Let $M(m;\abf;\omega)$ be as defined above. Then $\det M(m;\abf;\omega) \equiv 0 \modd p$ if and only if $a_0,...,a_{m-1}$ are all $0\modd p$.
\end{lemma}
    \begin{proof}
        The backward direction is clear. For the forward direction, suppose that at least one $a_i$ is nonzero mod $p$. Recall the notation $\underline{\omega}=(1,\omega,...,\omega^{m-1})^T$. Then $a=\abf^T \underline{\omega}$ is nonzero as an element of $\bF_{p^m}$ and hence has a multiplicative inverse $v=v_0+v_1\omega + \cdots +v_{m-1}\omega^{m-1}$. 
    Let $\dbf = (1,0,...,0)^T$ and $\mathbf{v} = (v_0,...,v_{m-1})^T$.
    Then $1=(\abf^T \underline{\omega})(\vbf^T \underline{\omega})$,
    so by Lemma \ref{lem:B_multiplication},
   $\dbf  \equiv M(m;\abf;\omega) \mathbf{v}$.
    Since the inverse $v$ is unique, this system of linear equations has a unique solution $\mathbf{v}$, and hence the matrix $M(m;\abf;\omega)$ must be nonsingular.
    \end{proof}

    Apply this lemma to the matrices $M(n_i;\zbf_i;\omega_i)$ to see that they are nonsingular mod $p$, since $z_i\in \bF_{p^{n_i}}\setminus \{0\}$ by supposition.
    Hence $M_\zbf$ is nonsingular mod $p$, so Lemma \ref{lem:keyfact_dual0} gives \eqref{eqn:dual_lattice0}, i.e.
   \begin{align}\label{eqn:dual_lattice_new0}
        p\Lcal_\zbf^\ast = \{(\ubf,\vbf)\in \bZ^{2n}: M_\zbf^T(A^{-1})^T \ubf \equiv - (A'^{-1})^T \vbf \modd p\},
    \end{align} 
and this completes the second step.

In the third step, our goal is to write the matrix equations in \eqref{eqn:dual_lattice_new0} in terms of equations over finite fields.
If $n_i=1$ for all $1\leq i\leq s$, then $M_\zbf$ is diagonal, so $M_\zbf^T=M_\zbf$. (This is the only case considered in \cite{BC10}, where $F$ is a product of $n$ linear forms over $\bF_p$ in $n$ variables.) In this special case, we can immediately apply Lemma \ref{lem:B_multiplication} to the $s$ matrix equations in \eqref{eqn:dual_lattice_new0}, each associated to a block (here, each block is $1\times 1$) of $M_\zbf^T$, to rewrite $p\Lcal_\zbf^\ast$ in the form presented in \eqref{eqn:dual_lattice_want} 
with $A'' = -(A'^{-1})^T$ and $A''' = (A^{-1})^T$.

In general, however, we do not have $M_\zbf^T=M_\zbf$, so it is not obvious how we can write this new matrix equation in terms of a set of simultaneous equations in different finite fields. Here, a new idea overcomes this difficulty. 

\begin{lemma}
    Let $p\Lcal^\ast_\zbf$ be as in \eqref{eqn:dual_lattice_new0}. Then 
     \begin{align}\label{eqn:prop52_proof}
            p\Lcal_\zbf^\ast = \{(\ubf,\vbf)\in \bZ^{2n}: \lam''_i(\vbf) = z_i \lam_i'''(\ubf) \in \bF_{p^{n_i}} \text{ for }1\leq i \leq s\},
        \end{align}
    where $\lam''_i(\vbf) = L''_{i,1}(\vbf)+L''_{i,2}(\vbf)\omega_i+\cdots + L''_{i,n_i}(\vbf)\omega_i^{n_i-1}$ and $\lam'''_i(\ubf) = L'''_{i,1}(\ubf)+L'''_{i,2}(\ubf)\omega_i+\cdots + L'''_{i,n_i}(\ubf)\omega_i^{n_i-1}$, and the matrices $A'', A'''$, defined respectively by the linear forms $L_{i,j}''$ and $L_{i,j}'''$, are nonsingular mod $p$.
\end{lemma}

\begin{proof}
Notice that the structure of the lattice as written in \eqref{eqn:lattice_structure} accommodates matrices that are not symmetric by allowing for a change of variables. More precisely, if there exists a nonsingular matrix $C$ mod $p$ for which
    \begin{align}\label{eqn:BC=CBt}
        M_\zbf C \equiv CM_\zbf^T \modd p,
    \end{align}
 then we can write $A\xbf\equiv M_\zbf A'\ybf= (M_\zbf C)(C^{-1}A')\ybf$ so the dual lattice, by Lemma \ref{lem:keyfact_dual0}, is
    \begin{align*}
           p\Lcal_\zbf^\ast = \{(\ubf,\vbf)\in \bZ^{2n}: (M_\zbf C)^T(A^{-1})^T \ubf \equiv - ((C^{-1}A')^{-1})^T \vbf \modd p\}.
    \end{align*}
Since $(M_\zbf C)^T =(CM^T_\zbf)^T = M_\zbf C^T$ by \eqref{eqn:BC=CBt},
\begin{align}\label{eqn:dual_lattice_new}
           p\Lcal_\zbf^\ast = \{(\ubf,\vbf)\in \bZ^{2n}: M_\zbf C^T(A^{-1})^T \ubf \equiv - ((C^{-1}A')^{-1})^T \vbf \modd p\}.
    \end{align}
This is in the correct form of $A''\vbf \equiv M_\zbf A'''\ubf$ with $A'' = - ((C^{-1}A')^{-1})^T $ and $A''' = C^T(A^{-1})^T$. Finally, apply Lemma \ref{lem:B_multiplication} $s$ times, to each of the $s$ systems of equations in \eqref{eqn:dual_lattice_new}, determined by each of the $s$ blocks of $M_\zbf$. This rewrites the matrix equations as desired in \eqref{eqn:prop52_proof}.
This proves the lemma, except for justifying the existence of $C$, which is guaranteed by the next lemma on similar matrices.
\end{proof}

Recall that two matrices $A,A'\in M_{m\times m}(R)$ over a ring $R$ are similar if there exists an invertible matrix $C\in M_{m\times m}(R)$ such that $A=CA'C^{-1}$. We refer the reader to e.g. \cite[\S 2.6]{Kap69} or \cite{TZ59} for the following fact.

\begin{lemma}\label{lem:algebra2}
    Let $K$ be a field, and let $A\in M_{m\times m}(K)$. Then $A$ is similar to $A^T$. In fact, there exists a symmetric, nonsingular $C$ for which $A=CA^TC^{-1}$.
\end{lemma}

\section{Energy estimates: proof of Theorem \ref{thm:main_energy}}\label{sec:energy}
In this section, we prove Theorem \ref{thm:main_energy}, which gives an optimal bound on the energy $E((\Nbf,\Nbf+\Hbf]) = |\{(\xbf,\xbf',\ybf,\ybf')\in (\Nbf,\Nbf+\Hbf]^4: \lambda_i(\xbf)\lambda_i(\ybf')=\lambda_i(\xbf')\lambda_i(\ybf) \in \bF_{p^{n_i}} \text{ for all } i\}|$,  for $\Hbf$ such that $H\leq H_i\leq 2H$ for all $i$ and $1\leq H\leq p^{1/2}$. Symmetry plays a significant role in the proof, and here we summarize the main steps:

(1) In \S \ref{sec:energy_shifting}, we show that it suffices to bound the energy where the domain of $(\xbf,\xbf',\ybf,\ybf')$ is $[-\Hbf,\Hbf]^4$. This symmetry allows us to apply results from the geometry of numbers in \S \ref{sec:lattices}.

 (2) We introduce a generalized, \textit{asymmetric} energy $E(\Hbf,\Kbf)$, where  $\lambda_i(\xbf), \lambda_i(\ybf')$, $\lambda_i(\xbf'), \lambda_i(\ybf)$ are defined by linear forms associated to \textit{different} matrices, and $(\xbf,\xbf')$ and $(\ybf,\ybf')$ lie in \textit{different} boxes. We need this flexibility to prove a recursive formula, which consists of energies with different domains (in terms of size and dimension) and different linear forms (hence different matrices).

 (3) In \S \ref{sec:restricted_energy}, we decompose $E(\Hbf,\Kbf)$ into energies $E'(\Hbf,\Kbf), E''(\Hbf,\Kbf)$ with additional restrictions. This decomposition is quite natural as they correspond, respectively, to the not-zero-divisor and zero-divisor cases from \S \ref{sec:energy_prelim}.
        We will prove bounds on these two energies, which represent savings, respectively, from a decrease in the box size and from a reduction in dimension.

(4) In \S \ref{sec:recursive_formula}, we prove a recursive relation on, essentially, the term $E(\Hbf,\Hbf)$, which first requires a de-symmetrization, utilizing results on lattices from \S \ref{sec:lattices}, and then applying bounds from step (3). We remark that the condition  $H\leq p^{1/2}$ is required specifically in \S \ref{sec:assemble_cases_s}. Lastly, by iterating this recursive formula, we derive Theorem \ref{thm:main_energy} in Corollary \ref{cor:recursive}.

In \S \ref{sec:cor_energy}, we will prove Theorem \ref{cor:main_energy}, which is a short deduction from Theorem \ref{thm:main_energy}.

\begin{remark}
    It has been noted in \cite{BC10} that one could potentially prove a nontrivial bound on $S(F;\Nbf,\Hbf)$ for $\Hbf$ with $\|\Hbf\|>p^{n/4+\kappa}$, at least for forms that satisfy the hypotheses of Theorem \ref{thm:main} and in addition split into linear forms over $\bF_p$. This condition on $\Hbf$ is weaker than the condition $H_i>p^{1/4+\kappa}$ for all $i$ achieved in the current paper. Such a result would require an optimal bound on the energy without the dyadic restriction imposed on $\Hbf$ in Theorem \ref{thm:main_energy}, and this likely requires new ideas that lift this restriction in Proposition \ref{prop:keyfacts}.

    We mention here progress on energy estimates that do not require $H_i$ to be of the same size for all $i$. For $n=2,3$, Gabdullin \cite{Gab18} proved that Konyagin's \cite{Kon10} estimate of the energy in \eqref{eqn:literature_energy} holds for $\Hbf$ with $H_i\leq p^{1/2}/2$, thus producing a nontrivial upper bound on the corresponding character sum over the finite field $\bF_{p^n}$ for $\|\Hbf\|>p^{(1/4+\kappa)n}$. Kerr \cite{Ker21} extended Konyagin's energy estimate for $\Hbf$ such that $\max H_i \ll p^{1/n} \min H_i$, for arbitrary $n$.
\end{remark}

\subsection{Preliminary shifting}\label{sec:energy_shifting} Let $D_1,D_2 \subseteq\bR^{n}$ denote sets of integral tuples, and define
    \begin{align*}
        E(D_1,D_2) = |\{(\xbf,\xbf',\ybf,\ybf')\in D_1^2\times D_2^2: \lambda_i(\xbf)\lambda_i(\ybf')=\lambda_i(\xbf')\lambda_i(\ybf) \text{ for } 1\leq i \leq s\}|.
    \end{align*}
Let $E(D_1)$ denote $E(D_1,D_1)$.
Since the lattice results we will apply demand that $D_1,D_2$ are symmetric about the origin, we first reduce the task of bounding $E((\Nbf,\Nbf+\Hbf])$ to bounding $E(D_{\Hbf})$, where recall $D_\Hbf=[-\Hbf,\Hbf]$. 

\begin{lemma}\label{lem:shifting}
 Let $\Nbf=(N_1,...,N_n)$ and $\Hbf=(H_1,...,H_n)$. Then $ E((\Nbf,\Nbf+\Hbf]) \leq E(D_{\Hbf})$.
\end{lemma}
A similar deduction can be found in e.g. \cite[Lemma 3]{Kar70}, and we present it here.  
For a fixed $(\ybf,\ybf')\in (\Nbf,\Nbf+\Hbf]^2$ and an integral set $D_1$, define
    \begin{align*}
        S(\ybf,\ybf';D_1) := \{(\xbf,\xbf')\in D_1^2: \lambda_i(\xbf)\lambda_i(\ybf')=\lambda_i(\xbf')\lambda_i(\ybf) \text{ for } 1\leq i \leq s\}.
    \end{align*}
Then
    \begin{align*}
        E((\Nbf,\Nbf+\Hbf]) = \sum_{(\ybf,\ybf')\in (\Nbf,\Nbf+\Hbf]^2} |S(\ybf,\ybf';(\Nbf,\Nbf+\Hbf])|.
    \end{align*}
Fix a solution $(\xbf,\xbf')\in  S(\ybf,\ybf';(\Nbf,\Nbf+\Hbf])$, and define a map $f: S(\ybf,\ybf';(\Nbf,\Nbf+\Hbf]) \rightarrow S(\ybf,\ybf';D_{\Hbf})$ by $  f(\tilde{\xbf},\tilde{\xbf}') = (\xbf-\tilde{\xbf},\xbf'-\tilde{\xbf}').$
    We check this is well-defined. Indeed, $\lam_i(\xbf - \tilde{\xbf})\lam_i(\ybf') = \lam_i(\xbf)\lam_i(\ybf')-\lam_i( \tilde{\xbf})\lam_i(\ybf') = \lam_i(\xbf' - \tilde{\xbf}')\lam_i(\ybf)$ and $|x_i-\tilde{x}_i|,|x_i'-\tilde{x}_i'|\leq H_i$. This map is clearly injective, so $|S(\ybf,\ybf';(\Nbf,\Nbf+\Hbf])| \leq |S(\ybf,\ybf';D_{\Hbf})|$. Since this holds for every $(\ybf,\ybf')$,
    \begin{align*}
        E((\Nbf,\Nbf+\Hbf]) \leq \sum_{(\ybf,\ybf')\in (\Nbf,\Nbf+\Hbf]^2} |S(\ybf,\ybf';D_{\Hbf})| = E(D_{\Hbf},(\Nbf,\Nbf+\Hbf]).
    \end{align*}
By reversing the roles of $(\ybf,\ybf')$ and $(\xbf,\xbf')$, we deduce that
    \begin{align*}
        E(D_{\Hbf},(\Nbf,\Nbf+\Hbf]) \leq E(D_{\Hbf}),
    \end{align*}
so to verify Theorem \ref{thm:main_energy}, it suffices to bound the right-hand side.

\subsection{Bounds on asymmetric and restricted energies}\label{sec:restricted_energy} 
Instead of directly bounding $E(D_{\Hbf})$, we will prove a more general case. We must do so since the recursion that will lead to Theorem \ref{thm:main_energy} demands it. Fix $n\geq 1$, and let $\nbf=(n_1,...,n_s)$ be a partition of $n$.
Fix matrices $A_1,A_2,A_3,A_4\in \GL_n(\bF_p)$.  For $\ell=1,2,3,4$, write
    \begin{align}\label{eqn:def_Aell}
    A_\ell\equiv
    \begin{bmatrix}
        U^\ell_1\\
        \vdots\\
        U^\ell_{s}
    \end{bmatrix}_{n\times n}, \qquad  \text{ where } U^\ell_i :=  \begin{bmatrix}
            a_{i,1,1}^\ell&\cdots & a^\ell_{i,1,n}\\
            \vdots &&\vdots\\
            a_{i,n_i,1}^\ell&\cdots& a^\ell_{i,n_i,n}
        \end{bmatrix}_{n_i\times n}.
    \end{align}
    Let $L_{i,j}^\ell\in \bF_p[X_1,...,X_n]$ denote the linear form with coefficients $a_{i,j,1}^\ell,..., a_{i,j,n}^\ell$. 
Let $1,\omega_i,...,\omega_i^{n_i-1}$ be a basis for $\bF_{p^{n_i}}$, and define $ \lambda_i^\ell(\cdot) = L_{i,1}^\ell(\cdot) + L_{i,2}^\ell(\cdot)\omega_i + \cdots + L_{i,n_i}^\ell(\cdot)\omega_i^{n_i-1}.$ Consider the simultaneous equations
    \begin{align}\label{eqn:simult_lam}
\lambda_i^1(\xbf)\lambda_i^4(\ybf')=\lambda_i^2(\xbf')\lambda_i^3(\ybf) \in \bF_{p^{n_i}}, \qquad 1\leq i \leq s.
    \end{align}
Define the ``generalized'' energy $  E(A_1,A_2,A_3,A_4):=  |\{(\xbf,\xbf',\ybf,\ybf')\in D_\Hbf^4: \eqref{eqn:simult_lam}\}|$
and an upper bound
    \begin{align}\label{eqn:def_C}
        C(\nbf,\Hbf):=\sup_{\substack{A_1,A_2,A_3,A_4\\ A_i\in \GL_n(\bF_p)}}  E(A_1,A_2,A_3,A_4)
    \end{align}
that is uniform in $A_1,...,A_4$. Later in \S \ref{sec:recursive_formula} we will prove a recursive relation for $C(\nbf,\Hbf)$.  
For now, we introduce some notation. For a partition $\nbf=(n_1,...,n_s)$ of $n$, define $|\nbf|=n_1+\cdots + n_s$. For two vectors $\nbf=(n_1,...,n_s)$ and $\nbf'=(n'_1,...,n'_s)$ such that $\nbf \neq \nbf'$, define the binary relation $\nbf' \prec \nbf$ if for each $i$, either $n'_i=0$ or $n'_i=n_i$. For example, $(0,1,1)\prec (2,1,1)$ and $(0,0,1)\prec (2,1,1)$, but $(1,1,1)\not\prec(2,1,1)$. The purpose of this binary relation is to capture a decrease in dimension, where increments depend on the partition $(n_1,...,n_s)$. 
By convention, let $C(\mathbf{0},\Hbf)=1$.

Let $1\leq H,K\leq p$, and let $\Hbf=(H_1,...,H_n)$ with $H\leq H_i\leq 2H$ and $\Kbf=(K_1,...,K_n)$ with $K\leq K_i\leq 2K$.
To prepare for recursion, we must give bounds on the following energy where $(\xbf,\xbf')$ and $(\ybf,\ybf')$ lie in boxes of different size:
\begin{align*}
   E(\Hbf,\Kbf):=  |\{(\xbf,\xbf',\ybf,\ybf')\in 
   D_\Hbf^2 \times D_\Kbf^2:
   \eqref{eqn:simult_lam}   \}|.
\end{align*} 
We must additionally decompose $E(\Hbf,\Kbf) = E'(\Hbf,\Kbf)+E''(\Hbf,\Kbf)$ into an energy with nonzero conditions
    \begin{align*}
         E'(\Hbf,\Kbf) :=  |\{(\xbf,\xbf',\ybf,\ybf')\in 
         D_\Hbf^2 \times D_\Kbf^2:
         \eqref{eqn:simult_lam}
        \text{ and } \lambda_i^{1}(\xbf) \lambda_i^{4}(\ybf')\neq 0 \text{ for all } i\}|
    \end{align*}
and its complement
    \begin{align*}
         E''(\Hbf,\Kbf) :=   |\{(\xbf,\xbf',\ybf,\ybf')\in 
         D_\Hbf^2 \times D_\Kbf^2:
         \eqref{eqn:simult_lam}
    \text{ and }\lambda_{j}^{1}(\xbf) \lambda_{j}^{4}(\ybf')=0 \text{ for some } j\}|.
    \end{align*}
We now prove bounds on these two energies. In particular, the bound on $E'(\Hbf,\Kbf)$ is profitable when $K<H$, producing savings from reducing the size of the box, while the bound on $E''(\Hbf,\Kbf)$ gains advantage from reducing the dimension from $n=|\nbf|$ to $|\tilde{\nbf}|$, for $\tilde{\nbf}\prec \nbf$. These ultimately contribute to the first and second terms, respectively, of the bound in Proposition \ref{prop:recursive}.

\begin{prop}\label{prop:keyestimates1}
    Fix matrices $A_1,A_2,A_3,A_4\in \GL_n(\bF_p)$, and fix a partition $\nbf=(n_1,...,n_s)$ of $n$. Let $1\leq H,K\leq p$, and let $\Hbf=(H_1,...,H_n)$ and $\Kbf=(K_1,...,K_n)$ with $H\leq H_i\leq 2H$ and $K\leq K_i \leq 2K$ for all $i$. Then
    \begin{align}\label{eqn:keyestimate1}
        E'(\Hbf,\Kbf) &\leq C(\nbf,\Hbf)^{1/2}C(\nbf,\Kbf)^{1/2}.
    \end{align}
\end{prop}

\begin{proof}
The proof is similar to the steps in \S \ref{sec:burgess_redundancies}. 
We apply Cauchy-Schwarz to symmetrize the box in which $\xbf,\xbf',\ybf,\ybf'$ lie.
    Note that $\lambda_i^{1}(\xbf) \lambda_i^{4}(\ybf')\neq 0$ if and only if $\lambda_i^{1}(\xbf), \lambda_i^{2}(\xbf'),\lambda_i^{3}(\ybf), \lambda_i^{4}(\ybf')$ are all nonzero, and this nonzero restriction in $E'(\Hbf,\Kbf)$ allows us to rewrite $\lambda_i^{1}(\xbf)\lambda_i^{4}(\ybf')=\lambda_i^{2}(\xbf')\lambda_i^{3}(\ybf)$ in \eqref{eqn:simult_lam} as $\lambda_i^{1}(\xbf)\lambda_i^{2}(\xbf')^{-1}=\lambda_i^{3}(\ybf) \lambda_i^{4}(\ybf')^{-1}$ in $\bF_{p^{n_i}}$, and hence
    \begin{multline}
    \label{eqn:rewrite_E'}
    E'(\Hbf,\Kbf)= \sum_{\substack{z_1,...,z_s\\z_i\in \bF_{p^{n_i}}\setminus \{0\}}} |\{(\xbf,\xbf')\in D_{\Hbf}^{2}: \lambda_i^1(\xbf)= z_i \lambda_i^2(\xbf'),\ \lambda_i^1(\xbf), \lambda_i^2(\xbf')\neq 0\}| \\
        \cdot |\{(\ybf,\ybf')\in D_{\Kbf}^{2}: \lambda_i^3(\ybf)= z_i \lambda_i^4(\ybf'),\ \lambda_i^3(\ybf), \lambda_i^4(\ybf')\neq 0\}|.
    \end{multline}
    An application of Cauchy-Schwarz gives
    \begin{multline*}
    E'(\Hbf,\Kbf)
        \leq |\{(\xbf,\xbf',\ybf,\ybf')\in D_{\Hbf}^{4}:  \lambda_i^1(\xbf)\lambda_i^2(\ybf') = \lambda_i^2(\xbf')\lambda_i^1(\ybf)\}|^{1/2} \\
        \cdot |\{(\xbf,\xbf',\ybf,\ybf')\in D_{\Kbf}^{4}: \lambda_i^3(\ybf)\lambda_i^4(\xbf') = \lambda_i^3(\xbf)\lambda_i^4(\ybf')\}|^{1/2}.
    \end{multline*}
Then apply the definition of $C(\nbf,\Hbf)$ in \eqref{eqn:def_C} to summarize this as \eqref{eqn:keyestimate1}.
\end{proof}

\begin{prop}\label{prop:keyestimates} 
Assume the hypotheses of Proposition \ref{prop:keyestimates1}.
Then
    \begin{align}\label{eqn:keyestimate2}
    E''(\Hbf,\Kbf)&\ll_n \sum_{\tilde{\nbf}\prec \nbf} \max\{H^{2(n-|\tilde{\nbf}|)},K^{2(n-|\tilde{\nbf}|)}\}
    \max_{\substack{\tilde{\Hbf}\in \bZ^{|\tilde{\nbf}|}\\ \tilde{H}_i\in [H,2H]}}C(\tilde{\nbf},\tilde{\Hbf})^{1/2}
     \max_{\substack{\tilde{\Kbf}\in \bZ^{|\tilde{\nbf}|}\\ \tilde{K}_i\in [K,2K]}} C(\tilde{\nbf},\tilde{\Kbf})^{1/2}.
    \end{align}
\end{prop}

We begin by setting notation for all of the possible zero conditions in the definition of $E''(\Hbf,\Kbf)$.
Fix $1\leq r \leq s$, and suppose that precisely $r$ equations $\lambda_j^1(\xbf)\lambda_j^4(\ybf')=\lambda_j^2(\xbf')\lambda_j^3(\ybf)$ vanish. Let $R=\{i_1,...,i_r\}$ denote the index set of the vanishing equations, and let $m = \sum_{j\in R}n_j$. 
This implies that for each $j\in R$, at least two of $\lam_j^1(\xbf), \lam_j^2(\xbf'),\lam_j^3(\ybf),\lam_j^4(\ybf')$ must be zero.
    Let $R_1,R_2,R_3,R_4\subseteq R$ denote, respectively, the set of indices $\{j_1\}, \{j_2\},\{j_3\},\{j_4\}$ for which $\lambda_{j_1}^1(\xbf)=0,\lambda_{j_2}^2(\xbf')=0,\lambda_{j_3}^3(\ybf)=0$, and $\lambda_{j_4}^4(\ybf')=0$. For each $1\leq \ell\leq 4$, define $r_\ell=|R_\ell|$ and $m_\ell = \sum_{j\in R_\ell}n_j$.  
    Note that $\bigcup_{\ell=1}^4 R_\ell$ contains two copies of $R$, so
        \begin{align}\label{eqn:sum_r_rho}
            2r\leq \sum_{\ell=1}^4 r_\ell \leq 4r \quad \text{and} \quad 2m \leq \sum_{\ell=1}^4m_\ell \leq 4m.
        \end{align}
Going forward, fix the index sets $R, R_1,R_2,R_3,R_4$ (and hence $r, r_1,r_2,r_3,r_4$). Without loss of generality, let $R=\{1,...,r\}$ and $R_1 = \{1,...,r_1\}$ where $r_1\leq r$.

We first treat the $r=s$ case since it differs from the rest. This is the case where $\tilde{\nbf} = \mathbf{0}$ in the sum in \eqref{eqn:keyestimate2}. The condition $\lam^1_{i}(\xbf)=0$ for all $i\in R_1$ implies that
    \begin{align*}
        \xbf\in \ker 
        [U^1_i]_{1\leq i \leq r_1},
    \end{align*}
where we recall the notation of $U_i^\ell$ from \eqref{eqn:def_Aell}. In particular, the above matrix of $U_i^1$'s has rank $m_1$.
    By \eqref{eqn:LA02}, there are $\ll_n H^{n-m_1}$ choices of $\xbf$.
Repeat this for $\xbf',\ybf,\ybf'$ to see that the total number of choices for $(\xbf,\xbf',\ybf,\ybf')$ is $\ll_n H^{2n-m_1-m_2} K^{2n-m_3-m_4}$.
Since $r=s$, we have $m=n$, so $\sum_\ell m_\ell \geq 2n$ by \eqref{eqn:sum_r_rho}. Hence $ E''(\Hbf,\Kbf)\ll_n \max\{H,K\}^{2n}$, which is consistent with the claim in \eqref{eqn:keyestimate2} for the case $\tilde{\nbf}=\mathbf{0}$, with the convention $C(\mathbf{0},\Hbf)=1$.

    For the $r<s$ case, the proof is rather technical, so we begin with an overview of the strategy:

   (1) In the first step, we exploit the zero condition $\lambda_i^1(\xbf)=0$ for $1\leq i\leq r_1$
   to reduce the dimension of $\xbf$ by $m_1$, producing a new variable $\tilde{\xbf}=(x_{m_1+1},...,x_n)$. Then we write the remaining nonvanishing 
   $\lam_i^1(\xbf)$ where
   $r_1+1\leq i \leq s$, 
   in terms of $\tilde{\xbf}$, in matrix form, producing an $(n-m_1)\times (n-m_1)$ matrix $Q_1$.
    Repeat this for $\xbf',\ybf,\ybf'$ to get new variables $\tilde{\xbf}',\tilde{\ybf},\tilde{\ybf}'$ and matrices $Q_2,Q_3,Q_4$, where $Q_\ell$ has dimensions $(n-m_\ell)\times (n-m_\ell)$. 
            This step gives an initial reduction in dimension that is ``local'', in the sense that it derives from the separate $r_\ell$ zero conditions on $\lam_i^\ell$ for each $\ell$. 

(2) In the second step, we perform a ``global'' reduction in dimension, arising from the $r$ vanishing equations $\lambda_i^{1}(\xbf)\lambda_i^{4}(\ybf')=\lambda_i^{2}(\xbf')\lambda_i^{3}(\ybf)$, $1\leq i \leq r$. We discard these equations by truncating each of $Q_\ell$ to produce an $(n-m)\times (n-m_\ell)$ matrix $\tilde{Q}_\ell$.
    
(3)  In the third step, 
   we manipulate the remaining $s-r$ nonzero equations $  \lambda_i^{1}(\xbf)\lambda_i^{4}(\ybf')=\lambda_i^{2}(\xbf')\lambda_i^{3}(\ybf)$ for $r+1\leq i \leq s$, 
   to reduce the variables $\tilde{\xbf},\tilde{\xbf}',\tilde{\ybf},\tilde{\ybf}'$ (which, recall from step (1), have dimensions $n-m_1,n-m_2,n-m_3,n-m_4$ respectively) to new variables $\hat{\xbf},\hat{\xbf}',\hat{\ybf},\hat{\ybf}'$, all of which have dimension $n-m$. 
    At this point, we can pass to an analogous setting in a lower dimension $n-m$ and apply the definition of the uniform bound in \eqref{eqn:def_C} to obtain a recursive relation.

 Here we begin the first step.
 By definition, $\lambda_i^1(\xbf)=0$ for $1\leq i\leq r_1$ is equivalent to $L^1_{i,j}(\xbf)\equiv 0$ for $1\leq i \leq r_1$ and $1\leq j \leq n_i$: 
    \begin{align*}
    \begin{bmatrix}
        L_{1,1}^1(\xbf)\\
        \vdots\\
        L_{1,n_1}^1(\xbf)\\
        \vdots\\
        L_{r_1,1}^1(\xbf)\\
        \vdots\\
        L_{r_1,n_{r_1}}^1(\xbf)
    \end{bmatrix}
    \equiv
        \begin{bmatrix}
            a^1_{1,1,1}&\cdots& a^1_{1,1,n}\\
            \vdots&&\vdots\\
            a^1_{1,n_1,1}&\cdots & a^1_{1,n_1,n}\\
            \vdots&&\vdots\\
            a^1_{r_1,1,1}&\cdots& a^1_{r_1,1,n}\\
            \vdots&&\vdots\\
            a^1_{r_1,n_{r_1},1}&\cdots & a^1_{r_1,n_{r_1},n}\\
        \end{bmatrix}
        \begin{bmatrix}
            x_1\\
            \vdots\\
            x_n
        \end{bmatrix}
        \equiv 0.
    \end{align*}
Recall $m_1 = n_1+\cdots + n_{r_1}$.
Since $A_1$ has rank $n$, the above $m_1\times n$ submatrix of $A_1$ has rank $m_1$, and so it has an $m_1\times m_1$ submatrix $B_1$ with nonzero determinant. Without loss of generality, suppose $B_1$ is comprised of the left-most $m_1$ columns, and let $B_2$ denote the remaining $m_1\times (n-m_1)$ submatrix.
Then by the nonsingularity of $B_1$, we can express the variables $x_1,...,x_{m_1}$ in terms of the remaining variables $x_{m_1+1},...,x_n$:
    \begin{align}\label{eqn:pf_gen_2}
        \begin{bmatrix}
            x_1\\
            \vdots\\
            x_{m_1}
        \end{bmatrix} 
        \equiv -B_1^{-1}B_2 \begin{bmatrix}
            x_{m_1+1}\\
            \vdots\\
            x_n
        \end{bmatrix}.
    \end{align}
We may now rewrite the remaining linear forms $L^1_{r_1+1,1}(\xbf),..., L^1_{r_1+1,n_{r_1+1}}(\xbf), ..., L^1_{s,1}(\xbf), ..., L^1_{s,n_{s}}(\xbf)$ in terms of only the variables $x_{m_1+1},...,x_n$.
Indeed for some matrices $B_3, B_4$, we have
    \begin{align*}
        \begin{bmatrix}
            L^1_{r_1+1,1}(\xbf)\\
            \vdots\\
            L^1_{s,n_{s}}(\xbf)
        \end{bmatrix}
        &\equiv
        B_3
        \begin{bmatrix}
            x_1\\
            \vdots\\
            x_{m_1}
        \end{bmatrix} 
        +
        B_4
        \begin{bmatrix}
            x_{m_1+1}\\
            \vdots\\
            x_{n}
        \end{bmatrix}
         \equiv
        (-B_3B_1^{-1}B_2+B_4)  \begin{bmatrix}
            x_{m_1+1}\\
            \vdots\\
            x_{n}
        \end{bmatrix},
    \end{align*}
    where the second congruence holds by \eqref{eqn:pf_gen_2}. 
Let $\tilde{\xbf}:=(x_{m_1+1},...,x_n)$ and let $Q_1:=-B_3B_1^{-1}B_2+B_4$ denote this $(n-m_1)\times (n-m_1)$ matrix. 
Repeat this for $\xbf',\ybf,\ybf'$ to obtain $\tilde{\xbf}',\tilde{\ybf},\tilde{\ybf}'$, which respectively have dimensions $n-m_2, n-m_3, n-m_4$, and matrices $Q_2,Q_3,Q_4$.
The following lemma guarantees that $Q_\ell$ is nonsingular mod $p$ for each $1\leq \ell \leq 4$.
\begin{lemma}\label{lem:LA2}
    Let $A\in \GL_n(\bF_p)$. Let $1\leq s \leq n$ and write
    \begin{align*}
        A \equiv \begin{bmatrix}
            B_1 & B_2\\
            B_3 & B_4
        \end{bmatrix}
    \end{align*}
    where $B_1$ is an  $s\times s$ block with $\det B_1\not\equiv 0 \modd p$. Then $\det  (-B_3B_1^{-1}B_2+B_4)\not\equiv 0 \modd p$.
\end{lemma}

\begin{proof}         
Note that the matrix $-B_3 B_1^{-1}B_2+B_4$ has dimensions $(n-s)\times (n-s)$. Suppose by contradiction that it has determinant $0\modd p$. Then there exists a nonzero $\underline{\alpha}=(\alpha_1,...,\alpha_{n-s}) \in (\bF_p)^{n-s}$ such that $ (-B_3 B_1^{-1}B_2+B_4) 
        \underline{\alpha}
        \equiv \mathbf{0}$.
Let $\underline{\beta} = B_1^{-1}B_2 \underline{\alpha}\in (\bF_p)^{s}$,
so that rearranging the previous matrix equation gives $ B_4 \underline{\alpha} \equiv B_3 \underline{\beta}.$
Writing $\underline{\beta}, \underline{\alpha}$ as row vectors,
we then have $A (  \underline{\beta}, -\underline{\alpha})^T \equiv \mathbf{0}$,   
which contradicts the nonsingularity of $A$.     
\end{proof}

This completes the first step, where we utilized each vanishing relation in $\bF_{p^{n_i}}$ to lower the dimension by $n_i$, and in total by $m_1,m_2,m_3,m_4$ respectively for each variable $\xbf,\xbf',\ybf,\ybf'$.

We begin the second step, where we truncate each $Q_\ell$ to encode only the $s-r$ nonzero equations.
Indeed, the purpose of this step is to discard superfluous information. For example, the matrix $Q_1$ from the previous step represents 
the $s-r_1$ factors $\lam^1_{j_1}(\xbf)$ that do not vanish. However, in total, $s-r\leq s-r_1$ equations $\lam^1_j(\xbf)\lam^4_j(\ybf')=\lam_j^2(\xbf')\lam_j^3(\ybf)$ are nonzero, 
so we must eliminate from $Q_1$ data from the $r-r_1$ extra vanishing equations.

Thus let $\Qtil_1$ denote the $(n-m)\times (n-m_1)$ matrix obtained from $Q_1$ by discarding the $m-m_1$ rows that define the linear forms associated to $\lam^1_j(\xbf)$ for  $j\in \{1,...,r\}\setminus R_1=\{r_1+1,...,r\}$.
Note that $\Qtil_1$ has rank $n-m$.
For $r+1\leq i \leq s$ and $1\leq j \leq n_i$, define the linear forms $\Ltil^1_{i,j}(\tilde{\xbf})$ by 
\begin{align*}
        \begin{bmatrix}
            \Ltil^1_{i,1}(\tilde{\xbf})\\
            \vdots\\
            \Ltil^1_{i,n_{i}}(\tilde{\xbf})
        \end{bmatrix}
        := 
         \Util^1_i \tilde{\xbf}, \qquad \text{ where } 
              \begin{bmatrix}
            \Util^1_{r+1}\\
            \vdots\\
            \Util^1_s
        \end{bmatrix}:=\Qtil_1.
\end{align*}
Repeat this for $Q_2, Q_3, Q_4$ to obtain matrices $\Qtil_2,\Qtil_3,\Qtil_4$, where $\Qtil_\ell$ is $(n-m)\times (n-m_\ell)$, has rank $n-m$, and is comprised of submatrices $\Util^\ell_{r+1},...,\Util^\ell_s$.
Similarly, define linear forms $\Ltil^2_{i,j}(\tilde{\xbf}'), \Ltil^3_{i,j}(\tilde{\ybf}),\Ltil^4_{i,j}(\tilde{\ybf}')$ for $r+1\leq i \leq s$ and $1\leq j \leq n_i$. 
    This completes the second step.

In the third step, we shift focus to the remaining, nonvanishing equations $  \lambda_i^{1}(\xbf)\lambda_i^{4}(\ybf')=\lambda_i^{2}(\xbf')\lambda_i^{3}(\ybf)$ for $r+1 \leq i \leq s.$
 Let $\tilde{\lam}_i^\ell(\cdot) = \tilde{L}_{i,1}^\ell(\cdot) +\tilde{L}_{i,2}^\ell(\cdot)\omega_i + \cdots + \tilde{L}_{i,n_i}^\ell(\cdot) \omega_i^{n_i-1}$ for $1\leq \ell \leq 4$. Then we can rewrite these equations as
    \begin{align}\label{eqn:lastlemma}
        \tilde{\lambda}^1_i(\tilde{\xbf})\tilde{\lambda}^4_i(\tilde{\ybf}')
        =\tilde{\lambda}^2_i(\tilde{\xbf}')\tilde{\lambda}^3_i(\tilde{\ybf}),\qquad r+1 \leq i \leq s,
    \end{align}
    so that in this notation, we have reduced the energy to
        \begin{align}\label{eqn:lastlemma_energy}
            E''(\Hbf,\Kbf) = |\{(\tilde{\xbf},\tilde{\xbf}',\tilde{\ybf},\tilde{\ybf}')\in D_{\tilde{\Hbf}} \times D_{\tilde{\Hbf}'}\times D_{\tilde{\Kbf}}\times D_{\tilde{\Kbf}'}:\eqref{eqn:lastlemma}\}|
        \end{align}
        where $\tilde{\Hbf},\tilde{\Hbf}'$ denote the coordinates of $\Hbf$ that correspond respectively to $\tilde{\xbf},\tilde{\xbf}'$, and similarly define $\tilde{\Kbf},\tilde{\Kbf}'$.
    Our goal in this third step is to reduce the dimension of each of $\tilde{\xbf},\tilde{\xbf}',\tilde{\ybf},\tilde{\ybf}'$ to $n-m$, and we summarize this in the following lemma.
\begin{lemma}\label{lem:energy_k<n} Fix integers $k,\mu_1,\mu_2,\mu_3,\mu_4, \sigma \geq 1$ such that $k\leq \mu_\ell$ for all $\ell$. Fix $H,K\geq 1$ and let $\tilde{\Hbf}=(H_1,...,H_{\mu_1})$, $\tilde{\Hbf}'=(H_1',...,H'_{\mu_2})$, $\tilde{\Kbf}=(K_1,...,K_{\mu_3})$, $\tilde{\Kbf}'=(K'_1,...,K'_{\mu_4})$ with $H\leq H_i,H_i'\leq 2H$ and $K\leq K_i,K_i'\leq 2K$. Let $\kbf=(k_1,...,k_{\sigma})$ be a partition of $k$. For $1\leq \ell \leq 4$, let $\tilde{L}_{i,j}^\ell(\Xbf)\in \bF_{p}[X_1,...,X_{\mu_\ell}]$ be linear forms for $1\leq i \leq \sigma$ and $1\leq j \leq k_i$, and suppose the $k\times \mu_\ell$ matrices $\tilde{Q}_\ell$ whose rows are defined by the coefficients of $\tilde{L}^\ell_{i,j}$ have full rank $k$. Define $\tilde{\lam}_i^\ell(\cdot) = \tilde{L}_{i,1}^\ell(\cdot) +\tilde{L}_{i,2}^\ell(\cdot)\omega_i + \cdots + \tilde{L}_{i,k_i}^\ell(\cdot) \omega_i^{k_i-1}$, and define the energy
    \begin{align*}
        E=|\{(\tilde{\xbf},\tilde{\xbf}',\tilde{\ybf},\tilde{\ybf}')\in D_{\tilde{\Hbf}} \times D_{\tilde{\Hbf}'}\times D_{\tilde{\Kbf}}\times D_{\tilde{\Kbf}'}:\tilde{\lambda}^1_i(\tilde{\xbf})\tilde{\lambda}^4_i(\tilde{\ybf}')
        =\tilde{\lambda}^2_i(\tilde{\xbf}')\tilde{\lambda}^3_i(\tilde{\ybf}) \text{ for } 1 \leq i \leq \sigma\}|.
    \end{align*}
Then
    \begin{align*}
    E\ll_{k,\mu_1,\mu_2,\mu_3,\mu_4} 
    \max\{H,K\}^{\mu_1+\mu_2+\mu_3+\mu_4-4k}
   \max_{\substack{\hat{\Hbf}\in \bZ^{k}\\ \hat{H}_i\in [H,2H]}} C(\kbf,\hat{\Hbf})^{1/2} \max_{\substack{\hat{\Kbf}\in \bZ^{k}\\\hat{K}_i\in [K,2K]}} C(\kbf,\hat{\Kbf})^{1/2}.
    \end{align*}    
\end{lemma}
    \begin{proof}
For $1\leq \ell \leq 4$, let $\tilde{U}^\ell_i$ where $1\leq i \leq \sigma$ define the submatrices of $\tilde{Q}_\ell$, i.e. $\tilde{Q}_\ell = [\tilde{U}^\ell_i]_{1\leq i \leq \sigma}$.
 We start by reducing the dimensions of $\tilde{\xbf}$ and $\tilde{\xbf}'$. By Lemma \ref{lem:B_multiplication}, for each $i$, we get a system of linear equations mod $p$ of the form
     \begin{align*}
 M(k_i;\Util^4_i\tilde{\ybf}';\omega_i)
 \tilde{U}_i^1\tilde{\xbf}
            \equiv 
            M(k_i;\Util^3_i\tilde{\ybf}; \omega_i)
            \tilde{U}_i^2\tilde{\xbf}',
    \end{align*}
    where $M(k_i;\Util^4_i\tilde{\ybf}';\omega_i)$ and $M(k_i;\Util^3_i\tilde{\ybf}; \omega_i)$ are as defined in \eqref{eqn:M_multiplication_matrix}.
    Let $M_{\Qtil_4 \tilde{\ybf}'}$ and $M_{\Qtil_3 \tilde{\ybf}}$  denote the block matrices, respectively, with blocks $M(k_i;\Util^4_i\tilde{\ybf}';\omega_i)$ for $1\leq i\leq \sigma$ and $M(k_i;\Util^3_i\tilde{\ybf}; \omega_i)$ for $1\leq i\leq \sigma$.
Combine the linear systems for all $1\leq i\leq \sigma$ to get $M_{\Qtil_4\tilde{\ybf}'}
        \Qtil_1 \tilde{\xbf}   \equiv 
            M_{\Qtil_3 \tilde{\ybf}}
            \Qtil_2 \tilde{\xbf}',$
which is further equivalent to
    \begin{align*}
          \begin{bmatrix}
            M_{\Qtil_4 \tilde{\ybf}'} &-M_{\Qtil_3 \tilde{\ybf}} 
        \end{bmatrix}
            \begin{bmatrix}
                \Qtil_1&0\\
                0&\Qtil_2
            \end{bmatrix}
         \begin{bmatrix}
             \tilde{\xbf}\\
             \tilde{\xbf}'
            \end{bmatrix}
            \equiv 0.
    \end{align*}
For each $\ell=1,2$, $\Qtil_\ell$ has rank $k$, so there exists an $k\times k$ submatrix $\hat{Q}_\ell$ with nonzero determinant. Hence from the $\mu_1$-dimensional tuple $\tilde{\xbf}$, we can extract the $k$ coordinates that correspond to the columns of $\hat{Q}_1$.
Denote these $k$ coordinates by $\hat{\xbf}$ and the remaining $\mu_1-k$ coordinates by $\check{\xbf}$. Repeat this for $\tilde{\xbf}'$ to get $\hat{\xbf}'$ and $\check{\xbf}'$. 
We proceed with the aim of removing $\check{\xbf}$ and $\check{\xbf}'$ from the matrix estimate.
First move $\check{\xbf}$ and $\check{\xbf}'$ to the right-hand side:
        \begin{align}\label{eqn:before_coset}
          \begin{bmatrix}
            M_{\Qtil_4 \tilde{\ybf}'} &-M_{\Qtil_3 \tilde{\ybf}} 
        \end{bmatrix}
            \begin{bmatrix}
                \hat{Q}_1&0\\
                0&\hat{Q}_2
            \end{bmatrix}
         \begin{bmatrix}
             \hat{\xbf}\\
             \hat{\xbf}'
            \end{bmatrix}
            \equiv -
        \begin{bmatrix}
            M_{\Qtil_4 \tilde{\ybf}'} &-M_{\Qtil_3 \tilde{\ybf}} 
        \end{bmatrix}
            \begin{bmatrix}
                \check{Q}_1&0\\
                0&\check{Q}_2
            \end{bmatrix}
         \begin{bmatrix}
             \check{\xbf}\\
             \check{\xbf}'
            \end{bmatrix}.
    \end{align}
Let $\hat{\Hbf}$ and $\check{\Hbf}$ denote the coordinates of $\Hbf$ that correspond respectively to $\hat{\xbf}$ and $\check{\xbf}$.
In this notation, we then record the reduction of $E$ to a count with fewer variables:
    \begin{align}\label{eqn:E''_1}
        E
        =\sum_{(\check{\xbf}, \check{\xbf}')\in D_{\check{\Hbf}}\times D_{\check{\Hbf}'}} 
        \sum_{(\tilde{\ybf}, \tilde{\ybf}')\in D_{\tilde{\Kbf}}\times D_{\tilde{\Kbf}'}}
        |\{ (\hat{\xbf}, \hat{\xbf}')\in D_{\hat{\Hbf}}\times D_{\hat{\Hbf}'}: \eqref{eqn:before_coset}\}|.
    \end{align}
Fix $\check{\xbf},\check{\xbf}',\tilde{\ybf},\tilde{\ybf}'$ such that there exists a solution $(\hat{\xbf},\hat{\xbf}')$ to \eqref{eqn:before_coset}. (If there is no solution, then the count is simply zero.) By Lemma \ref{lem:linear_algebra} applied to \eqref{eqn:before_coset} with
    \begin{align*}
    M 
    =         \begin{bmatrix}
            M_{\Qtil_4 \tilde{\ybf}'} &-M_{\Qtil_3 \tilde{\ybf}} 
        \end{bmatrix}
            \begin{bmatrix}
                \hat{Q}_1&0\\
                0&\hat{Q}_2
            \end{bmatrix}
            \quad \text{and} \quad
    b
    = -
           \begin{bmatrix}
            M_{\Qtil_4 \tilde{\ybf}'} &-M_{\Qtil_3 \tilde{\ybf}} 
        \end{bmatrix}
            \begin{bmatrix}
                \check{Q}_1&0\\
                0&\check{Q}_2
            \end{bmatrix}
         \begin{bmatrix}
             \check{\xbf}\\
             \check{\xbf}'
            \end{bmatrix},
    \end{align*}
we get
    \begin{align*}
        |\{(\hat{\xbf}, \hat{\xbf}')\in D_{\hat{\Hbf}}\times D_{\hat{\Hbf}'}: M \begin{bmatrix}
            \hat{\xbf}\\
            \hat{\xbf}'
        \end{bmatrix}\equiv b\}| 
        \ll_k |\{(\hat{\xbf}, \hat{\xbf}')\in D_{\hat{\Hbf}}\times D_{\hat{\Hbf}'}: M \begin{bmatrix}
            \hat{\xbf}\\
            \hat{\xbf}'
        \end{bmatrix}\equiv 0\}|.
    \end{align*}
    Rearrange the matrix equation on the right-hand side to get $M_{\Qtil_4\tilde{\ybf}'}
        \hat{Q}_1 \hat{\xbf}   \equiv 
            M_{\Qtil_3 \tilde{\ybf}}
            \hat{Q}_2 \hat{\xbf}'$,
where recall $\hat{Q}_1,\hat{Q}_2$ are both $k\times k$. Define the linear forms
    \begin{align*}
        \begin{bmatrix}
\hat{L}^1_{i,1}(\hat{\xbf})\\
\vdots\\
\hat{L}^1_{i,k_i}(\hat{\xbf})
        \end{bmatrix}
        := \hat{U}^1_i \hat{\xbf}, \ \text{where}
            \begin{bmatrix}
            \hat{U}^1_{1}\\
            \vdots\\
            \hat{U}^1_{\sigma}
        \end{bmatrix}:=\hat{Q}_1
\quad \text{and}\quad
         \begin{bmatrix}
\hat{L}^2_{i,1}(\hat{\xbf}')\\
\vdots\\
\hat{L}^2_{i,k_i}(\hat{\xbf}')
        \end{bmatrix}
        := \hat{U}^2_i \hat{\xbf}', \ \text{where}
            \begin{bmatrix}
            \hat{U}^2_{1}\\
            \vdots\\
            \hat{U}^2_{\sigma}
        \end{bmatrix}:=\hat{Q}_2 . 
    \end{align*}
By Lemma \ref{lem:B_multiplication}, the above matrix equation is equivalent to the equations
 \begin{align}\label{eqn:after_coset}
        \hat{\lambda}^1_i(\hat{\xbf})\tilde{\lambda}^4_i(\tilde{\ybf}')
        =\hat{\lambda}^2_i(\hat{\xbf}')\tilde{\lambda}^3_i(\tilde{\ybf}), \qquad 1\leq i \leq \sigma,
    \end{align}
where $\hat{\lam}_i^\ell(\cdot) = \hat{L}_{i,1}^\ell(\cdot) +\hat{L}_{i,2}^\ell(\cdot)\omega_i + \cdots + \hat{L}_{i,k_i}^\ell(\cdot) \omega_i^{k_i-1}$ for $\ell=1,2$.
Then, continuing \eqref{eqn:E''_1},
    \begin{align}\label{eqn:E''_2}
            E&\ll_k  \sum_{ (\check{\xbf}, \check{\xbf}')\in 
            D_{\check{\Hbf}}\times D_{\check{\Hbf}'}}
            \sum_{(\tilde{\ybf}, \tilde{\ybf}')\in 
            D_{\tilde{\Kbf}}\times D_{\tilde{\Kbf}'}} 
            |\{(\hat{\xbf}, \hat{\xbf}')\in 
            D_{\hat{\Hbf}}\times D_{\hat{\Hbf}'}:
            \eqref{eqn:after_coset}\}|\nonumber \\
    &\ll_{k,\mu_1,\mu_2} H^{\mu_1+\mu_2-2k} |\{(\tilde{\ybf}, \tilde{\ybf}',\hat{\xbf}, \hat{\xbf}')\in 
    D_{\tilde{\Kbf}}\times D_{\tilde{\Kbf}'} \times D_{\hat{\Hbf}} \times D_{\hat{\Hbf}'}:  
    \eqref{eqn:after_coset}\}|,
    \end{align}
    where in the second step, we have bounded $\check{\xbf},\check{\xbf}'$ (which have dimensions $\mu_1-k$ and $\mu_2-k$ respectively) trivially since they do not appear in \eqref{eqn:after_coset}. 
This completes the first half of the third step, where we have reduced the dimensions of each of $\tilde{\xbf}$ and $\tilde{\xbf}'$ to $k$.

Now we repeat the above for $\tilde{\ybf}$ and $\tilde{\ybf}'$ with respect to the new equations \eqref{eqn:after_coset}. 
Again by Lemma \ref{lem:B_multiplication}, for each equation in \eqref{eqn:after_coset} we get a system of linear equations mod $p$ of the form
     \begin{align*}
       M(k_i;\hat{U}^1_i\hat{\xbf};\omega_i) 
       \Util_i^4 \tilde{\ybf}'
            \equiv M(k_i;\hat{U}^2_i\hat{\xbf}';\omega_i)
            \Util_i^3 \tilde{\ybf}.
    \end{align*}
As before, let $ M_{\hat{Q}_1 \hat{\xbf}}$ and $M_{\hat{Q}_2 \hat{\xbf}'}$ denote the appropriate block matrices,
then combine the linear systems for all $1\leq i \leq \sigma$ to get
    \begin{align*}
        \begin{bmatrix}
            M_{\hat{Q}_1 \hat{\xbf}} & - M_{\hat{Q}_2 \hat{\xbf}'}
        \end{bmatrix}
            \begin{bmatrix}
                \Qtil_4&0\\
                0&\Qtil_3
            \end{bmatrix}
         \begin{bmatrix}
             \tilde{\ybf}'\\
             \tilde{\ybf}
            \end{bmatrix}
            \equiv 0.
    \end{align*}
For $\ell = 3,4$, $\Qtil_\ell$ has rank $k$, so there exists an $k\times k$ submatrix $\hat{Q}_\ell$ with nonzero determinant. Hence from $\tilde{\ybf}$ we can again extract the appropriate $k$ coordinates $\hat{\ybf}$ and the remaining $\mu_3-k$ coordinates $\check{\ybf}$. Repeat this for $\tilde{\ybf}'$ to get $\hat{\ybf}'$  and $\check{\ybf}'$. 
Move $\check{\ybf}',\check{\ybf}$ to the right-hand side to get
        \begin{align}\label{eqn:before_coset2}
         \begin{bmatrix}
            M_{\hat{Q}_1 \hat{\xbf}} & - M_{\hat{Q}_2 \hat{\xbf}'}
        \end{bmatrix}
            \begin{bmatrix}
                \hat{Q}_4&0\\
                0&\hat{Q}_3
            \end{bmatrix}
         \begin{bmatrix}
             \hat{\ybf}'\\
             \hat{\ybf}
            \end{bmatrix}
            \equiv -
         \begin{bmatrix}
            M_{\hat{Q}_1 \hat{\xbf}} & - M_{\hat{Q}_2 \hat{\xbf}'}
        \end{bmatrix}
            \begin{bmatrix}
                \check{Q}_4&0\\
                0&\check{Q}_3
            \end{bmatrix}
         \begin{bmatrix}
             \check{\ybf}'\\
             \check{\ybf}
            \end{bmatrix}.
    \end{align}
Then, continuing \eqref{eqn:E''_2},
    \begin{align}\label{eqn:E''_3}
      E
        &\ll_{k,\mu_1,\mu_2} H^{\mu_1+\mu_2-2k} \sum_{(\hat{\xbf},\hat{\xbf}')\in D_{\hat{\Hbf}}\times D_{\hat{\Hbf}'}}
        \sum_{(\check{\ybf},\check{\ybf}')\in D_{\check{\Kbf}}\times D_{\check{\Kbf}'}}
        |\{(\hat{\ybf},\hat{\ybf}')\in D_{\hat{\Kbf}}\times D_{\hat{\Kbf}'}:\eqref{eqn:before_coset2} \}|.
    \end{align}
Fix $\hat{\xbf},\hat{\xbf}',\check{\ybf},\check{\ybf}'$ such that there exists a solution $(\hat{\ybf},\hat{\ybf}')$ to \eqref{eqn:before_coset2}. (If there is no solution, then the count is zero.) By Lemma \ref{lem:linear_algebra} applied to \eqref{eqn:before_coset2} with
    \begin{align*}
    M 
    =         \begin{bmatrix}
            M_{\hat{Q}_1 \hat{\xbf}} & - M_{\hat{Q}_2 \hat{\xbf}'}
        \end{bmatrix}
            \begin{bmatrix}
                \hat{Q}_4&0\\
                0&\hat{Q}_3
            \end{bmatrix} \quad \text{and} \quad
    b
    = -
         \begin{bmatrix}
            M_{\hat{Q}_1 \hat{\xbf}} & - M_{\hat{Q}_2 \hat{\xbf}'}
        \end{bmatrix}
            \begin{bmatrix}
                \check{Q}_4&0\\
                0&\check{Q}_3
            \end{bmatrix}
         \begin{bmatrix}
             \check{\ybf}'\\
             \check{\ybf}
            \end{bmatrix},
    \end{align*}
we get
    \begin{align*}
        |\{(\hat{\ybf}, \hat{\ybf}')\in D_{\hat{\Kbf}}\times D_{\hat{\Kbf}'}: M \begin{bmatrix}
            \hat{\ybf}'\\
            \hat{\ybf}
        \end{bmatrix}\equiv b\}| 
        \ll_k |\{(\hat{\ybf}, \hat{\ybf}')\in D_{\hat{\Kbf}}\times D_{\hat{\Kbf}'}: M \begin{bmatrix}
            \hat{\ybf}'\\
            \hat{\ybf}
        \end{bmatrix}\equiv 0\}|.
    \end{align*}
Rearrange the matrix equation on the right-hand side to get $M_{\hat{Q}_1\hat{\xbf}}
        \hat{Q}_4 \hat{\ybf}'   \equiv 
            M_{\hat{Q}_2 \hat{\xbf}'}
            \hat{Q}_3 \hat{\ybf}$, and
define the linear forms
    \begin{align*}
        \begin{bmatrix}
\hat{L}^3_{1,1}(\hat{\ybf})\\
\vdots\\
\hat{L}^3_{\sigma,k_\sigma}(\hat{\ybf})
        \end{bmatrix}
        := \hat{Q}_3 \hat{\ybf} \quad
        \text{and} \quad
         \begin{bmatrix}
\hat{L}^4_{1,1}(\hat{\ybf}')\\
\vdots\\
\hat{L}^4_{\sigma,k_\sigma}(\hat{\ybf}')
        \end{bmatrix}
        := \hat{Q}_4 \hat{\ybf}'. 
    \end{align*}
By Lemma \ref{lem:B_multiplication}, the above matrix equation is equivalent to the equations
 \begin{align}\label{eqn:after_coset2}
\hat{\lambda}^1_i(\hat{\xbf})\hat{\lambda}^4_i(\hat{\ybf}')
        =\hat{\lambda}^2_i(\hat{\xbf}')\hat{\lambda}^3_i(\hat{\ybf}), \qquad 1\leq i \leq \sigma,
    \end{align}
where $\hat{\lam}_i^\ell(\cdot) = \hat{L}_{i,1}^\ell(\cdot) +\hat{L}_{i,2}^\ell(\cdot)\omega_i + \cdots + \hat{L}_{i,k_i}^\ell(\cdot) \omega_i^{k_i-1}$.
Then, continuing \eqref{eqn:E''_3},
    \begin{align*}
    E
            &\ll_{k,\mu_1,\mu_2} H^{\mu_1+\mu_2-2k} \sum_{(\hat{\xbf},\hat{\xbf}')\in D_{\hat{\Hbf}}\times D_{\hat{\Hbf}'}}
    \sum_{(\check{\ybf},\check{\ybf}')\in D_{\check{\Kbf}}\times D_{\check{\Kbf}'}}|\{(\hat{\ybf}',\hat{\ybf})\in D_{\hat{\Kbf}}\times D_{\hat{\Kbf}'}:\eqref{eqn:after_coset2} \}|\\
    &\ll_{k,\mu_1,\mu_2,\mu_3,\mu_4} H^{\mu_1+\mu_2-2k} K^{\mu_3+\mu_4-2k} |\{(\hat{\xbf}, \hat{\xbf}',\hat{\ybf}, \hat{\ybf}')\in D_{\hat{\Hbf}}\times D_{\hat{\Hbf}'}\times D_{\hat{\Kbf}}\times D_{\hat{\Kbf}'}:  \eqref{eqn:after_coset2}\}|, 
    \end{align*} 
where in the second step, we have bounded $\check{\ybf},\check{\ybf}'$ (which have dimensions $\mu_3-k$ and $\mu_4-k$ respectively) trivially since they do not appear in \eqref{eqn:after_coset2}. 
      
    At this point, we have reduced $E$ to a count of tuples $(\hat{\xbf}, \hat{\xbf}',\hat{\ybf}, \hat{\ybf}')$, each in $k$ dimensions, that satisfy simultaneous nonvanishing equations over finite field extensions that together are defined by $k$ linear forms.
By two applications of Cauchy-Schwarz as in Proposition \ref{prop:keyestimates1},
    \begin{align*}
        E
    \ll_{k,\mu_1,\mu_2,\mu_3,\mu_4} \max\{H,K\}^{\mu_1+\mu_2+\mu_3+\mu_4-2k}  
    C(\kbf,\hat{\Hbf})^{1/4}C(\kbf,\hat{\Hbf}')^{1/4}C(\kbf,\hat{\Kbf})^{1/4}C(\kbf,\hat{\Kbf}')^{1/4}.
    \end{align*}
To account for all possible choices of coordinates for each of $\hat{\xbf},\hat{\xbf}',\hat{\ybf},\hat{\ybf}'$, we take the maximum over all such possibilities:
  \begin{align*}
    E
    \ll_{k,\mu_1,\mu_2,\mu_3,\mu_4} \max\{H,K\}^{\mu_1+\mu_2+\mu_3+\mu_4-2k}  
\max_{\substack{\hat{\Hbf}\in \bZ^{k}\\ \hat{H}_i\in[H,2H]}}C(\kbf,\hat{\Hbf})^{1/2}\max_{\substack{\hat{\Kbf}\in \bZ^{k}\\ \hat{K}_i\in[K,2K]}} C(\kbf,\hat{\Kbf})^{1/2}.
    \end{align*}
This completes the proof of the lemma.
       \end{proof} 
Now apply this Lemma \ref{lem:energy_k<n} to \eqref{eqn:lastlemma_energy} with $k=n-m, \mu_\ell = n-m_\ell$, and $\sigma=s-r$. Note that $\mu_1+\mu_2+\mu_3+\mu_4-4k =4m - (m_1+\cdots + m_4)$, and by \eqref{eqn:sum_r_rho}, $4m - (m_1+\cdots + m_4)\leq 2m =2(n_1+\cdots + n_r)=2(n-(n_{r+1}+\cdots +n_s))$, so $\max\{H,K\}^{4m-(m_1 +\cdots+m_4)} \leq \max\{H,K\}^{2(n-(n_{r+1}+\cdots+n_s))}$.
Define $\nbf_{1,...,r}$ to be the vector obtained from $\nbf=(n_1,...,n_s)$ by setting $n_j$ to 0 for $1\leq j \leq r$. Then from the lemma, we obtain
  \begin{align*}
    E''(\Hbf,\Kbf)
    \ll_n \max\{H,K\}^{2(n-|\nbf_{1,...,r}|)}
\max_{\substack{\hat{\Hbf}\in \bZ^{|\nbf_{1,...,r}|}\\ \hat{H}_i\in[H,2H]}}C(\nbf_{1,...,r},\hat{\Hbf})^{1/2}\max_{\substack{\hat{\Kbf}\in \bZ^{|\nbf_{1,...,r}|}\\ \hat{K}_i\in[K,2K]}} C(\nbf_{1,...,r},\hat{\Kbf})^{1/2}.
    \end{align*}
Hence for an arbitrary choice of $R=\{i_1,...,i_r\}$, with $\nbf_{i_1,...,i_r}$ defined to be the vector obtained from $\nbf$ by setting $n_j$ to 0 for $j\in R$,
 \begin{align*}
        E''(\Hbf,\Kbf)\ll_n \max\{H,K\}^{2(n-|\nbf_{i_1,...,i_r}|)} \max_{\substack{\tilde{\Hbf}\in \bZ^{|\nbf_{i_1,...,i_r}|}\\ H_i\in[H,2H]}}C(\nbf_{i_1,...,i_r},\tilde{\Hbf})^{1/2}
        \max_{\substack{\tilde{\Kbf}\in \bZ^{|\nbf_{i_1,...,i_r}|}\\ K_i\in[K,2K]}}
        C(\nbf_{i_1,...,i_r},\tilde{\Kbf})^{1/2}.
    \end{align*}
Finally, for each $1\leq \ell\leq 4$, there are $r\leq n$ choices of $r_\ell$ and $\binom{r}{r_\ell}<2^r \leq 2^n$ choices of the set $R_\ell$. Recalling the notation $\tilde{\nbf}\prec \nbf$, we combine these possibilities and summarize the bound as
    \begin{align*}
      E''(\Hbf,\Kbf) 
        \ll_n \sum_{\tilde{\nbf}\prec \nbf}
         \max\{H,K\}^{2(n-|\tilde{\nbf}|)} 
          \max_{\substack{\tilde{\Hbf}\in \bZ^{|\tilde{\nbf}|}\\ H_i\in[H,2H]}}
          C(\tilde{\nbf},\tilde{\Hbf})^{1/2}
        \max_{\substack{\tilde{\Kbf}\in \bZ^{|\tilde{\nbf}|}\\ K_i\in[K,2K]}}
        C(\tilde{\nbf},\tilde{\Kbf})^{1/2}.
    \end{align*}
This completes the proof of Proposition \ref{prop:keyestimates}.

\subsection{A recursive formula for $C(\nbf,\Hbf)$ and proof of Theorem \ref{thm:main_energy}}\label{sec:recursive_formula}
We state and prove the main recursion relation that will lead to Theorem \ref{thm:main_energy}. As previously advertised, the first term on the right-hand side below describes savings from reducing the box size (represented by $\Hbf$), while the second term describes savings from reducing the dimension (represented by $\nbf$).

\begin{prop}\label{prop:recursive} Fix a prime $p$ and an integer $n\geq 1$. Fix $s\geq 1$ and a partition $\nbf = (n_1,...,n_s)$ of $n$. Let $\kappa>0$, and let $H$ be an integer such that $p^\kappa\leq H \leq p^{1/2}$. Let $\Hbf = (H_1,...,H_n)$ such that $H\leq H_i\leq 2H$ for all $i$. Let $C(\nbf,\Hbf)$ be as defined in \eqref{eqn:def_C}.  Then
    \begin{multline}\label{eqn:prop_recursive}
       C(\nbf,\Hbf)\leq C_1 H^{2n} (H^{-2n}C(\nbf,\Hbf))^{1/2} \max_{\substack{\Hbf^{(1)}\in \bZ^{|\nbf|}\\ H^{(1)}_i \in [H^{(1)},2H^{(1)}]}} ((H^{(1)})^{-2n}C(\nbf,\Hbf^{(1)}))^{1/2} \\
       + C_2 H^{2n} p^{|\nbf| \kappa} \sum_{\nbf^{(1)} \prec \nbf} \max_{\substack{\tilde{H}\leq H}}\max_{\substack{\tilde{\Hbf}\in \bZ^{|\nbf^{(1)}|}\\ \tilde{H}_i\in [\tilde{H}, 2\tilde{H}]}} \tilde{H}^{-2|\nbf^{(1)}|}C(\nbf^{(1)},\tilde{\Hbf}),
    \end{multline}
where $H^{(1)} = \lfloor p^{-\kappa}H \rfloor$, $C_1=C_1(n)$, and $C_2=C_2(n)$.
\end{prop}

We begin the proof with the decomposition
    \begin{align*}
        E(\Hbf,\Hbf)=E'(\Hbf,\Hbf)+E''(\Hbf,\Hbf),
    \end{align*}
    where, immediately, we can apply Proposition \ref{prop:keyestimates} to get
    \begin{align}\label{eqn:E''(H,H)}
        E''(\Hbf,\Hbf)\ll_n \sum_{\tilde{\nbf}\prec \nbf} H^{2(n-|\tilde{\nbf}|)} \max_{\substack{\tilde{\Hbf}\in \bZ^{|\tilde{\nbf}|}\\ \tilde{H}_i\in [H,2H]}} C(\tilde{\nbf},\tilde{\Hbf}).
    \end{align}
Estimating $E'(\Hbf,\Hbf)$ is more involved. A priori, we can apply Proposition \ref{prop:keyestimates1} with $K=H$, but to achieve a nontrivial bound, it is crucial that we apply it when $K<H$. To introduce such a term, we apply results from the geometry of numbers, as prepared in \S \ref{sec:lattices}.
For each $\zbf=(z_1,...,z_s)$ with $z_i\in \bF_{p^{n_i}}\setminus\{0\}$, define the lattice
    \begin{align*}
        \Lcal_{\zbf}=\{(\ybf,\ybf')\in \bZ^{2n}: \lambda^3_i(\ybf) = z_i \lambda^4_i(\ybf') \text{ for } 1\leq i \leq s\}.
    \end{align*}
Then, by rewriting the energy as we have done in \eqref{eqn:rewrite_E'},
    \begin{align*}
        E'(\Hbf,\Hbf)\leq
        \sum_{\substack{z_1,...,z_s\\z_i\in \bF_{p^{n_i}}\setminus \{0\}}} |\{(\xbf,\xbf')\in 
        D_\Hbf^2:
        \lambda^1_i(\xbf)=z_i\lambda^2_i(\xbf')
        \text{ and }  \lambda^1_i(\xbf) \lambda^2_i(\xbf')\neq 0 \text{ for all } i\}|
        \cdot |B_\Hbf^{2}\cap \Lcal_\zbf|.
    \end{align*}
To proceed, recall the definition of $s=s(\zbf)$ from \eqref{eqn:def_s}. We separate the $\zbf$'s such that $s(\zbf)\leq n$ from the $\zbf$'s such that $n< s(\zbf)\leq 2n$. Let $Z_1, Z_2$ denote these respective sets, and let $\Sigma_1, \Sigma_2$ denote their respective sums.

\subsubsection{The case $s(\zbf)\leq n$} Fix $\zbf$ such that $s(\zbf)\leq n$, and fix $\Hbf^{(1)}=(H^{(1)}_1,...,H^{(1)}_n)$ with $H^{(1)}_i\in [H^{(1)},2H^{(1)}]$, where $H^{(1)}\leq H$ will be chosen later to be $\lfloor p^{-\kappa}H\rfloor$. Then by \eqref{eqn:keyfact2}, 
    \begin{align*}
      |B_\Hbf^{2}\cap \Lcal_\zbf| 
        \ll_n (\frac{H}{H^{(1)}})^{s(\zbf)}|B_{\Hbf^{(1)}}^{2}\cap \Lcal_\zbf|,
    \end{align*}
so that
    \begin{align}\label{eqn:sigma1}
        \Sigma_1
        \ll_n 
        (\frac{H}{H^{(1)}})^n \sum_{\zbf\in Z_1} |\{(\xbf,\xbf')\in 
        D_\Hbf^2:
        \lambda^1_i(\xbf)=z_i\lambda^2_i(\xbf') 
        \text{ and }  \lambda^1_i(\xbf)\lambda^2_i(\xbf')\neq 0 \text{ for all } i\}| \cdot |B_{\Hbf^{(1)}}^{2}\cap \Lcal_\zbf|.
    \end{align}
To apply Proposition \ref{prop:keyestimates}, convert the sum on the right-hand side back into a count of the tuples $(\xbf,\xbf',\ybf,\ybf')$, recalling the equations $\lambda^1_i(\xbf)\lambda^4_i(\ybf')=\lambda^2_i(\xbf')\lambda^3_i(\ybf)$ for $1\leq i \leq s$ in \eqref{eqn:simult_lam}:
    \begin{align*}
        \Sigma_1 \ll_n 
        (\frac{H}{H^{(1)}})^n
        |\{(\xbf,\xbf',\ybf,\ybf')\in 
         D_\Hbf^2 \times D^2_{\Hbf^{(1)}}: 
        \eqref{eqn:simult_lam}      
        \text{ and }  \lambda^1_i(\xbf)\lambda^2_i(\xbf')\neq 0 \text{ for all } i\}|.
    \end{align*}
Further separate this count into those $(\xbf,\xbf',\ybf,\ybf')$ for which $\lambda^3_i(\ybf) \lambda^4_i(\ybf')\neq 0$ for all $i$, and those for which $\lambda^3_j(\ybf)\lambda^4_j(\ybf')= 0$ for at least one $j$, that is,
    \begin{align*}
        \Sigma_1\ll_n 
        (\frac{H}{H^{(1)}})^n (\Sigma_{1}' + \Sigma_{1}''),
    \end{align*}
where
    \begin{align*}
        \Sigma_{1}'= 
        |\{(\xbf,\xbf',\ybf,\ybf')\in
        D_\Hbf^2 \times D^2_{\Hbf^{(1)}}:
        \eqref{eqn:simult_lam},\ 
   \lambda^1_i(\xbf)\lambda^2_i(\ybf')\neq 0 \text{ for all } i, \ \lambda^3_i(\ybf) \lambda^4_i(\ybf')\neq 0  \text{ for all } i\}|
    \end{align*}
and
    \begin{multline*}
        \Sigma_{1}'' = 
        |\{(\xbf,\xbf',\ybf,\ybf')
        \in D_\Hbf^2 \times D^2_{\Hbf^{(1)}}: 
       \eqref{eqn:simult_lam},\    
     \lambda^1_i(\xbf) \lambda^2_i(\xbf')\neq 0 \text{ for all } i,\ 
    \lambda^3_{j}(\ybf)\lambda^4_{j}(\ybf')=0 \text{ for some } j\}|.
    \end{multline*}
Notice that $\Sigma_{1}'=E'(\Hbf,\Hbf^{(1)})$ while $\Sigma_{1}''\leq E''(\Hbf,\Hbf^{(1)})$, so
 \begin{align*}
     \Sigma_1\ll_n (\frac{H}{H^{(1)}})^n (E'(\Hbf,\Hbf^{(1)})+E''(\Hbf,\Hbf^{(1)})).
 \end{align*}
 Apply Proposition \ref{prop:keyestimates1} with $K=H^{(1)}$ to $E'(\Hbf,\Hbf^{(1)})$ and Proposition \ref{prop:keyestimates} with $K=H^{(1)}$ to $E''(\Hbf,\Hbf^{(1)})$, to get
\begin{multline}\label{eqn:case_s(z)<n}
    \Sigma_1 \ll_n (\frac{H}{H^{(1)}})^n[C(\nbf,\Hbf)^{1/2}C(\nbf,\Hbf^{(1)})^{1/2} 
    + [ \sum_{\tilde{\nbf}\prec \nbf} \max\{H^{2(n-|\tilde{\nbf}|)},(H^{(1)})^{2(n-|\tilde{\nbf}|)}\}   \\
    \max_{\substack{\tilde{\Hbf}\in \bZ^{|\tilde{\nbf}|}\\ \tilde{H}_i\in [H,2H]}} C(\tilde{\nbf},\tilde{\Hbf})^{1/2}
       \max_{\substack{\tilde{\Hbf}^{(1)}\in \bZ^{|\tilde{\nbf}|}\\ \tilde{H}_i^{(1)}\in [H^{(1)},2H^{(1)}]}}C(\tilde{\nbf},\tilde{\Hbf}^{(1)})^{1/2}]].
\end{multline}

\subsubsection{The case $n< s(\zbf)\leq 2n$}\label{sec:s(z)>n}
Fix $\zbf$ such that $n< s(\zbf)\leq 2n$, and fix $\Hbf^{(2)}=(H^{(2)}_1,...,H^{(2)}_n)$ with $H^{(2)}_i\in [H^{(2)},2H^{(2)}]$, for some
    \begin{align}\label{eqn:condition:H_2_OLD}
        H^{(2)}\leq p/H
    \end{align}
that we choose later.
Then by \eqref{eqn:keyfact3} with $a=p$ followed by \eqref{eqn:keyfact25}, we get
    \begin{align*}
        |B_\Hbf^{2}\cap \Lcal_\zbf|\ll_n \frac{H^{2n}}{p^n} |B_{p/\Hbf}^{2}\cap p\Lcal_\zbf^\ast| \ll_n \frac{H^{2n}}{p^n}(\frac{p/H}{H^{(2)}})^{2n-s(\zbf)} |B_{\Hbf^{(2)}}^{2} \cap p\Lcal_\zbf^\ast|,
    \end{align*}
so that
    \begin{align*}
        \Sigma_2 \ll_n \frac{H^{2n}}{p^n} (\frac{p/H}{H^{(2)}})^{n} \sum_{\zbf\in Z_2}
        |\{(\xbf,\xbf')\in D_{\Hbf}^2: \lambda^1_i(\xbf)=z_i\lambda^2_i(\xbf'),\ 
        \lambda^1_i(\xbf) \lambda^2_i(\xbf')\neq 0 \text{ for all } i\}|\cdot|B_{\Hbf^{(2)}}^{2} \cap p\Lcal_\zbf^\ast|.
    \end{align*}
By Proposition \ref{prop:keyfact_dual}, we have
    \begin{align*}
        B_{\Hbf^{(2)}}^{2} \cap p\Lcal_\zbf^\ast = \{(\ybf,\ybf')\in 
        D_{\Hbf^{(2)}}^2:
        \lam_i''(\ybf') = z_i \lam_i'''(\ybf) \text{ for all } i\},
    \end{align*}
where $\lam''_i(\vbf) = L''_{i,1}(\vbf)+L''_{i,2}(\vbf)\omega_i+\cdots + L''_{i,n_i}(\vbf)\omega_i^{n_i-1}$ and $\lam'''_i(\ubf) = L'''_{i,1}(\ubf)+L'''_{i,2}(\ubf)\omega_i+\cdots + L'''_{i,n_i}(\ubf)\omega_i^{n_i-1}$, and the matrices $A'', A'''$ defined by the linear forms $L_{i,j}''$ and $L_{i,j}'''$ respectively are nonsingular mod $p$.
This crucial step allows us to write the bound for $\Sigma_2$ as an expression that is akin to the bound for $\Sigma_1$ in \eqref{eqn:sigma1}.

Hence from this point on, we repeat the arguments from the $s(\zbf)\leq n$ case. Consider the simultaneous equations
    \begin{align}\label{eqn:simult_lam2}
        \lambda^1_i(\xbf)\lambda'''_i(\ybf)=\lambda^2_i(\xbf')\lambda''_i(\ybf'), \qquad 1\leq i \leq s.
    \end{align}
First rewrite the above as 
    \begin{align*}
        \Sigma_2 \ll_n 
        (\frac{H}{H^{(2)}})^n
        |\{(\xbf,\xbf',\ybf,\ybf')\in D_{\Hbf}^2 \times 
        D_{\Hbf^{(2)}}^2:
        \eqref{eqn:simult_lam2}
        \text{ and }  \lambda^1_i(\xbf)\lambda^2_i(\xbf')\neq 0 \text{ for all } i\}|.
    \end{align*}
Separate this count into those $(\xbf,\xbf',\ybf,\ybf')$ for which $\lam_i''(\ybf')\lam_i'''(\ybf)\neq 0$ for all $i$, and those for which $\lam_j''(\ybf')\lam_j'''(\ybf)= 0$ for at least one $j$, that is,
    \begin{align*}
        \Sigma_2\ll_n (\frac{H}{H^{(2)}})^n (\Sigma_{2}' + \Sigma_{2}''),
    \end{align*}
where 
\begin{align*}
 \Sigma_{2}' =|\{(\xbf,\xbf',\ybf,\ybf')\in D_{\Hbf}^2\times 
 D_{\Hbf^{(2)}}^2: 
 \eqref{eqn:simult_lam2}, \ 
 \lambda^1_i(\xbf)\lam^2_i(\ybf)\neq 0 \text{ for all } i, \ 
 \lam_i''(\ybf')\lam_i'''(\ybf)\neq 0 \text{ for all } i
 \}|
\end{align*}
and
    \begin{align*}
        \Sigma_{2}'' = |\{(\xbf,\xbf',\ybf,\ybf')\in D_{\Hbf}^2\times 
        D_{\Hbf^{(2)}}^2:
        \eqref{eqn:simult_lam2}, \ 
          \lambda^1_i(\xbf) \lambda^2_i(\xbf')\neq 0 \text{ for all } i, \ 
    \lam''_{j}(\ybf')\lam'''_{j}(\ybf)=0 \text{ for some } j\}|.
    \end{align*}
Bound $\Sigma_2'$ by $E'(\Hbf,\Hbf^{(2)})$ and $\Sigma_2''$ by $E''(\Hbf,\Hbf^{(2)})$, so 
 \begin{align*}
     \Sigma_2\ll_n  (\frac{H}{H^{(2)}})^n (E'(\Hbf,\Hbf^{(2)})+E''(\Hbf,\Hbf^{(2)})).
 \end{align*}
 Apply Proposition \ref{prop:keyestimates1} with $K=H^{(2)}$ to $E'(\Hbf,\Hbf^{(2)})$ and Proposition \ref{prop:keyestimates} with $K=H^{(2)}$ to $E''(\Hbf,\Hbf^{(2)})$, to get
    \begin{multline}\label{eqn:case_s(z)>n_OLD}
        \Sigma_2\ll_n 
         (\frac{H}{H^{(2)}})^n
         [C(\nbf,\Hbf)^{1/2}C(\nbf,\Hbf^{(2)})^{1/2}
         +  [\sum_{\tilde{\nbf}\prec \nbf} \max\{H^{2(n-|\tilde{\nbf}|)},(H^{(2)})^{2(n-|\tilde{\nbf}|)}\} \\
         \max_{\substack{\tilde{\Hbf}\in \bZ^{|\tilde{\nbf}|}\\ \tilde{H}_i\in [H,2H]}} C(\tilde{\nbf},\tilde{\Hbf})^{1/2}
       \max_{\substack{\tilde{\Hbf}^{(2)}\in \bZ^{|\tilde{\nbf}|}\\ \tilde{H}_i^{(2)}\in [H^{(2)},2H^{(2)}]}}C(\tilde{\nbf},\tilde{\Hbf}^{(2)})^{1/2}]].
    \end{multline}

\subsubsection{Assembling both cases}\label{sec:assemble_cases_s}
Comparing the bounds \eqref{eqn:case_s(z)<n} for $\Sigma_1$ and \eqref{eqn:case_s(z)>n_OLD} for $\Sigma_2$, we see that they agree when $ H^{(2)}=H^{(1)}$.
This is compatible with the condition \eqref{eqn:condition:H_2_OLD} that $H^{(2)}\leq p/H$
when 
    \begin{align*}
        H\leq p^{1/2}.
    \end{align*}
Thus with $H^{(2)}=H^{(1)}$, the bounds on $\Sigma_1$ and $\Sigma_2$ together yield, after rearrangement,
    \begin{multline*}
         E'(\Hbf,\Hbf)
        \ll_n
     H^{2n}(H^{-2n} C(\nbf,\Hbf))^{1/2}\max_{\substack{\Hbf^{(1)}\in \bZ^{|\nbf|}\\ H^{(1)}_i \in [H^{(1)},2H^{(1)}]}}((H^{(1)})^{-2n}C(\nbf,\Hbf^{(1)}))^{1/2}\\
     + H^{2n}\sum_{\tilde{\nbf}\prec \nbf} (\frac{H}{H^{(1)}})^{n-|\tilde{\nbf}|} 
    \max_{\substack{\tilde{\Hbf}\in \bZ^{|\tilde{\nbf}|}\\ \tilde{H}_i\in [H,2H]}}     
(H^{-2|\tilde{\nbf}|}C(\tilde{\nbf},\tilde{\Hbf}))^{1/2}
 \max_{\substack{\tilde{\Hbf}^{(1)}\in \bZ^{|\tilde{\nbf}|}\\ \tilde{H}_i^{(1)}\in [H^{(1)},2H^{(1)}]}}
((H^{(1)})^{-2|\tilde{\nbf}|}C(\tilde{\nbf},\tilde{\Hbf}^{(1)}))^{1/2}. 
    \end{multline*}
Choose $H^{(1)}=\lfloor p^{-\kappa} H\rfloor $ for some $\kappa>0$, so that $(H/H^{(1)})^{n-|\tilde{\nbf}|}=p^{\kappa(n-|\tilde{\nbf}|)}$. 
Combining this with \eqref{eqn:E''(H,H)}, we have
\begin{multline}\label{eqn:assemble_cases}
    E(\Hbf,\Hbf) 
    \ll_n H^{2n}(H^{-2n}C(\nbf,\Hbf))^{1/2}\max_{\substack{\Hbf^{(1)}\in \bZ^{|\nbf|}\\ H^{(1)}_i \in [H^{(1)},2H^{(1)}]}}((H^{(1)})^{-2n}C(\nbf,\Hbf^{(1)}))^{1/2} \\
    +H^{2n} p^{\kappa(n-1)}\sum_{\tilde{\nbf}\prec \nbf} H^{-2|\tilde{\nbf}|}\max_{\tilde{H}\leq H}\max_{\substack{\tilde{\Hbf}\in\bZ^{|\tilde{\nbf}|}\\ \tilde{H}_i \in [\tilde{H},2\tilde{H}]}} C(\tilde{\nbf},\tilde{\Hbf}).
\end{multline}
Finally, given any $\varepsilon>0$, choose $A_1,A_2,A_3,A_4$ such that $C(\nbf,\Hbf)-E(A_1,A_2,A_3,A_4)<\varepsilon$.
This selection exists by definition of the supremum. Since \eqref{eqn:assemble_cases} holds for all choices of matrices $A_1,A_2,A_3,A_4$, we conclude the same bound for $C(\nbf,\Hbf)$.
This completes the proof of Proposition \ref{prop:recursive}.

\begin{cor}\label{cor:recursive} Fix $\kappa>0$, and let $p^\kappa \leq H\leq p^{1/2}$. Let $\Hbf = (H_1,...,H_n)$ such that $H\leq H_i\leq 2H$ for all $i$. Let $C(\nbf,\Hbf)$ be as defined in \eqref{eqn:def_C}.
Then
    \begin{align}\label{eqn:cor_recursive_original}
        C(\nbf,\Hbf) \ll_n C_1^{1/\kappa} p^{n^2 \kappa} H^{2n},
    \end{align}
 where $C_1=C_1(n)$. As a consequence, let $1\leq H\leq p^{1/2}$ and $\Hbf=(H_1,...,H_n)$ such that $H\leq H_i\leq 2H$ for all $i$. Then for any $\varepsilon>0$,
    \begin{align}\label{eqn:cor_recursive_varep}
        C(\nbf,\Hbf)\ll_{n,\varepsilon} H^{2n}p^\varepsilon.
    \end{align}
\end{cor}

    \begin{proof}
We first deduce \eqref{eqn:cor_recursive_varep} from \eqref{eqn:cor_recursive_original}. Fix $\varepsilon>0$ and let $\kappa=\varepsilon/4n^2$. If $H<p^{\varepsilon/4n^2}$, then we bound trivially $C(\nbf,\Hbf)\ll_n H^{4n}\ll p^{\varepsilon}$. If $p^{\varepsilon/4n^2}\leq H \leq p^{1/2}$, then apply \eqref{eqn:cor_recursive_original}. This proves \eqref{eqn:cor_recursive_varep} and hence Theorem \ref{thm:main_energy}.

   Now we show \eqref{eqn:cor_recursive_original}. 
First let us rewrite the recursive formula \eqref{eqn:prop_recursive} in Proposition \ref{prop:recursive} in a more convenient form.
Let $f(\nbf,\Hbf):=H^{-2 |\nbf|}C(\nbf,\Hbf)$, so that \eqref{eqn:prop_recursive} is
    \begin{align}\label{eqn:recursive_rewrite}
        f(\nbf,\Hbf)\leq M_1+M_2,
    \end{align}
    where
    \begin{align*}
         M_1:=C_1f(\nbf,\Hbf)^{1/2} \max_{\substack{\Hbf^{(1)}\in \bZ^{|\nbf|}\\ H^{(1)}_i \in [H^{(1)},2H^{(1)}]}} f(\nbf,\Hbf^{(1)})^{1/2},
         \qquad M_2:=C_2 p^{|\nbf| \kappa} \sum_{\nbf^{(1)} \prec \nbf} \max_{\substack{\tilde{H}\leq H}}\max_{\substack{\tilde{\Hbf}\in \bZ^{|\nbf^{(1)}|}\\ \tilde{H}_i\in [\tilde{H}, 2\tilde{H}]}} f(\nbf^{(1)},\tilde{\Hbf}).
    \end{align*}
Separate the case $\nbf=\mathbf{0}$ from the sum in $M_2$, so that by the convention $C(\mathbf{0},\Hbf)=1$, we have
    \begin{align}\label{eqn:iteration1b}
        M_2= C_2 p^{|\nbf| \kappa} \sum_{\substack{\nbf^{(1)} \prec \nbf\\ \nbf^{(1)}\neq \mathbf{0}}} \max_{\substack{\tilde{H}\leq H}}\max_{\substack{\tilde{\Hbf}\in \bZ^{|\nbf^{(1)}|}\\ \tilde{H}_i\in [\tilde{H}, 2\tilde{H}]}} f(\nbf^{(1)},\tilde{\Hbf})
        + C_2p^{|\nbf| \kappa}.
    \end{align}
Further rewrite \eqref{eqn:recursive_rewrite} as
       \begin{align*}
           f(\nbf,\Hbf)
           \ll \max \{M_1, M_2\},
       \end{align*}
    so that either $ f(\nbf,\Hbf) \ll M_1$ or $ f(\nbf,\Hbf)\ll M_2$. Rearranging the first case $ f(\nbf,\Hbf) \ll M_1$, we obtain $ f(\nbf,\Hbf) \ll M_1'$ where
        \begin{align}\label{eqn:def_T1'}
            M_1':=C_1^2 \max_{\substack{\Hbf^{(1)}\in \bZ^{|\nbf|}\\ H^{(1)}_i \in [H^{(1)},2H^{(1)}]}}  f(\nbf,\Hbf^{(1)}),
        \end{align}
        recalling that $H^{(1)} = \lfloor p^{-\kappa}H \rfloor$.
Hence \eqref{eqn:prop_recursive} is equivalent to
        \begin{align}\label{eqn:main3}
            f(\nbf,\Hbf)
           \ll \max \{M_1', M_2\},
        \end{align}
        which prepares us for iteration in the next step.
       
We claim that for all positive integers $t$,
    \begin{multline}\label{eqn:recursive_cor_1}
        f(\nbf,\Hbf) \ll_n \max_{0\leq t_1 \leq t} \{C_1^{2t_1}C_2^{t-t_1}p^{(t-t_1)|\nbf|\kappa}[\sum_{\substack{\nbf^{(1)}\prec \nbf\\ \nbf^{(1)}\neq \mathbf{0}}}\cdots \sum_{\substack{\nbf^{(t-t_1)}\prec \nbf^{(t-t_1-1)}\\ \nbf^{(t-t_1)}\neq \mathbf{0}}}\\
        \max_{\tilde{H}\leq p^{-t_1\kappa}H} \max_{\substack{\tilde{\Hbf}\in \bZ^{|\nbf^{(t-t_1)}|}\\ \tilde{H}_i\in [\tilde{H},2\tilde{H}]}}f(\nbf^{(t-t_1)},\tilde{\Hbf})]
        + C_1^{2t_1}C_2^{t-t_1}p^{(t-t_1)|\nbf|\kappa}\}.
    \end{multline}
We show this by induction on $t$.
The base case $t=1$ is precisely \eqref{eqn:main3}. 
Now suppose this holds for $t$, and apply \eqref{eqn:main3} to $f(\nbf^{(t-t_1)},\tilde{\Hbf}) $.
Then $f(\nbf^{(t-t_1)},\tilde{\Hbf}) \ll \max\{M_1',M_2\}$, where by \eqref{eqn:def_T1'} and \eqref{eqn:iteration1b},
            \begin{align*}
                M_1'&= C_1^2 \max_{\substack{\tilde{\Hbf}^{(1)}\in \bZ^{|\nbf^{(t-t_1)}|}\\ \tilde{H}^{(1)}_i\in [\tilde{H}^{(1)},2\tilde{H}^{(1)}]}}f(\nbf^{(t-t_1)},\tilde{\Hbf}^{(1)}) \quad \text{and}\\
                M_2&=   C_2p^{|\nbf^{(t-t_1)}| \kappa } \sum_{\substack{\nbf^{(t+1-t_1)}\prec \nbf^{(t-t_1)}\\ \nbf^{(t+1-t_1)}\neq \mathbf{0}}}\max_{\hat{H}\leq \tilde{H}}\max_{\substack{\hat{\Hbf}\in \bZ^{|\nbf^{(t+1-t_1)}|}\\ \hat{H}_i\in [\hat{H},2\hat{H}]}} f(\nbf^{(t+1-t_1)},\hat{\Hbf})
            +C_2 p^{|\nbf^{(t-t_1)}| \kappa }.
            \end{align*}
       If $f(\nbf^{(t-t_1)},\tilde{\Hbf})\ll M_1'$, then
        \begin{multline*}
            f(\nbf,\Hbf) \ll 
            \max_{1\leq t_1 \leq t} \{C_1^{2(t_1+1)}C_2^{t-t_1}p^{(t-t_1)|\nbf|\kappa}[\sum_{\substack{\nbf'\prec \nbf\\ \nbf'\neq \mathbf{0}}}
            \cdots \sum_{\substack{\nbf^{(t-t_1)}\prec \nbf^{(t-t_1-1)}\\ \nbf^{(t-t_1)}\neq \mathbf{0}}}\\
            \max_{\tilde{H}^{(1)}\leq p^{-(t_1+1)\kappa}H} \max_{\substack{\tilde{\Hbf}^{(1)}\in \bZ^{|\nbf^{(t-t_1)}|}\\ \tilde{H}^{(1)}_i\in [\tilde{H}^{(1)},2\tilde{H}^{(1)}]}}f(\nbf^{(t-t_1)},\tilde{\Hbf}^{(1)})]
        + C_1^{2t_1}C_2^{t-t_1}p^{(t-t_1)|\nbf|\kappa}\}.
        \end{multline*}
    By a change of variables $s_1= t_1+1$, we can write the right-hand side as
 \begin{multline}\label{eqn:recursive_cor_2}
            f(\nbf,\Hbf) \ll 
            \max_{2\leq s_1 \leq t+1} \{C_1^{2s_1}C_2^{t+1-s_1}p^{(t+1-s_1)|\nbf|\kappa}
            [\sum_{\substack{\nbf'\prec \nbf\\ \nbf'\neq \mathbf{0}}} \cdots \sum_{\substack{\nbf^{(t+1-s_1)}\prec \nbf^{(t+1-s_1-1)}\\ \nbf^{(t+1-s_1)}\neq \mathbf{0}}}\\
             \max_{\tilde{H}\leq p^{-s_1\kappa}H} \max_{\substack{\tilde{\Hbf}\in \bZ^{|\nbf^{(t+1-s_1)}|}\\ \tilde{H}_i\in [\tilde{H},2\tilde{H}]}}f(\nbf^{(t+1-s_1)},\tilde{\Hbf}) ]
        + C_1^{2(s_1-1)}C_2^{t+1-s_1}p^{(t+1-s_1)|\nbf|\kappa}\}.
        \end{multline}        
On the other hand, if $f(\nbf^{(t-t_1)},\tilde{\Hbf})\ll M_2$, then
     \begin{multline*}
        f(\nbf,\Hbf) \ll_n \max_{1\leq t_1 \leq t} \{C_1^{2t_1}C_2^{t+1-t_1}p^{(t+1-t_1)|\nbf|\kappa}[\sum_{\substack{\nbf^{(1)}\prec \nbf\\ \nbf^{(1)}\neq \mathbf{0}}} \cdots \sum_{\substack{\nbf^{(t-t_1)}\prec \nbf^{(t-t_1-1)}\\ \nbf^{(t-t_1)}\neq \mathbf{0}}}
        \max_{\tilde{H}\leq p^{-t_1\kappa}H}\\
        \sum_{\substack{\nbf^{(t+1-t_1)}\prec\nbf^{(t-t_1)}\\ \nbf^{(t+1-t_1)}\neq \mathbf{0}}}
        \max_{\substack{\tilde{\Hbf}\in \bZ^{|\nbf^{(t+1-t_1)}|}\\ \tilde{H}_i\in [\tilde{H},2\tilde{H}]}} f(\nbf^{(t+1-t_1)},\tilde{\Hbf})]
                + C_1^{2t_1}C_2^{t+1-t_1}p^{(t+1-t_1)|\nbf|\kappa}\}.
    \end{multline*}
Combining this with \eqref{eqn:recursive_cor_2} gives \eqref{eqn:recursive_cor_1} for $t+1$. Thus  \eqref{eqn:recursive_cor_1} holds for all $t$.

Finally we use \eqref{eqn:recursive_cor_1} to deduce the inequality claimed in \eqref{eqn:cor_recursive_original}. 
Note that as $t$ increases by one, on the right-hand side of \eqref{eqn:recursive_cor_1}, either $H$ decreases by a factor of $p^{\kappa}$ or the dimension $|\nbf|$ decreases since $\nbf'\prec \nbf$. This process ends when either $H\approx p^\kappa$ (at most $T_1:=\kappa^{-1}\log H/\log p$ steps) or when $\nbf'=\mathbf{0}$ (at most $s\leq n$ steps). 
We claim that if $t\geq T_1+n$, then either $t_1\geq T_1$ or $t-t_1\geq n$. Indeed, if $t_1<T_1$, then $t-t_1\geq t-T_1\geq n$. Hence as long as $t\geq T_1+n$, 
the first term on the right-hand side of \eqref{eqn:recursive_cor_1} vanishes, for all choices of $t_1$, since the sum is empty and so
    \begin{align*}
        f(\nbf,\Hbf) \ll_n 
        C_1^{2T_1}C_2^{n}p^{n|\nbf|\kappa}
        \ll_n C_1^{1/\kappa}p^{n^2\kappa},
    \end{align*}
where we used $H\leq p^{1/2}$,
and this proves the corollary.
    \end{proof}



\section{Energy estimates: proof of Theorem \ref{cor:main_energy}}\label{sec:cor_energy}

We separate the cases $n\leq k$ and $k\leq n$ since their proofs are different. First consider the case $n\leq k$. Let $A$ denote the $k\times n$ matrix whose rows are defined by the coefficients of the $k$ linear forms $L_{i,j}$ in \eqref{def:energy_general}.
Since $A$ has full rank $n$ by supposition, we can enlarge $A$ to a nonsingular $k\times k$ matrix $\tilde{A}$ by, for example, extending the $n$ linearly independent columns vectors of $A$ to a basis in $(\bF_p)^k$ and then appending the $k-n$ new column vectors to $A$. 
Let $\tilde{L}_{i,j}$ with $1\leq i \leq s$ and $1\leq j \leq k_i$ be the linear forms whose coefficients are defined by the rows of $\tilde{A}$. Define $  \tilde{\lam}_i(\cdot) = \Ltil_{i,1}(\cdot)+\Ltil_{i,2}(\cdot)\omega_i + \cdots + \Ltil_{i,k_i}(\cdot)\omega_i^{k_i-1}$.
Let $\tilde{\Nbf} = (\Nbf,N_{n+1},...,N_k)$, and similarly define $\tilde{\Hbf},\tilde{\xbf},\tilde{\xbf}',\tilde{\ybf},\tilde{\ybf}'$. In particular, $\Ltil_{i,j}((\xbf,\mathbf{0})) = L_{i,j}(\xbf)$. 
Let $S_{n,k}((\Nbf,\Nbf+\Hbf])$ denote the set whose cardinality is the energy in \eqref{def:energy_general}.
Define analogously
    \begin{align*}
        S_{k,k}((\tilde{\Nbf},\tilde{\Nbf}+\tilde{\Hbf}]) =\{(\tilde{\xbf},\tilde{\xbf}',\tilde{\ybf},\tilde{\ybf}')\in (\tilde{\Nbf},\tilde{\Nbf}+\tilde{\Hbf}]^4:   \tilde{\lam}_i(\tilde{\xbf})  \tilde{\lam}_i(\tilde{\ybf}')=  \tilde{\lam}_i(\tilde{\xbf}')  \tilde{\lam}_i(\tilde{\ybf}) \text{ for } 1\leq i\leq s\}.
    \end{align*}
By Theorem \ref{thm:main_energy},
$ |S_{k,k}((\tilde{\Nbf},\tilde{\Nbf}+\tilde{\Hbf}])|\ll_{k,\varepsilon} H^{2k}p^\varepsilon$
for all $\varepsilon>0$.
Consider the map
    \begin{align*}
    f:S_{n,k}((\Nbf,\Nbf+\Hbf]) \rightarrow S_{k,k}((\tilde{\Nbf},\tilde{\Nbf}+\tilde{\Hbf}])
    \end{align*}
defined by $f(\xbf,\xbf',\ybf,\ybf') = ((\xbf,\mathbf{0}),(\xbf',\mathbf{0}),(\ybf,\mathbf{0}),(\ybf',\mathbf{0}))$. This is clearly injective, so we conclude in this case of $n\leq k$ that $E_{n,k}((\Nbf,\Nbf+\Hbf])\ll_{k,\varepsilon} H^{2k}p^{\varepsilon}$.

Next consider the case $k\leq n$. Apply Lemma \ref{lem:shifting} to shift the domain of $(\xbf,\xbf',\ybf,\ybf')$ to $D_{\Hbf}^4$, and then apply Lemma \ref{lem:energy_k<n} to \eqref{def:energy_general} with $K=H, \sigma=s$, and $\mu_\ell = n$ for all $\ell$ to get the bound
    \begin{align*}
         E_{n,k}((\Nbf,\Nbf+\Hbf])\ll_{n,k} 
    H^{4n-4k}
   \max_{\substack{\hat{\Hbf}\in \bZ^{k}\\ \hat{H}_i\in [H,2H]}} C(\kbf,\hat{\Hbf}).
    \end{align*}
Finally, apply Corollary \ref{cor:recursive} to obtain $ E_{n,k}((\Nbf,\Nbf+\Hbf])\ll_{n,k,\varepsilon} H^{4n-2k}p^{\varepsilon}$. Combining the two cases proves Theorem \ref{cor:main_energy}.

\subsection{The sharpness of Theorem \ref{cor:main_energy}}\label{ref:sharp}
Fix $\Nbf,\Hbf$ and, for ease of notation, let $E_{n,k} = E_{n,k}((\Nbf,\Nbf+\Hbf])$.
For general $n,k$, we have the lower bound of $E_{n,k}\gg_n H^{2n}$, given by the diagonal solutions $(\xbf,\ybf)=(\xbf',\ybf')$. Heuristically, $k$ is the number of constraints, so when $n<k$, we should expect $E_{n,k}$ to be at most $E_{n,n}\ll_{n,\varepsilon} H^{2n}p^\varepsilon$. Let us look at a concrete example. Consider the case where $s=k$ and $k_i=1$ for all $1\leq i \leq s$. Then $\lambda_i(\cdot) = L_i(\cdot)$ is a linear form over $\bF_p$ and  all $s$ equations lie in $\bF_p$. 
By supposition, the $k\times n$ matrix defined by the coefficients of $L_i$ has full rank $n$; without loss of generality, suppose this is achieved by the coefficients of $L_1,...,L_n$. Then 
    \begin{align*}
      E_{n,k} &= |\{(\xbf,\xbf',\ybf,\ybf')\in (\Nbf,\Nbf+\Hbf]^4: L_i(\xbf)L_i(\ybf')=L_i(\xbf')L_i(\ybf) \in \bF_{p}, \ 1\leq i\leq s\}|\\
      &\leq  |\{(\xbf,\xbf',\ybf,\ybf')\in (\Nbf,\Nbf+\Hbf]^4: L_i(\xbf)L_i(\ybf')=L_i(\xbf')L_i(\ybf) \in \bF_{p}, \ 1\leq i\leq n\}| \\
      &\ll_{n,\varepsilon} H^{2n}p^\varepsilon,
    \end{align*}
which is sharp (up to the $p^\varepsilon$ factor), for this special case. For the case $k<n$, on the other hand, we expect the upper bound $H^{2(2n-k)}p^\varepsilon$ to be sharp (up to $p^\varepsilon$). Again consider the above setting where $s=k$ and $k_i=1$ for all $1\leq i \leq s$. Suppose further that $L_i(\Xbf) = X_i$, so that
    \begin{align*}
         E_{n,k} &= |\{(\xbf,\xbf',\ybf,\ybf')\in (\Nbf,\Nbf+\Hbf]^4: x_iy_i'= x_i'y_i \in \bF_{p}, \ 1\leq i\leq k\}|.
    \end{align*}
The last $n-k$ coordinates of each of $\xbf,\xbf',\ybf,\ybf'$ are unconstrained, so that by the lower bound in Lemma \ref{lem:energy_elementary} applied to the first $k$ coordinates of each of $\xbf,\xbf',\ybf,\ybf'$, we have $E_{n,k} \gg_{n,k} H^{4(n-k)}H^{2k} = H^{2(2n-k)}$, which agrees (up to $p^\varepsilon$) with the upper bound in Theorem \ref{cor:main_energy}.
    
\section*{Acknowledgements}
The author would like to thank her advisor Lillian B. Pierce 
for providing ongoing support and encouragement for the entirety of this project and beyond. The author is partially supported by NSF DMS-2200470 and the Katherine Goodman Stern Fellowship from The Graduate School at Duke University. The number theory group at Duke is supported by the RTG DMS-2231514. The author thanks Lillian B. Pierce and the Hausdorff Center for Mathematics for a productive research visit in July 2024. The author is grateful to Sarah Peluse for many discussions on multiplicative energy and to Chongyao Chen and Rohit Kumar for helpful conversations. Finally, the author thanks the anonymous referees for numerous insightful comments and suggestions.

\bibliographystyle{alpha}
\bibliography{_bibliography}
\end{document}